\documentclass[a4paper]{amsart}
\pdfoutput=1
\usepackage[ansinew]{inputenc}
\usepackage[danish,english]{babel}
\usepackage{amssymb}
\usepackage{amsmath}
\usepackage{amsthm}
\usepackage{eucal}
\usepackage{verbatim}
\usepackage{subfigure}
\usepackage[pdftex]{color, graphicx}
\usepackage{fancyhdr}
\usepackage{tikz}
\usepackage{amsfonts}
\usepackage{mathrsfs}
\usepackage{graphicx}
\usepackage[final]{pdfpages}
\usepackage{comment}
\usepackage{enumerate}

\usepackage[pagebackref,colorlinks,citecolor=blue,linkcolor=blue,urlcolor=blue,filecolor=blue]{hyperref}

\numberwithin{equation}{section}
\numberwithin{table}{section}
\theoremstyle{plain}
\newtheorem{thm}{Theorem}[section]
\newtheorem{thma}{Theorem}

\newtheorem{prop}[thm]{Proposition}

\newtheorem{lem}[thm]{Lemma}
\newtheorem{cor}[thm]{Corollary}

\theoremstyle{definition}
\newtheorem{dfn}[thm]{Definition}

\newtheorem{cons}[thm]{Construction}
\theoremstyle{remark}

\newtheorem{rmk}[thm]{Remark}
\newtheorem{notation}[thm]{Notation}

\DeclareMathOperator{\Ob}{Ob}
\DeclareMathOperator{\Hom}{Hom}

\usetikzlibrary{arrows,shapes,decorations.pathmorphing,patterns,backgrounds,fit,positioning,matrix}
\tikzset{dot/.style={circle,fill=black,thick,inner sep=0pt,minimum size=1mm,draw}}
\tikzset{arrow/.style={semithick,>=stealth',shorten >=1pt,shorten <=1pt}}
\tikzset{equal/.style={kant,double distance=2pt}}

\DeclareMathAlphabet{\mathpzc}{OT1}{pzc}{m}{it}

\newcommand{\N}{\mathbb N}
\newcommand{\Z}{\mathbb Z}
\newcommand{\R}{\mathbb R}

\newcommand{\Cat}{\mathscr{C}}
\newcommand{\tCat}{\widetilde{\mathscr{C}}}
\newcommand{\Dcat}{\mathscr{D}}
\newcommand{\Ecat}{\mathscr{E}}

\newcommand{\Acat}{\mathscr{A}}
\newcommand{\Bcat}{\mathscr{B}}
\newcommand{\Tcat}{\mathscr{T}}

\newcommand{\Lcat}{\mathscr{L}}
\newcommand{\tLcat}{\widetilde{\mathscr{L}}}
\newcommand{\Pcat}{\mathscr{P}}

\newcommand{\cof}{\hookrightarrow}
\newcommand{\fib}{\twoheadrightarrow}

\newcommand{\adj}{\leftrightarrows}

\newcommand{\Fat}{\mathpzc{Fat}}
\newcommand{\Fatoc}{\mathpzc{Fat}^{\mathpzc{oc}}}
\newcommand{\Fatad}{\mathpzc{Fat}^{\mathpzc{ad}}}
\newcommand{\Eoc}{\mathpzc{EFat}^{\mathpzc{oc}}}
\newcommand{\Ead}{\mathpzc{EFat}^{\mathpzc{ad}}}
\newcommand{\Fatocg}{\mathpzc{Fat}^{\mathpzc{oc}}_{S}}
\newcommand{\Fatadg}{\mathpzc{Fat}^{\mathpzc{ad}}_{S}}
\newcommand{\Eocg}{\mathpzc{EFat}^{\mathpzc{oc}}_{S}}
\newcommand{\Eadg}{\mathpzc{EFat}^{\mathpzc{ad}}_{S}}

\newcommand{\Modgpq}{\mathrm{Mod}(S)}
\newcommand{\Mod}{\mathrm{Mod}}
\newcommand{\Diff}{\mathrm{Diff}}
\newcommand{\Homeo}{\mathrm{Homeo}}
\newcommand{\Sg}{S}

\newcommand{\bwgraphs}{\mathscr{BW}-Graphs}

\newcommand{\md}{\mathrm{deg^m}}
\newcommand{\bwd}{\mathrm{deg^{bw}}}

\newcommand{\Fatng}{\mathpzc{Fat}^{n}_{S}}
\newcommand{\FatNg}{\mathpzc{Fat}^{N}_{S}}
\newcommand{\bL}{\eth\mathscr{L}}
\newcommand{\btL}{\eth\widetilde{\mathscr{L}}}
\newcommand{\Cquasi}{\mathrm{C}^{\mathrm{quasi}}}
\newcommand{\dquasi}{\mathrm{d}^{\mathrm{quasi}}}
\newcommand{\dep}{\mathrm{depth}}
\newcommand{\ie}{\scalebox{1.5}{$\mathpzc{e}$}}

\newcommand{\Arc}{\text{Arc}}

\newcommand{\cM}{\mathcal{M}}
\newcommand{\cT}{\mathcal{T}}

\title{Comparing fat graph models of Moduli space}
\author{Daniela Egas Santander}

\begin{document}

\begin{abstract}
Godin introduced the categories of open closed fat graphs $\Fatoc$ and admissible fat graphs $\Fatad$ as models of the mapping class group of open closed cobordism. We use the contractibility of the arc complex to give a new proof of Godin's result that $\Fatad$ is a model of the mapping class group of open-closed cobordisms.  
Similarly, Costello introduced a chain complex of black and white graphs $\bwgraphs$, as  a rational homological model of mapping class groups.  We use the result on admissible fat graphs to give a new integral proof of Costellos's result that $\bwgraphs$ is a homological model of mapping class groups.  The nature of this proof also provides a direct connection between both models which were previously only known to be abstractly equivalent. 
Furthermore, we endow Godin's model with a composition structure which models composition of cobordisms along their boundary and we use the connection between both models to give $\bwgraphs$ a composition structure and show that $\bwgraphs$ are actually a model for the open-closed cobordism category.
\end{abstract}

\maketitle
\setcounter{tocdepth}{2}
\tableofcontents

\section{Introduction} 
The moduli space of Riemann surfaces is a classical object in mathematics as it is related to the classification of Riemann surfaces and families of such. Moreover, when considering Riemann surfaces with boundary, this moduli space is also related to mathematical physics, as it plays a central role in the study of two dimensional field theories.
There are several constructions of moduli space coming from different areas of mathematics: algebraic geometry, hyperbolic geometry, conformal geometry among others, with a rich interplay among them.  However, this space is not yet fully understood.  

There are many different combinatorial models of moduli space.  Such models can give us further insight on the homotopy type of moduli space via direct calculations.  For example, there are explicit computations of the unstable homology of moduli space using combinatorial models \cite{abhau, Godinunstable}.
On the other hand, combinatorial models of moduli space (and compactifications of it) also play a central role in some constructions of two dimensional field theories and in particular in the construction of string operations 
\cite{costellotcft, kaufmann_penner, godin, kaufmann_I, kaufman_II, poirier, kaufman_oc, wahlwesterland, poirierrounds, wahluniversal, drummond_poirier_rounds}.

Although all these combinatorial models are abstractly equivalent, since they all model the homotopy type of moduli space, a direct connection between them is in general not understood. In this paper, we give a direct connection between two such models: the admissible fat graph model due to Godin and the black and white graph model due to Costello.  Furthermore, we endow both with a notion of composition or gluing which models composition in moduli space, which is given by sewing surfaces along their boundaries.  As far as we know, these are the first models in terms of fat graphs to include a notion of composition along closed boundary components.

In this introduction we start by first recalling the notion of the moduli space of Riemann surfaces and open-closed two dimensional cobordisms.  Then we briefly describe fat graphs as in the work of Penner, Harer, Igusa and Godin.  Afterwards, we describe Costello's model and its relation to Godin's model.  Finally, we describe how one can endow both models with a notion of composition which corresponds to composition in moduli space.

\subsection{Cobordisms and their moduli}

The study of surfaces and their structure has been a central theme in many areas of mathematics. One approach to study the genus $g$ closed oriented surface $\Sigma_g$, is by the \emph{moduli space of $\Sigma_g$} which we denote $\cM_g$, which is a space that classifies all compact Riemann surfaces of genus $g$ up to complex-analytic isomorphism. We recall some the concepts involved in this field mainly following \cite{primer_mcg,hamenst_theich}. A \emph{marked metric complex structure} on $\Sigma_g$, is a tuple $(X,\varphi)$, where $X$ is a Riemann surface and $\varphi:\Sigma_g\to X$ is an orientation preserving diffeomorphism.  Two complex structures $(X,\varphi)$ and $(X',\varphi')$ are \emph{equivalent} if there is a biholomorphic map $f:X\to X'$ such that $f\circ \varphi$ and $\varphi'$ are isotopic.  As a set, the \emph{Teichm\"{u}ller space of $\Sigma_g$} which we denote $\cT_g$, is the set of all equivalence classes of marked metric complex structures.  It can be given a topology with which it is a contractible manifold of dimension $6g-6$. The \emph{mapping class group of $\Sigma_g$}, which we denote $\Mod(\Sigma_g)$, is the group of components of the topological group of orientation preserving self-diffeomorphisms of the surface i.e., $\pi_0(\Diff^+(\Sigma_g))$.  One can show that this definition is equivalent to many others namely 
\[\Mod(\Sigma_g)\cong\pi_0(\Homeo^+(\Sigma_g))\cong\Diff^+(\Sigma_g)\diagup_{\sim_i}\cong\Homeo^+(\Sigma_g)\diagup_{\sim_h}\]
where $\sim_i$ and $\sim_h$ denote the isotopy and homotopy relations respectively.  The mapping class group acts on Teichm\"{u}ller space by precomposition with the marking and the moduli space of $\Sigma_g$ is the quotient of Teichm\"{u}ller space by this action i.e., $\cM_g:=\cT_g/\Mod(\Sigma_g)$. 

These definitions can be extended to surfaces with additional structure.  We will study the case of open-closed cobordisms, which has applications in string topology and topological field theories. An \emph{open-closed cobordism} $S$ is an oriented surface with boundary together with a partition of the boundary into three parts $\partial_{in}S$, $\partial_{out}S$ and $\partial_{free}S$ and parametrizing diffeormorphisms 
\[ \begin{array}{rcl}
\partial_{in} S\to N_{in} &\phantom{bigbigspace}& \partial_{out} S\to N_{out}
\end{array} \]
where $N_{in}$ is a space with $p=p_1+p_2$ ordered connected components, $p_1$ of these components are standard circles (i.e., $S^1$'s) which parametrize the incoming closed boundaries and $p_2$ of them are unit intervals which parametrize the incoming open boundaries. Similarly, $N_{out}$ is a space with $q=q_1+q_2$ ordered connected components, $q_1$ of these components are standard circles which parametrize the outgoing closed boundaries and $q_2$ of them are unit intervals which parametrize the outgoing open boundaries.  The parametrizing diffeormorphisms additionally give an ordering of the incoming and outgoing boundary components, see Figure \ref{cobordism}.  
Each incoming boundary component $\partial_i S$ comes equipped with a collar i.e., a map from $(-\epsilon,0]\times S^1$ (if $\partial_i S$ is closed) or from $(-\epsilon,0]\times [0,1]$ (if $\partial_i S$ is open) to a neighborhood of $\partial_i S$ which restricts to the boundary parametrization.  Similarly, each outgoing boundary component $\partial_i S$ comes equipped with a collar i.e., a map from $[0,\epsilon)\times S^1$ (if $\partial_i S$ is closed) or from $[0,\epsilon)\times [0,1]$ (if $\partial_i S$ is open) to a neighborhood of $\partial_i S$ which restricts to the boundary parametrization.  Note that since the surface $\Sg$ is oriented, then up to homeomorphism, to give the parametrizing diffeomorphisms is equivalent to fixing a marked point in each component of $\partial_{in} S\cup\partial_{out} S$ and giving an ordering of these. As in the case of surfaces, $2$-dimensional open-closed cobordisms $\Sg$ and $\Sg'$ have the same \emph{topological type as open-closed cobordisms} if there is an orientation preserving homeomorphism $h:\Sg\to\Sg'$ that respects the collars.
 
\begin{figure}[h!]
  \centering
    \includegraphics[scale=0.3]{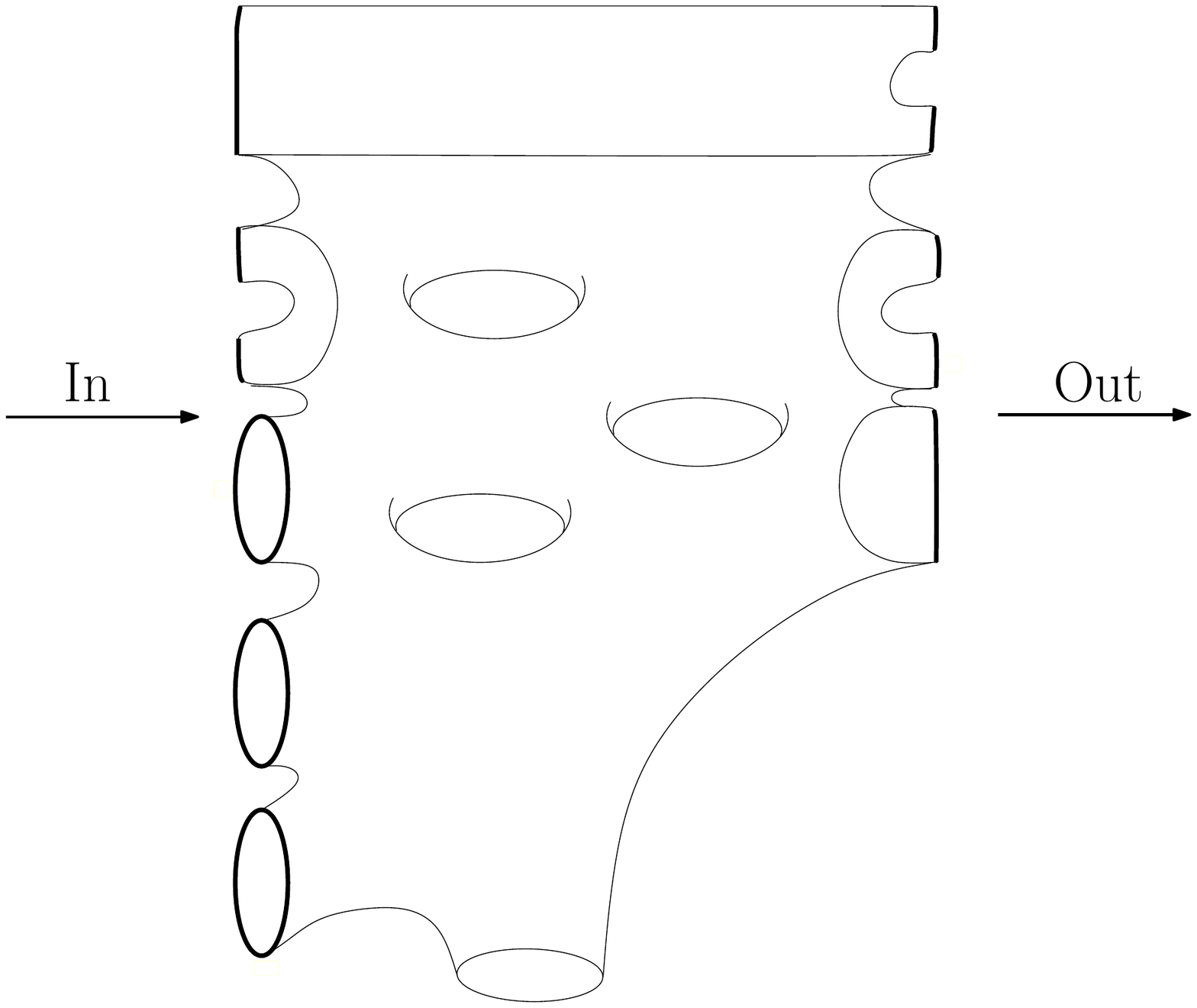}
  \caption{An open-closed cobordism whose underlying surface has genus $3$ and $8$ boundary components.  There are $3$ incoming closed boundaries and no outgoing closed boundaries.  There are $3$ incoming open boundaries and $5$ outgoing open boundaries.}
  \label{cobordism}
\end{figure}

The notions of Teichm\"{u}ller space, moduli space and mapping class groups are extended in a natural way.  More precisely, a \emph{marked metric complex structure} on $\Sg$ is a tuple $(\varphi, X)$, where $X$ is a Riemann surface with boundary parametrizations and collars and $\varphi:S\to X$ is an orientation preserving diffeomorphism that respects the collars.  Two complex structures $(X,\varphi)$ and $(X',\varphi')$ are \emph{equivalent} if there is a biholomorphic map $f:X\to X'$ that respects the collars such that $f\circ \varphi$ and $\varphi'$ are isotopic.  The \emph{Teichm\"{u}ller space of $\Sg$} which we denote $\cT_{S}$, is the space of all equivalence classes of marked metric complex structures. The \emph{mapping class group of} $\Sg$  is 
\[\Mod(\Sg):=\pi_0(\Diff^+(\Sg,\partial_{in} S\cup\partial_{out} S))\]
where $\Diff^+(\Sg,\partial_{in} S\cup\partial_{out} S)$ is the space of orientation preserving diffeomorphisms that fix the collars.  The mapping class group acts on Teichm\"{u}ller space by precomposition with the marking and $\cM_{S}:=\cT_{S}/\Mod(\Sg)$.  When there is at least one marked point in a the boundary of $\Sg$, the action of $\Mod(\Sg)$ is free and thus $\cM_{S}$ is a classifying space of $\Mod(\Sg)$ .

\subsection{Admissible fat graphs}
Informally, a \emph{fat graph} or (\emph{ribbon graph}) is a graph in which each vertex has a cyclic ordering of the edges that are attached to it, see Definition \ref{fat_definition} for a precise definition.  This cyclic ordering allows us to fatten the graph to obtain a surface.  In \cite{penner}, Penner constructs a triangulation of the decorated Teichm\"{u}ller space of surfaces with punctures, which is equivariant under the action of the mapping class group, giving a model of the decorated moduli space of punctured surfaces.  In \cite{igusahrt}, Igusa constructs a category $\mathpzc{Fat}$, with objects fat graphs whose vertices have valence greater or equal to three. He shows that this category rationally models the mapping class groups of punctured surfaces.  Following these ideas, in \cite{Godinunstable}, Godin constructs a category $\mathpzc{Fat}^{\mathpzc{b}}$ of fat graphs with leaves and shows that this category models the mapping class groups of bordered surfaces.  
In \cite{godin}, she extends this construction and defines a category $\Fatoc$, of \emph{open-closed fat graphs} which are fat graphs with labeled leaves, see Definition \ref{oc_def} for a precise definition, and shows that this category models the mapping class groups of open-closed cobordisms.  Moreover, she defines a full subcategory $\Fatad$, of \emph{admissible fat graphs}, which are a special kind of open-closed fat graphs with disjoint embedded circles corresponding to the outgoing closed boundary components, see Definition \ref{ad_def} for a precise definition, and shows that this sub-category also models the mapping class groups of open-closed cobordisms. 
However, there is a step missing in the proof of this last result which we do not know how to complete.  More precisely, Godin proves this by comparing certain fiber sequences, but a map connecting them is not explicitly constructed and we do not know how to construct such map.  
In this paper, we give a new proof of Godin's result, shown in Theorem \ref{thmA}, which is more geometric in nature, by using the contractibility of the arc complex.

\begin{thma}
\label{thmA}
The categories of open-closed fat graphs and admissible fat graphs are models for the classifying spaces of mapping class groups of open-closed cobordisms.  More specifically there is a homotopy equivalence
\[|\Fatoc|\simeq \coprod_{\Sg} \mathrm{B}\Modgpq\]
where the disjoint union runs over all topological types of open-closed cobordisms in which each connected component has at least one boundary component which is not free.  Moreover, this restricts on the subcategory of admissible fat graphs to a homotopy equivalence 
\[|\Fatad|\simeq \coprod_{\Sg} \mathrm{B}\Modgpq\]
where the disjoint union runs over all topological types of open-closed cobordisms in which each connected component has at least one boundary component which is neither free nor outgoing closed. 
\end{thma}
We show this (in both cases) on each connected component by constructing principal $\Modgpq$-bundles $|\Eoc_S|\fib |\Fatoc_S|$ and $|\Ead_S|\fib |\Fatad_S|$ in which all spaces are finite CW-complexes and $|\Eoc_S|$ and $|\Ead_S|$ are contractible.  

The restriction to the subcategory of admissible fat graphs gives a smaller model of mapping class groups, which might permit further computations of the homology of the mapping class group.  Furthermore, the restriction to the admissible case allows us to define a composition map in terms of graphs which models composition of cobordisms as it will be explained in the last subsection of the introduction.  The admissibility condition is essential to our composition construction as it can not be extended to all open-closed fat graphs.

\subsection{Black and white graphs}

In \cite{costellotcft}, Costello shows that there is an action of the chains of the moduli space of Riemann surfaces on the Hochschild chains of any $A_{\infty}$- Frobenius algebra. The proof of this result uses a model of the chains of moduli space described in \cite{costellorg, costellotcft}.  
To build this model, Costello uses a modular space of surfaces with possibly nodal boundary and shows that the boundary of this partial compactification of moduli space is rationally equivalent to the moduli space of Riemann surfaces.  The boundary of this partial compactification has a natural CW-structure and the generators of its cellular complex are given by disks and annuli glued at the boundary.  

In their study of operations on the Hochschild complex of $A_{\infty}$- algebras with extra structure \cite{wahlwesterland}, Wahl and Westerland use a dual representation of the disks and annuli and describe this chain complex as a complex of fat graphs with two types of vertices: black vertices corresponding to the center of the disks and white vertices corresponding to the inner boundary of the annuli.  See Definitions \ref{bw_generalized_def} and \ref{bw_def} for a concrete definition of a \emph{black and white graph}.  In the Frobenius case, Wahl and Westerland recover Costello's theorem.  Moreover, they give an explicit recipe for this action, which recovers the action given by Kontsevich and Soibelman in \cite{kontsevich_soibelman} for finite dimensional $A_{\infty}$-algebras.  In genus $0$ this action recovers the  $A_{\infty}$-cyclic Deligne conjecture as described in \cite{kaufmann_schwell}.

Following the terminology of \cite{wahlwesterland}, we denote Costello's model of moduli space the \emph{chain complex of black and white graphs}, see Definition \ref{bw_cpx_def} for a concrete definition of this complex. Costello gives a geometric proof of the following theorem, giving a flow of this partial compactification of moduli space onto its boundary. 

\begin{thma}
\label{thmB}
The chain complex of black and white graphs is a (rational) model for the classifying spaces of mapping class groups of open-closed cobordisms.  More specifically there is an isomorphism
\[\mathrm{H}_*(\bwgraphs)\cong\mathrm{H}_*\left( \coprod_{\Sg}\mathrm{B}\Modgpq\right) \]
where the disjoint union runs over all topological types of open-closed cobordisms in which each connected component has at least one boundary component which is neither free nor outgoing closed. 
\end{thma}

In this paper, we give a new proof of the integral version of this theorem using Theorem \ref{thmA}.  More precisely, we construct a filtration 
\[\Fatad\ldots \supset \Fat^{n+1} \supset \Fat^{n} \supset \Fat^{n-1} \ldots \Fat^{1}\supset\Fat^{0}\]
that gives a cell-like structure on $\Fatad$ where the quasi-cells are indexed by black and white graphs i.e., $\vert \Fat^{n}\vert / \vert \Fat^{n-1}\vert\cong \vee S^n$, where the wedge sum is indexed by black and white graphs of degree $n$.  

Although the admissible fat graph model and the black and white graph model are 
abstractly equivalent since they are both models for the classifying space of mapping class groups, a direct connection between them was to our knowledge, so far missing.  Besides proving Theorem \ref{thmB}, the structure of the proof gives a direct connection between these models.  Furthermore, this connection is used to define a notion of composition of cobordisms in terms of black and white graphs as it is explained below.

\subsection{Models of the open-closed cobordism category}

The (positive-boundary) open-closed cobordism category $\mathcal{OC}$ is the category enriched over chain complexes with objects pairs of natural numbers $[^{p_1}_{p_2}]\in \N\times\N$ and mapping spaces given by 
\[\Hom_{\mathcal{OC}}([^{p_1}_{p_2}],[^{q_1}_{q_2}]):=\bigoplus_S \mathrm{C}_*(\cM_S)\simeq \bigoplus_S \mathrm{C}_*(\mathrm{BMod}(S))\]
where the direct sum is taken over all cobordisms $S$ with $p_1$ incoming closed boundaries, $p_2$ incoming open boundaries, $q_1$ outgoing closed boundaries and $	q_2$ outgoing open boundaries, such that each connected component of $S$ has a boundary component which is neither free nor outgoing closed. Composition is given by sewing cobordisms along the boundary using the parametrizations.  Theorem \ref{thmA} states that admissible fat graphs model the mapping spaces of $\mathcal{OC}$.  We use the ideas of Kaufmann, Livernet and Penner in \cite{kaufmann_livernet_penner} to define a composition of admissible fat graphs which models composition in $\mathcal{OC}$.  More precisely, let $|\Fatad_S|$ denote the connected component of $|\Fatad|$ corresponding to the cobordism $S$.  Then we prove the following result. 
\begin{thma}
\label{thmC}
Let $S_1$ and $S_2$ be composable cobordisms such that the composite $S_2\circ S_1$ is an oriented cobordism in which each connected component has a boundary component which is neither free nor outgoing closed.  We construct a continuous map
\[|\Fatad_{S_2}| \times |\Fatad_{S_1}| \longrightarrow |\Fatad_{S_2 \circ S_1}|\]
which models composition on classifying spaces of mapping class groups under the equivalence of Theorem \ref{thmA}.
\end{thma}
The composition map is defined by scaling metric fat graphs and as a result it is not associative on the nose.  However, one can find chain models of $|\Fatad|$ on which composition is strictly associative.  An example of this will be described below.

Costello's black and white fat graphs also give a chain model for the mapping spaces in $\mathcal{OC}$.  However, since black and white graphs are born from modeling moduli spaces of surfaces with nodal boundary, they do not carry a natural notion of composition along closed boundary components.  In fact, Costello states in \cite{costellotcft} that he expects that one can not model composition via fat graphs.  However, we show here that this is in fact possible.  To do this, we use the direct connection between admissible fat graphs and black and white graphs established in the proof of Theorem \ref{thmB} to transfer Theorem \ref{thmC} to Costello's black and white model.  More precisely, let $\mathscr{BW}_S$ denote the sub-complex of $\bwgraphs$ corresponding to the cobordism $S$.  We prove the following result.

\begin{thma}
\label{thmD}
Let $S_1$ and $S_2$ be composable cobordisms such that the composite $S_2\circ S_1$ is an oriented cobordism in which each connected component has a boundary component which is neither free nor outgoing closed.  We describe a chain map
\[\circ_{BW}: \mathscr{BW}_{S_2} \otimes \mathscr{BW}_{S_1}\longrightarrow \mathscr{BW}_{S_2 \circ S_1}\]
which models composition on classifying spaces of mapping class groups.  Furthermore, composition is associative showing that $\bwgraphs$ are indeed a model of the open-closed cobordism category $\mathcal{OC}$. 
\end{thma}
This composition map was first described in \cite{wahlwesterland} as part of their study of operations on the Hochschild homology of structured algebras.  They show that it is indeed a chain map and that composition is associative.  When all the outgoing boundary components of $S_1$ are open, this map restricts to the one given by Costello in \cite{costellotcft}.

In \cite{kaufmann_penner}, Kaufmann and Penner describe a different partial model of the open-closed cobordism category in terms of families of arcs embedded in surfaces.  Their construction is a partial model of $\mathcal{OC}$, because in some cases composition leaves moduli space.  On the other hand, their construction is in some sense dual to the one presented here and it is probable that one can restrict their construction to special families of arcs, say admissible open-closed arc systems, to obtain a dual model of $\mathcal{OC}$.  See Section 6.7 of \cite{wahlwesterland} for more details on the Kaufmann-Penner model and this duality.

The organization of the paper is as follows.  Section 1 gives preliminary definitions of fat graphs, their morphisms and their fattening to a surface.  Section 2 describes the categorical models of fat graphs and gives the proof of Theorems \ref{thmA} and \ref{thmC}.  Section 3 describes the chain complex of black and white graphs and gives the proofs of Theorems \ref{thmB} and \ref{thmD}.

\emph{Acknowledgements.} I would like to thank Nathalie Wahl for many interesting questions and discussions. I would also like to thank Oscar Randal-Williams and Angela Klamt for helpful discussions and comments.  The author was supported by the Danish National Research Foundation through the Center for Symmetry and Deformation (DNRF92).

\section{Preliminary definitions}

We give the basic definitions regarding fat graphs, their realizations and morphisms.

\begin{dfn}
A \emph{combinatorial graph} $G$ is a tuple $G=(V,H,s,i)$, consisting of a finite set of \emph{vertices} $V$, a finite set of \emph{half edges} $H$, a \emph{source map} $s:H\to V$ and an involution with no fixed points $i:H\to H$.
The map $s$ ties each half edge to its source vertex and the involution $i$ attaches half edges together.  An \emph{edge} of the graph is an orbit of $i$.  The valence of a vertex $v\in V$, denoted $|v|$, is the cardinality of the set $s^{-1}(v)$ and a \emph{leave} of a graph is a univalent vertex.  
\end{dfn}

\begin{dfn}
The \emph{geometric realization} of a combinatorial graph $G$ is the CW-complex $|G|$ with one 0-cell for each vertex, one 1-cell for each edge and attaching maps given by $s$.
\end{dfn}

\begin{dfn}
A \emph{tree} is a graph whose geometric realization is a contractible space and a \emph{forest} is a graph whose geometric realization is the disjoint union of contractible spaces.
\end{dfn}

\begin{dfn}
\label{fat_definition}
A \emph{fat graph} or \emph{ribbon graph} $\Gamma=(G,\sigma)$ is a combinatorial graph together with a cyclic ordering $\sigma_v$ of the half edges incident at each vertex $v$. The \emph{fat structure} of the graph is given by the data $\sigma=(\sigma_v)$ which is a permutation of the half edges. Figure \ref{Fat_example} shows some examples of fat graphs.  We denote by $|\Gamma|$ the geometric realization of $\Gamma$.  Note that this is independent of the fat structure i.e., $|\Gamma|=|G|$. 
\end{dfn}

\begin{figure}[h!]
  \centering
    \includegraphics[scale=0.7]{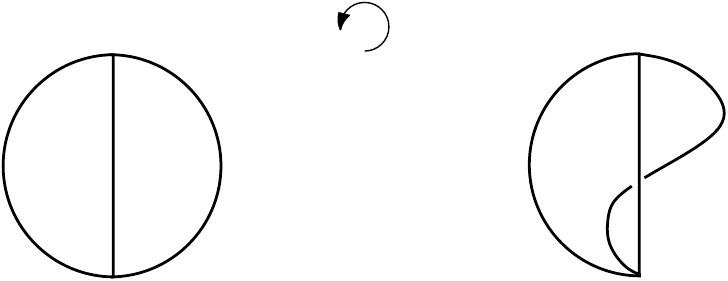}
  \caption{Two different fat graphs which have the same underlying combinatorial graph.  The fat structure is given by the orientation of the plane.}
  \label{Fat_example}
\end{figure}

\begin{dfn}
The \emph{boundary cycles} of a fat graph are the cycles of the permutation of half edges given by $\omega=\sigma\circ i$.  Each boundary cycle $c$ gives a list of half edges and determines a list of edges (possibly with multiplicities) of the fat graph $\Gamma$, those edges containing the half edges listed in $c$.  The \emph{boundary cycle sub-graph} corresponding to $c$ is the subspace of $\vert \Gamma \vert$ given by the edges determined by $c$ which are not leaves.  When clear from the context we will refer to a boundary cycle sub-graph simply as boundary cycle.
\end{dfn}

\begin{rmk}
From a fat graph $\Gamma=(G,\sigma)$ one can construct a surface with boundary $\Sigma_{\Gamma}$ by fattening the edges.  More explicitly, one can construct this surface by replacing each edge with a strip, each vertex with a disk and gluing these strips at a vertex according to the fat structure.  Notice that there is a strong deformation retraction of $\Sigma_{\Gamma}$ onto $|\Gamma|$ so one can think of $|\Gamma|$ as the skeleton of the surface.  The fat structure of $\Gamma$ is completely determined by $\omega$. Moreover, one can show that the boundary cycles of a fat graph $\Gamma=(G,\omega)$ correspond to the boundary components of $\Sigma_{\Gamma}$ \cite{Godinunstable}.  Therefore, the surface $\Sigma_{\Gamma}$ is completely determined, up to homeomorphism, by the combinatorial graph and its fat structure.
\end{rmk}

\begin{dfn}
A \emph{morphism of combinatorial graphs} $\varphi:G\to\tilde{G}$ is a map of sets $\varphi:V_G\coprod H_G\to V_{\tilde{G}}\coprod H_{\tilde{G}}$ such that
\begin{itemize}
\item[-]For every vertex $v\in V_{\tilde{G}}$ the preimage $\varphi^{-1}(v)$ is a tree in G.
\item[-]For every half edge $A\in H_{\tilde{G}}$ the preimage $\varphi^{-1}(A)$ contains exactly one half edge of $G$.
\item[-]The following diagrams commute
\begin{equation*}
\begin{array}{ccc}
\begin{tikzpicture}[scale=0.7]
\node (a) at (0,2){$V_G\coprod H_G$};
\node (b) at (4,2) {$V_G\coprod H_G$};
\node (c) at (0,0){$V_{\tilde{G}}\coprod H_{\tilde{G}}$};
\node (d) at (4,0) {$V_{\tilde{G}}\coprod H_{\tilde{G}}$};
\path[auto,arrow,->] (a) edge node{$\tilde{s}_{\scriptscriptstyle{G}}$}(b)
    		         (b) edge node{$\varphi$}(d)		             
    		         (a) edge node[swap]{$\varphi$}(c)
    		         (c) edge node [swap] {$\tilde{s}_{\tilde{\scriptscriptstyle{G}}}$}(d);
\end{tikzpicture}
& \phantom{test} &
\begin{tikzpicture}[scale=0.7]
\node (a) at (0,2){$V_G\coprod H_G$};
\node (b) at (4,2) {$V_G\coprod H_G$};
\node (c) at (0,0){$V_{\tilde{G}}\coprod H_{\tilde{G}}$};
\node (d) at (4,0) {$V_{\tilde{G}}\coprod H_{\tilde{G}}$};
\path[auto,arrow,->] (a) edge node{$\tilde{i}_{\scriptscriptstyle{G}}$}(b)
    		         (b) edge node{$\varphi$}(d)		             
    		         (a) edge node[swap]{$\varphi$}(c)
    		         (c) edge node [swap] {$\tilde{i}_{\tilde{\scriptscriptstyle{G}}}$}(d);
\end{tikzpicture}
\end{array}
\end{equation*} 
where $\tilde{i}$, respectively $\tilde{s}$, is the extension of the involution $i$, respectively the source map $s$, to $V\coprod H$ by the identity on $V$.
\end{itemize}
\end{dfn}

\begin{dfn}
A \emph{morphism of fat graphs} $\varphi:(G,\omega)\to(\tilde{G},\tilde{\omega})$ is a morphism of combinatorial graphs which respects the fat structure i.e., $\varphi(\omega)=\tilde{\omega}$. 
\end{dfn}
\begin{rmk}
\label{unique_iso}
Note that, if two fat graphs $\Gamma$, $\tilde{\Gamma}$ are isomorphic and they have at least one leaf in each connected component, and these leaves are labeled by $\lbrace 1,2, 3 \ldots k\rbrace$ i.e., the leaves are ordered, then there is unique morphism of graphs that realizes this isomorphism while respecting the labeling of the leaves.  Thus, a  fat graph $\Gamma$ that has at least one labeled leaf in each connected component has no automorphisms besides the identity morphism.
\end{rmk}
\begin{rmk}
Note that a morphism of combinatorial graphs induces a simplicial, surjective homotopy equivalence on geometric realizations and does not change the number of boundary cycles.  Thus, if there is a morphism of fat graphs $\varphi:\Gamma\to\tilde{\Gamma}$ then the surfaces $\Sigma_{\Gamma}$ and $\Sigma_{\tilde{\Gamma}}$ are homeomorphic.
\end{rmk}

\section{Categories of fat graphs}

\subsection{The definition}
We now define the basic objects and morphisms that form the categories of fat graphs that we will study.

\begin{dfn}
\label{oc_def}
Let $L_{\Gamma}$ denote the set of leaves of a fat graph $\Gamma$.  An \emph{open-closed fat graph} is a triple $\Gamma^{oc}=(\Gamma, In, Closed)$ where $\Gamma$ is a fat graph with leaves $In,Closed\subset L_{\Gamma}$, together with an ordering of the leaves in $In$ and the leaves in $L_{\Gamma}-In$.  If a leave $v\in In$ we call it \emph{incoming}, else we call it \emph{outgoing}.  Similarly if a leave $v\in Closed$ we call it \emph{closed}, else we call it \emph{open}. The triple $\Gamma^{oc}$ should be given such that the following hold:
\begin{itemize}
\item[-] All inner vertices are at least trivalent
\item[-] A closed leaf must be the only leaf in its boundary cycle
\end{itemize}
We allow degenerate graphs which are a corolla with $1$ or $2$ leaves.  Figure \ref{openclosed_example} shows an example of an open-closed fat graph.
\end{dfn}

\begin{figure}
  \centering
    \includegraphics[scale=0.7]{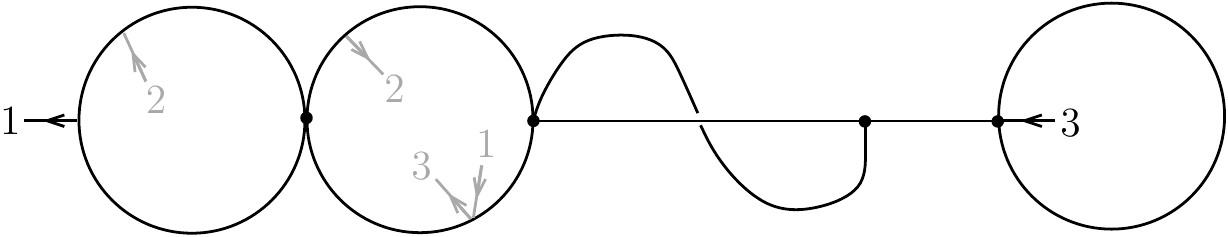}
  \caption{An example of a closed fat graph which is not admissible.  The incoming and outgoing leaves are marked by incoming or outgoing arrows. The closed leaves are depicted in black and the open ones in grey.}
  \label{openclosed_example}
\end{figure}

\begin{rmk}
\label{fatten_cobordism}
From an open-closed fat graph one can construct an open-closed cobordism $\Sg$.  First construct a bordered oriented surface $\Sigma_{\Gamma}$ as for a regular fat graph.  Now, divide the boundary by the following procedure. For a boundary component corresponding to a closed leave, label the entire boundary component as incoming or outgoing according to the labeling of the leaf and choose a marked point on the boundary. For a boundary component corresponding to one or more open leaves assign to each leaf a small part of the boundary (homeomorphic to the unit interval) such that none of these intervals intersect and such that they respect the cyclic ordering ordering of the leaves on the corresponding boundary cycle.  Then label such intervals as incoming or outgoing according to their corresponding leaves and choose a marked point in each interval. Label the rest of the boundary as free.  Finally order the marked points at the boundary according to the ordering of their corresponding leaves.  This gives and open-closed cobordism $\Sg$ well defined up to topological type.  
\end{rmk}

The following is a slight variation of a definition due to Godin in \cite{godin} of a special kind of open-closed fat graph.

\begin{dfn}
\label{ad_def}
An \emph{admissible fat graph} $\Gamma^{ad}=(\Gamma, In, Closed)$ is an open-closed fat graph in which all outgoing closed boundary cycles are disjoint embedded circles in $|\Gamma|$.  Figure \ref{Admissible_example} shows an example of an admissible fat graph and Figure \ref{openclosed_example} shows an example of an open-closed fat graph which is not admissible.
\end{dfn}
Note that an open-closed fat graph, which is not a corolla, can not be an admissible fat graph if all of its leaves are outgoing closed.

\begin{figure}[h!]
  \centering
    \includegraphics[scale=0.7]{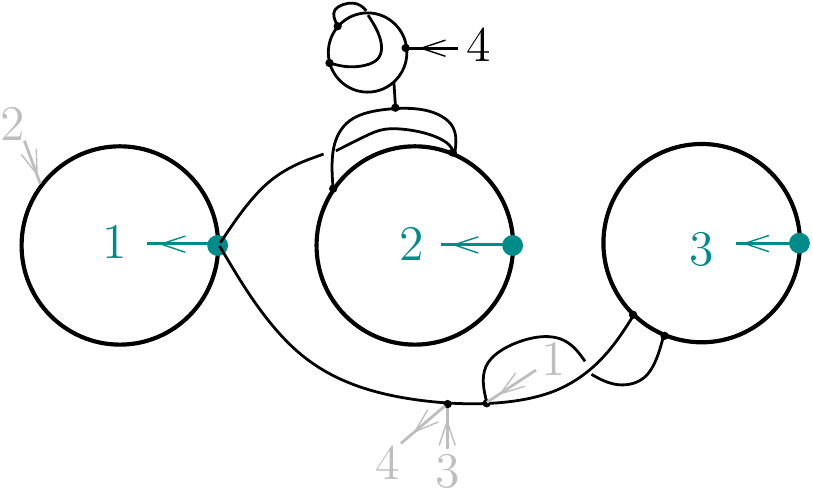}
  \caption{An example of an admissible fat graph.  The admissible leaves (outgoing closed) are pictured in green.}
  \label{Admissible_example}
\end{figure}

\begin{notation}
When it is clear from the context we will simply write $\Gamma$ instead of $\Gamma^{oc}$ or $\Gamma^{ad}$
\end{notation}

\begin{dfn}
A \emph{morphism of open-closed fat graphs} is a morphism of fat graphs which respects the labeling of the leaves. Two morphisms $\varphi_i:\Gamma_i\to\tilde{\Gamma_i}$ for $i=1,2$ are equivalent if there are isomorphisms which make the following diagram commute
\begin{equation*}
\begin{tikzpicture}[scale=0.7]
\node (a) at (0,2){$\Gamma_1$};
\node (b) at (3,2) {$\tilde{\Gamma}_1$};
\node (c) at (0,0){$\Gamma_2$};
\node (d) at (3,0) {$\tilde{\Gamma}_2$};
\path[auto,arrow,->] (a) edge node{$\varphi_1$}(b)
    		         (b) edge node{$\cong$}(d)		             
    		         (a) edge node[swap]{$\cong$}(c)
    		         (c) edge node [swap] {$\varphi_2$}(d);
\end{tikzpicture}
\end{equation*} 
\end{dfn}
\begin{rmk}
Let $[\Gamma]$ and $[\Gamma']$ be two isomorphism classes of open-closed fat graphs.  One can show that all morphisms $[\varphi]:[\Gamma]\to\tilde{[\Gamma]}$ can be realized uniquely as a collapse of a sub-forest of $\Gamma$ which does not contain any leaves.  The argument is exactly the same as the one given in \cite{Godinunstable} for the case where all leaves are incoming closed.
\end{rmk}

\begin{dfn}
The category of open-closed fat graphs $\Fatoc$ is the category with objects isomorphism classes of open-closed fat graphs with at least one leaf on each component and morphisms equivalences classes of morphisms.  The category of admissible fat graphs $\Fatad$ is the full subcategory of $\Fatoc$ on objects isomorphism classes of admissible fat graphs.
\end{dfn}

\begin{rmk}
These categories are slightly different than the ones given in \cite{godin} since there are no leaves for the free boundary components.  However, the exact same argument given in \cite{Godinunstable} shows that these categories are well defined.  More precisely, composition is well defined since as given in Remark \ref{unique_iso}, there is a unique isomorphism of open-closed fat graphs between two open-closed fat graphs with at least one leaf on each component and an open-closed fat graph of such kind has no automorphisms besides the identity morphism.
\end{rmk}

\subsection{Fat graphs as models for the mapping class group}
The categories $\Fatoc$ and $\Fatad$ are introduced by Godin in \cite{godin}.  In this paper, she shows that both categories are models of the classifying space of the mapping class group by comparing a sequence of fibrations.  However, there is a step missing in the proof which we do not know how to complete.  More precisely, Godin proves this by comparing certain fiber sequences, but a map connecting them is not explicitly constructed and we do not know how to construct such map.  In this section we give a new proof, more geometric in nature, that shows that these categories model mapping class groups, following the ideas of \cite{Godinunstable}.

\begin{thm}
\label{ad_oc}
The categories of open-closed fat graphs and admissible fat graphs are models for the classifying spaces of mapping class groups of open-closed cobordisms.  More specifically there is a homotopy equivalence
\[|\Fatoc|\simeq \coprod_{\Sg} \mathrm{B}\Modgpq\]
where the disjoint union runs over all topological types of open-closed cobordisms in which each connected component has at least one boundary component which is not free.  Moreover, this restricts on the subcategory of admissible fat graphs to a homotopy equivalence 
\[|\Fatad|\simeq \coprod_{\Sg} \mathrm{B}\Modgpq\]
where the disjoint union runs over all topological types of open-closed cobordisms in which each connected component has at least one boundary component which is neither free nor outgoing closed. 
\end{thm}

Let $\Fatocg$ and $\Fatadg$ denote the full subcategories  with objects open-closed fat graphs of topological type $S$ i.e., which fatten to a cobordism $S$ as in Remark \ref{fatten_cobordism}. Note that a morphism of open-closed fat graphs respects the structure that determines the topological type of the graph as an open-closed cobordism.  Therefore we have the following isomorphisms:
\begin{align*}
\Fatoc\cong\coprod_{\Sg}\Fatocg &  & \Fatad\cong\coprod_{\Sg}\Fatadg
\end{align*}
The idea of the proof of the theorem is to show there is a homotopy equivalence on each connected component by constructing coverings of $\Fatocg$ and $\Fatadg$ which have contractible realizations and admit a free action of their corresponding mapping class group which is transitive on the fibers. 

\begin{notation}
For each topological type of open-closed cobordism, with $p$ incoming boundary components and $q$ outgoing boundary components, choose and fix a representative $\Sg$  and let $x_k$ denote the marked point in the $k$-th incoming boundary for $1\leq k\leq p$ and $x_{p+k}$ denote the marked point on the $k$-th outgoing boundary $1\leq k\leq q$.   Given an open-closed fat graph $\Gamma^{oc}$, let $v_{in,k}$ denote the $k$-th incoming leaf and $v_{out,k}$ denote the $k$-th outgoing leave. 
\end{notation}

\begin{dfn}
A \emph{marking} of an open-closed fat graph is an isotopy class of embeddings $H:|\Gamma^{oc}|\cof \Sg$ such that $H(v_{in,k})=x_k$, $H(v_{out,k})=x_{p+k}$, $H(|\Gamma|)\subset S$ is a deformation retract of $S$ and the fat structure of $\Gamma^{oc}$ coincides with the one induced by the orientation of the surface.  We will call the pair $([\Gamma^{oc}],[H])$ a marked open-closed fat graph.   
\end{dfn}


\begin{rmk}
Let $\Gamma$ be an admissible fat graph, $F$ be a forest in $\Gamma$ which does not contain any leaves of $\Gamma$ and $H$ be a representative of a marking $[H]$ of $\Gamma$.  Since $[H]$ is a marking, the image of $H|_F$ (the restriction of $H$ to $|F|$) is contained in a disjoint union of disks away from the boundary.  Therefore, the marking $H$ induces a marking $H_F:|\Gamma/ F|\cof \Sg$ given by collapsing each of the trees of $F$ to a point of the disk in which their image is contained.  Note that $H_F$ is well defined up to isotopy and it makes the following diagram commute up to homotopy
\begin{equation*}
\begin{tikzpicture}[scale=0.5]
\node (a) at (0,3){$\vert\Gamma\vert$};
\node (b) at (5,3) {$\vert\Gamma/ F\vert$};
\node (c) at (5,0){$\Sg$};
\path[auto,arrow,->] (a) edge  [left hook->] node[swap]{$$ \mbox{\footnotesize $H$ } $$}  (c)
    		                      (b) edge  [right hook->] node{$$ \mbox{\footnotesize $H_F$ } $$} (c);
\path[auto,arrow,->>] (a) edge (b);
\end{tikzpicture}
\end{equation*}
\end{rmk}

\begin{dfn}
Define the category $\Eoc$ to be the category with objects marked open-closed fat graphs $([\Gamma^{oc}],[H])$ and morphisms given by morphisms in $\Fatoc$ where the map acts on the marking as stated in the previous remark.  Define $\Ead$ to be the full subcategory of $\Eoc$ with objects $([\Gamma^{ad}],[H])$ marked admissible fat graphs. 
\end{dfn}

\begin{proof}[Proof of Theorem \ref{ad_oc}]
It is enough to show the result in each connected component. Let $\Eocg$ and $\Eadg$ be the full subcategories of $\Eoc$ and $\Ead$ corresponding to marked fat graphs of topological type $\Sg$.  There are natural projections given by forgetting the marking, making the following square commute.
\begin{equation*}
\begin{tikzpicture}[scale=0.5]
\node (a) at (0,3){$\Eadg$};
\node (b) at (5,3) {$\Eocg$};
\node (c) at (0,0){$\Fatadg$};
\node (d) at (5,0) {$\Fatocg$};
\path[auto,arrow,->] (a) edge [right hook->] (b)
    		         (c) edge [right hook->] (d);
\path[auto,arrow,->>] (a) edge (c)
    		         (b) edge (d);
\end{tikzpicture}
\end{equation*}
where the horizontal maps are inclusions.  We show that there is a free action of $\Modgpq$ on $\Eocg$ with quotient $\Fatocg$ i.e., we show that $\Modgpq$ acts on $\vert \Eocg\vert$ and we show that this action is free and transitive on the fibers by showing that it is free and transitive on the $0$-simplices.

The mapping class group acts on $\Eoc$ by composition with the marking. Thus, it is enough to show that this group acts freely and transitively on the markings i.e., for any two markings $[H_1]$ and $[H_2]$ there is a unique $[\varphi]\in\Modgpq$ such that $[\varphi\circ H_1]=[H_2]$.  Given two such markings, we will construct a homeomorphism $g:\Sg\to\Sg$ such that $[g\circ H_1]=[H_2]$ which we can approximate by a diffeomorphism by Nielsen's approximation theorem \cite{nielsen}. By definition $\Sg\setminus H_1(\Gamma)$ has $p+q+f$ connected components where  $f$ is the number of free boundary components of $\Sg$, say $\Sg\setminus H_1(\Gamma):=\sqcup_i S_i$ for $1\leq i\leq p+q+f$.  Moreover, each component $S_i$ is of one of the following forms: 
\begin{itemize}
\item[$i$] If there is exactly one leaf in a boundary cycle, then $S_i$ is a disk bounded by the image under $H_1$ of the given boundary cycle and its leave and by the corresponding boundary component.
\item[$ii$] If there is more than one leaf on a boundary cycle, then $S_i$ is a disk bounded by the image under $H_1$ of part of the boundary cycle and part of the corresponding boundary component (the sections bounded by consecutive leaves).
\item[$iii$] If there is no leaf in a boundary cycle, then $S_i$ is an annulus with boundaries the image of $H_1$ of the given boundary cycle and its corresponding boundary component.
\end{itemize}

We construct $g$ by defining homeomorphisms in each component which can be glued together consistently. Order the $S_i$'s according to the ordering of the incoming and outgoing leaves and a chosen ordering of the free boundary components. If $S_i$ is of the types $(i)$ or $(ii)$ then the corresponding boundary component of the surface is not free.  So the restriction of $g$ to such component should give a map $g_i:S_i\to S_i$. In this case, define $\tilde{g}_i:\partial S_i\to\partial S_i$ to be the identity on the boundary section and to be $H_2\circ H_1^{-1}$ on the image of the boundary cycle.  Since $S_i$ is homeomorphic to a disk, we can extend $\tilde{g}_i$ to a map $g_i:S_i\to S_i$ which is uniquely defined up to homotopy. On the other hand, if $S_i$ is of type $(iii)$ then the corresponding boundary component is free and thus the restriction of $g$ should give a map $g_{ij}:S_i\to S_j$ where $S_j$ also corresponds to a free boundary component and it could be that $i=j$. In this case, define $\tilde{g}_{ij}:\partial S_i\to\partial S_j$ to be $H_2\circ H_1^{-1}$ on the image of the boundary cycle and a homeomorphism homotopic to the identity on the boundary of the surface.  This morphism can be extended, though not uniquely, to a map $g_{ij}:S_i\to S_j$.  Choose any extension of such map. These maps can be glued together giving the desired map $g$ which we can approximate by a diffeomorphism $\varphi$.  Finally, two non-homotopic extensions of $g_{ij}$ differ only by powers of a Dehn twists around the free boundary. Thus $[\varphi]$ is determined uniquely in $\Modgpq$.  This argument restricts to the subcategory $\Eadg$.  

Propositions \ref{A0_E}, \ref{A0_con}, \ref{B0_E} and \ref{B0_con} in the next subsection show that $\vert \Eocg\vert $ and $\vert \Eadg\vert$ are finite contractible CW-complexes, which finishes the proof. 
\end{proof}

\subsubsection{The categories of marked fat graphs are contractible}
In this section we describe how the categories of marked fat graphs are dual to the categories of arcs embedded in a surface and use this to show that $\Eocg$ and $\Eadg$ are contractible categories by using Hatcher's proof of the contractibility of the arc complex.  

\begin{dfn}
Let $\Sigma$ be an orientable surface and $V\subset\partial \Sigma$ a subspace with finitely many connected components each of which is a closed interval or a circle.  
\begin{itemize}
\item[-]An \emph{essential arc} $\alpha_0$, is an embedded arc in $\Sigma$ that starts and ends at $\partial \Sigma - V$, intersects $\partial \Sigma$ only at its endpoints and it is not boundary parallel i.e., $\alpha_0$ does not separate $\sigma$ into two components one of which is a disk that contains only one connected component of $V$ which is homeomorphic to an interval.
\item[-] An \emph{arc set} $\alpha$ in $\Sigma$ is a collection of arcs $\alpha=\lbrace\alpha_0,\alpha_1,\cdots\alpha_n\rbrace$ such that their interiors are pairwise disjoint and no two arcs are ambient isotopic relative to $V$.
\item[-] An \emph{arc system} $[\alpha]$ in $\Sigma$ is an ambient isotopy class of arc sets of $\Sigma$ relative to $V$.
\item[-] An arc system is \emph{filling} if it separates $\Sigma$ into polygons.
\end{itemize}
\end{dfn}

The following definition and result is originally due to Harer in \cite{Harer_arc} to which later on Hatcher gives a very beautiful and simple proof in \cite{hatcher_arc}

\begin{dfn}
Let $\Sigma$ be an orientable surface and $V\subset\partial \Sigma$ a subspace as described above.  The \emph{arc complex}, $\Acat(\Sigma,V)$, is the complex with vertices isotopy classes of essential arcs $[\alpha_0]$, $k$ simplices arc systems of the form $[\alpha]=[\alpha_0,\alpha_1,\cdots\alpha_k]$ and faces obtained by passing to sub-collections.
\end{dfn}

\begin{thm}[\cite{Harer_arc}, \cite{hatcher_arc}]
\label{harer_contractible}
The complex $\Acat(\Sigma,V)$ is contractible whenever $\Sigma$ is not a disk or an annulus with one of its boundary components contained in $V$.
\end{thm}

\begin{rmk}
The arc complex was originally described in terms of arcs with endpoints on a finite set of points $W$ in $\Sigma$.  These descriptions are equivalent.  To obtain the original description from the one above, start with a surface $\Sigma$ and $V\subset \partial \Sigma$ as above.  Then collapse each connected component of $\partial \Sigma - V$ to a single point.  One obtains a new surface $\Sigma'$ with a finite set of marked points $W$, one for each component of $\partial \Sigma - V$.  The vertices of the arc complex are then isotopy classes of arcs in $\Sigma'$ with endpoints in $W$.  
\end{rmk}

We now make the connection between arc systems and  marked fat graphs.

\begin{dfn}
Let $\Sg$ be an open-closed cobordism and $\Upsilon=\lbrace x_1,\ldots x_{p+q}\rbrace$ be the set of marked points on the boundary given by the boundary parametrizations.  For each $x_i\in\Upsilon$ which is closed, choose a closed interval $I_{x_i}\subset \partial S$ around $x_i$ and let 
\[\Delta=\partial_{\text{open}} S \bigcup \cup_{xi} I_{x_i}\]
where the union is taken over all closed marked points $x_i$.  Let $\Acat_0(\Sg,\Delta)$ denote the poset category of filling arc systems ordered by inclusion. In the case where $\Sg$ is a disk with $p+q\geq 1$, the surface is already a polygon and thus we consider the empty set to be a filling arc system.
\end{dfn}

\begin{prop}
\label{A0_E}
There is an isomorphism of categories $\Acat_0(\Sg,\Delta)^{op}\cong\Eocg$
\end{prop}
\begin{proof}
It is enough to show this for a connected cobordism $S$. Throughout the proof $S$ will be fixed, so for simplicity we will denote the category $\Eocg$ as $\Ecat$ and the category $\Acat_0(S,\Delta)$ as $\Acat_0$.  We will construct contravariant inverse functors 
\[\Phi:\Acat_0\adj\Ecat:\Psi.\]
We first define the functor $\Phi$ on objects. Let $[\alpha]=[\alpha_0,\ldots,\alpha_k]$ be a filling arc system and choose a representative arc set $\alpha=\lbrace\alpha_0,\ldots,\alpha_k\rbrace$.  Then, $S\setminus\alpha=\coprod_i T_i$ is a disjoint union of polygons.  Construct a fat graph $\Gamma$ on the surface $S$ by setting a vertex $v_i$ in each $T_i$.  If $T_i$ and $T_j$ are bordering components separated by an arc $\alpha_{ij}$ connect $v_i$ with $v_j$ with an edge $e_{ij}$ that crosses only $\alpha_{ij}$ and crosses it exactly once.  Moreover, if the marked point $x_j\in T_i$ connect $v_i$ with $x_j$ via an edge $l_j$. Make all edges non-intersecting on the surface.  Each polygon $T_i$ has an induced orientation coming from $S$, this gives a cyclic ordering of the edges incident at $v_i$.  Note that the $x_j$'s are leaves of $\Gamma$.  Moreover, by construction $\Gamma$ comes with a natural marking $[H]$ on $S$.  So set $\phi(\alpha)=([\Gamma], [H])$. Note that all the polygons $T_i$ have at least three bounding arcs or are of the form shown in Figure \ref{trivalent} with at least two components of $\Delta$ bounded by an essential arc.  Thus, all inner vertices in $\Gamma$ are at least trivalent, so $\phi(\alpha)$ is an object of $\Ecat$.  Moreover, setting $\Phi([\alpha])=\phi(\alpha)$ is well defined since two representatives $\alpha$ and $\beta$ of the arc system $[\alpha]$ are ambient isotopic so they split the surface in the same number of connected components giving isomorphic underlying fat graphs.  Moreover, we can use the ambient isotopy connecting both representatives to show that they induce the same marking on $[\Gamma]$. 

To define $\Phi$ on morphisms, let $[\beta]$ be a face of $[\alpha]$.  We can find representatives such that $\alpha=\beta\cup\lbrace\alpha_0,\ldots,\alpha_n\rbrace$.  Note that if the edges corresponding to $\lbrace\alpha_0,\ldots,\alpha_n\rbrace$ form a cycle on $\phi(\alpha)$ then $\beta$ is not filling. Therefore, the edges corresponding to these arcs must form a forest and there is a uniquely defined morphism obtained from collapsing such forest which gives the map $\Phi([\alpha])\to\Phi([\beta])$.  This construction behaves well with composition.

\begin{figure}
  \centering
    \includegraphics[scale=0.3]{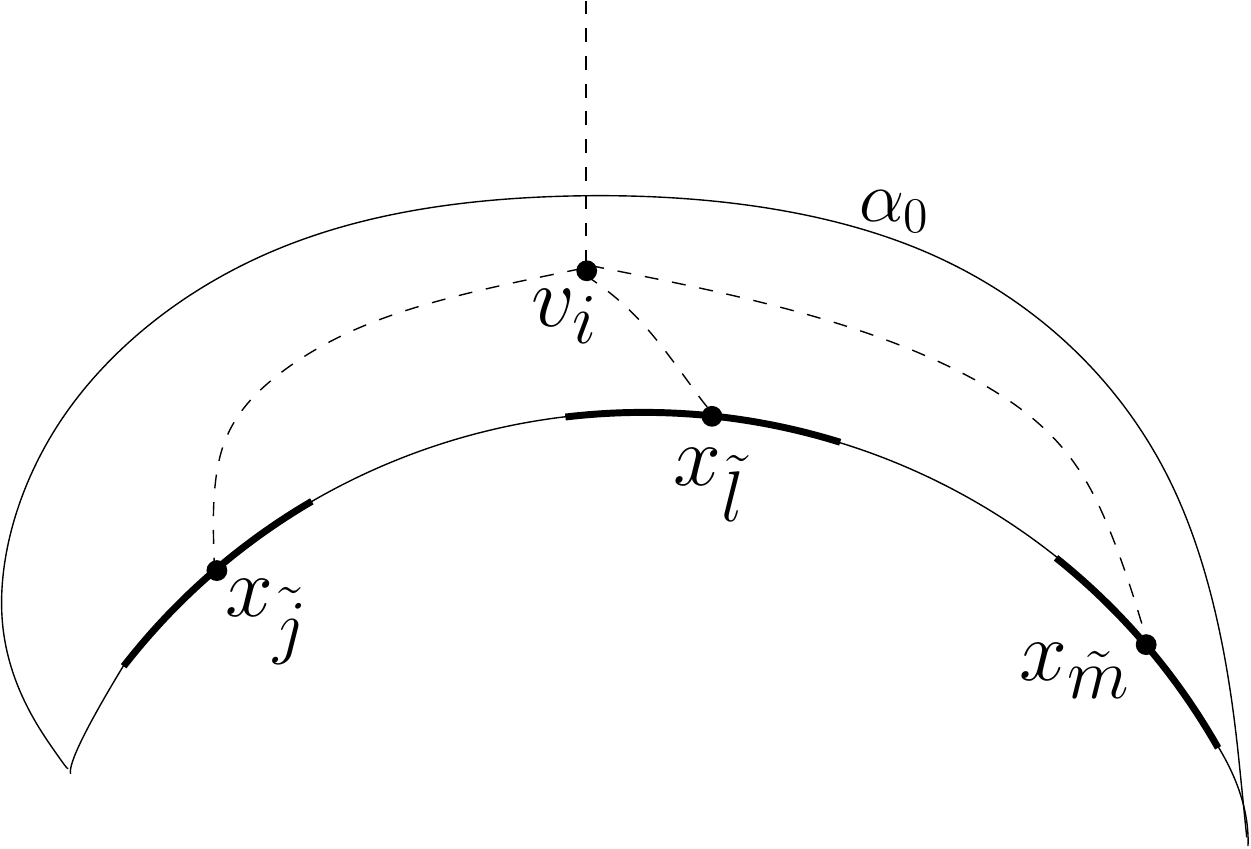}
  \caption{The components of $\Delta$ are marked by thin lines.  If $\alpha_0$ is essential, then there must be at least two components of $\Delta$ bounded by it on the boundary.}
  \label{trivalent}
\end{figure}

We now define the functor $\Psi$ on objects.  Let $W$ be the subspace $W:=\partial S-\Delta$.  The arcs in $\Acat_0$ have their endpoints in $W$, sometimes called the space of windows of $S$. Let $([\Gamma],[H])$ be an object of $\Ecat$.  Then for representatives $(\Gamma,H)$, the complement $S\setminus H(\Gamma)$ is a disjoint union of connected components, say $\coprod_i S_i$.  By construction the component $S_i$ is either a polygon which contains exactly one connected component of $W$ homeomorphic to an interval, or an annulus which contains exactly one connected component of $W$ homeomorphic to a circle.
Define an arc set $\psi(\Gamma, H)$ as follows.  If there is an edge $e_{ij}$ whose image under $H$ separates $S_i$ and $S_j$, then $\psi(\Gamma, H)$ has an arc $\alpha_{ij}$ crossing only $e_{ij}$.  The arc $\alpha_{ij}$ starts in $S_i\cap W$ and ends in $S_j\cap W$.  Notice that it might be that $i=j$ i.e., the arc starts and ends at the same component.  Now pull all arcs tight to make all their interiors non-intersecting and discard the arcs that crossed the leaves.  By construction this arc set is filling; thus, let $\Psi([\Gamma],[H])=[\psi(\Gamma, H)]$. As before this functor is well defined on objects and it is defined on morphisms in the a similar was as for $\Phi$.  Finally, the functors $\Phi$ and $\Psi$ are clearly inverses of each other.
\end{proof}

\begin{prop}
\label{A0_con}
The category $\Acat_0(\Sg,\Delta)$ is contractible.
\end{prop}
\begin{proof}
We follow a similar proof to the one given by Giansiracusa on \cite{giansiracusa} on a similar poset. Again let $W:=\partial S -\Delta$.  We will prove this by induction on the complexity of the cobordism namely on the tuple $k=(g,n,f,p+q)$ ordered lexicographically, where $g$ is the genus of the surface, $f$ is the number of connected components of $W$ which are homeomorphic to a circle (i.e., the number of free boundary circles of $S$), $n=\#(\pi_0(\partial S))-f$ (i.e., the number of boundary components of $S$ which contain a marked point) ,and  $p$ and $q$ are the number of incoming, respectively outgoing boundaries of the cobordism.  For the rest of the proof we will denote this by $(S,c(S)=k)$.

We start the induction with $k=(0,1,0,p+q)$ for any $p+q\geq1$.
In this case, the category $\Acat_0(S,c(S)=k)$ is contractible since it has the empty set as initial element.  Now let $k=(g,n,f,p+q)>(0,1,0,r)$ for any $r\geq1$ and assume contractibility holds for all $k'<k$. Let $\Pcat(\Acat(S,c(S)=k))$ be the poset category obtained from the arc complex, $\Acat(S,c(S)=k)$, by barycentric subdivision and let $\iota$ denote the inclusion:
\[\iota:\Acat_0(S,c(S)=k)\cof\Pcat(\Acat(S,c(S)=k))\]
For an object $[\alpha]$ in $\Pcat(\Acat(S,c(S)=k))$, consider the over category $[\alpha]\setminus\iota$ which in this case is the full subcategory of $\Acat_0(S,c(S)=k)$ with objects:
\[\Ob([\alpha]\setminus\iota)=\lbrace[\beta]\in\Acat_0(S,c(S)=k)|[\beta]\geq[\alpha]\rbrace\]
Note first the set of objects is not empty, since every arc system can be extended to a filling arc system.  Let $\alpha$ be a representative of $[\alpha]$.  Then, $(S,c(S)=k)\setminus\alpha$ is a disjoint union of cobordisms $\coprod_{i=1}^m (S_i,c(S_i)=k_i)$ and there is an isomorphism of categories
\[\Phi:[\alpha]\setminus\iota\cong\prod_{i=1}^m\Acat_0(S_i,c(S_i)=k_i):\Psi\]
The mutually inverse functors are given as follows.  Let $[\beta]$ be an object of $[\alpha]\setminus\iota$.  We can choose a representative such that $\beta=\alpha\cup\lbrace\beta_0,\ldots \beta_n\rbrace$. Each essential arc $\beta_i$ is completely contained in some $S_j$.  Denote $\lbrace\beta_{0_i}\cdots\beta_{l_i}\rbrace$ the arc set contained in $S_i$; this set fills $S_i$ and it is possibly empty if $S_i$ is already a disk. Set $\Phi([\beta])=\prod_{i=1}^m[\beta_{0_i}\cdots\beta_{l_i}]$ with the natural map on morphisms. This is a well defined functor with inverse $\Psi(\prod_{i=1}^m[\beta_{0_i}\cdots\beta_{l_i}])=[\alpha]\cup_{i=1}^m[\beta_{0_i}\cdots\beta_{l_i}]$ and the natural map on morphisms.  Now, since all arcs of $\alpha$ are essential, then $k_i<k$ for all $1\leq i\leq m$; so by the induction hypothesis $\Acat_0(S_i,c(S_i)=k_i)$ is contractible and thus $[\alpha]\setminus\iota$ is a contractible category.  Then, Quillen Theorem A gives that $\iota$ is a homotopy equivalence.  Finally, since $k > (0,1,0,p+q)$ then then either $f>0$, $n>1$ or $g>0$ so $(S,c(S)=k)$ is neither a disk, or an annulus with one of its boundary components contained in $\Delta$. Therefore, by Theorem \ref{harer_contractible}  $\Acat(S,c(S)=k)$ is contractible, which finishes the proof.
\end{proof}

We now look into the case of the admissible fat graphs, and give a geometric interpretation for such condition.

\begin{dfn}
\label{arc_admissible}
Let $\Sg$  be a cobordism with $p$ incoming boundary components, $q$ outgoing boundary components and $p+q\geq 1$.  Let $\partial_{i_1} S, \partial_{i_2} S \ldots \partial_{i_k} S$ be the boundary components which are outgoing closed and $\alpha$ be a filling arc set in the cobordism $\Sg$.  We say that $\alpha$ is \emph{admissible} if the following conditions hold.
\begin{itemize}
\item[$i$]  $\alpha$ has a subset of arcs which cut $\Sg$ into $k$ components such that the $j$-th component contains in its interior all the arcs with an endpoint at $\partial_{i_j} S$ for all $1\leq j\leq k$.
\item[$ii$]  $\alpha$ does not contain an arc with both endpoints at $\partial_{i_j} S$ for any $1\leq j\leq k$.
\item[$iii$] Let $\alpha_{j_1}, \alpha_{j_2} \ldots \alpha_{j_r}$ be the arcs in $\alpha$ with an endpoint in $\partial_{i_j}S$. For all $1\leq j\leq k$, the arc set $\alpha$ also contains arcs $\beta_{j_1},\beta_{j_2} \ldots \beta_{j_{\tilde{r}}}$ such that the subspace 
\[\left((\cup_l \alpha_{j_l})\bigcup(\cup_l \beta_{j_l})\right)\bigcap (S-\partial_{i_j}S)\subset S
\]
is connected. 
\end{itemize}
Note that conditions $i-iii$ are well defined for arc systems.  We define $\Bcat(\Sg,\Delta)$ to be the sub-complex of $\Acat(\Sg,\Delta)$ 
given by the simplicial closure of admissible arc systems.  Similarly, we define $\Bcat_0(\Sg,\Delta)$ to be the \emph{sub-poset of $\Acat_0(\Sg,\Delta)$ of filling admissible arc systems}.
\end{dfn}

\begin{prop}
\label{B0_E}
Let $[\alpha]$ be a filling arc system in the cobordism $\Sg$ and let $([\Gamma_{\alpha}],[H_{\alpha}])$ be its corresponding open-closed marked fat graph under the isomorphism of Theorem \ref{A0_E}.  The arc system $[\alpha]$ is admissible if and only if the graph $[\Gamma_{\alpha}]$ is admissible.   
\end{prop}

Before proving the proposition we will state an immediate corollary
\begin{cor}
There is an isomorphism of categories $\Bcat_0(\Sg,\Delta)^{op}\cong\Eadg$
\end{cor}
\begin{proof}
This isomorphism is just a restriction of the isomorphism of theorem \ref{A0_E}, which is well defined by the proposition above.
\end{proof}

\begin{proof}[Proof of Proposition \ref{B0_E}] 
Let $\partial_{i_1} S, \partial_{i_2} S \ldots \partial_{i_k} S$ be the boundary components of $\Sg$ which are outgoing closed, let $\alpha$ and $\Gamma_\alpha$ be representatives of the arc system and fat graph of the theorem.  Recall that $\Gamma_{\alpha}$ is an admissible fat graph if the boundary cycle sub-graphs corresponding to the $k$ outgoing closed leaves, say $C_1, C_2 \ldots C_k$ are disjoint circles in $|\Gamma_{\alpha}|$.  For this proof it will be convenient to reinterpret the admissibility condition in a different but equivalent way.  Each boundary cycle $\omega_j$ determines a map 
\[c_j:S^1\to  |\Gamma_\alpha|\]
well defined up to homeomorphism.  Informally, this map is given by tracing the boundary cycle $\omega_j$ along $|\Gamma_\alpha|$.  More precisely, let $\omega_j=(l_j,\bar{l}_j, h_1,h_2,\ldots, h_{n_j})$, where $\{l_j,\bar{l}_j\}$ is the leave of this boundary cycle.  Let $e_i$ denote the edge corresponding to the half edge $h_i$.  Subdivide $S^1$ into $n_j$ consecutive oriented intervals $I_i$, for $1\leq i \leq n_j$, such that their orientation coincides with the orientation of $S^1$.  Let $0_i$ denote the start point of $I_i$.  Then $c_j:S^1\to  |\Gamma_\alpha|$ is the map that sends $I_i\to |e_i|$ via a homeomorphism such that $c_j(0_i)=s(h_i)$.  The fat graph $\Gamma_\alpha$ is admissible if and only if all $c_j$'s are injective and their images are disjoint.

We will show that condition $i$ is equivalent to saying that all $c_j$'s have disjoint images and conditions $ii$ and $iii$ are equivalent to saying that each $c_j$ is injective. 

Note that condition $i$ does not hold for $\alpha$ if and only if for some $i_j$ and $i_s$ with $j \neq s$ at least one of the following hold  
\begin{itemize}
\item[$(a)$] There is an arc in $\alpha$, say $\alpha_{j,s}$, connecting $\partial_{i_j}S$ and $\partial_{i_s}S$.
\item[$(b)$] There is a component in $\Sg-\alpha$, say $T_{j,s}$, which has as part of its boundary: an arc $\alpha_j$ with an endpoint at $\partial_{i_j}S$ and an arc $\alpha_s$ with an endpoint at $\partial_{i_s}S$.  Let $v_{j,s}$ denote a marked point in the interior $T_{j,s}$.
\end{itemize}
If $a$ holds, then $\Gamma_\alpha$ must have an edge $e_{j,s}$ constructed by crossing $\alpha_{j,s}$.  The edge $e_{j,s}$, belongs to the $i_j$-th and $i_s$-th boundary cycles i.e., $c_j$ and $c_s$ intersect at the edge $e_{j,s}$. If $b$ holds, then there must be an edge $e_j$ (respectively $e_s$) constructed by crossing the boundary of $T_{j,s}$ at $\alpha_j$ (respectively $\alpha_s$) and connecting to $v_{j,s}$.  Moreover, the edges $e_j$ (resp. $e_s$) belongs to the $i_j$-th (resp. $i_s$-th) boundary cycles. Thus, $c_j$ and $c_s$ intersect at the vertex $v_{j,s}$. Finally, notice that if two outgoing boundary cycles of $\Gamma_\alpha$ intersect at an edge (respectively at a point) then condition $a$ (respectively $b$) hold on $\alpha$. Therefore condition $i$ is equivalent to saying that the image of all $c_j$'s are disjoint.

We will show now that condition $ii$ is equivalent to saying that the map $c_j$ does not intersect itself at an edge i.e., for any edge $e$ the restriction 
\[c_j|_{e}:c_j^{-1}(e)\to |\Gamma_{\alpha}|\]
is injective.
If $ii$ does not hold, then there must be and arc $\alpha_j$ in $\alpha$ that starts and ends at $\partial_{i_j}S$.  Let $e_j$ be its corresponding edge on $\Gamma_\alpha$.  Recall that $\alpha_j$ and $e_j$ cross exactly once.  Then $\alpha_j$ starts on one side of $e_j$ crosses to the other side at the intersection point and then returns to the initial side without any additional crossing.  This means that both sides of $e_j$ belong to the same boundary cycle i.e., $c_j$ intersects itself on the edge $e_j$.  The inverse assertion follows similarly.

Assume condition $ii$ holds in $\alpha$ for some $1\leq j \leq k$.  Then $c_j$ does not intersect itself at an edge, but it could still intersect itself at a point i.e., there could be a vertex $v$ such that the restriction 
\[c_j|_{v}:c_j^{-1}(v)\to |\Gamma_{\alpha}|\]
is not injective.  
Let $\alpha_{j_1}, \alpha_{j_2} \ldots \alpha_{j_r}$ be the arcs in $\alpha$ with an endpoint in $\partial_{i_j}S$.  Since condition $ii$ holds, each of these arcs must have their other end point in a different boundary component. The orientation of the surface together with the marked point on $\partial_{i_j}S$ give an ordering of the arcs.  Assume that the labeling given above respects this order.  Let $A_1,A_2\ldots A_r$ denote the areas of the surface between these arcs in that given order, see Figure \ref{caseii} below. 
\begin{figure}[h]
  \centering
    \includegraphics[scale=0.3]{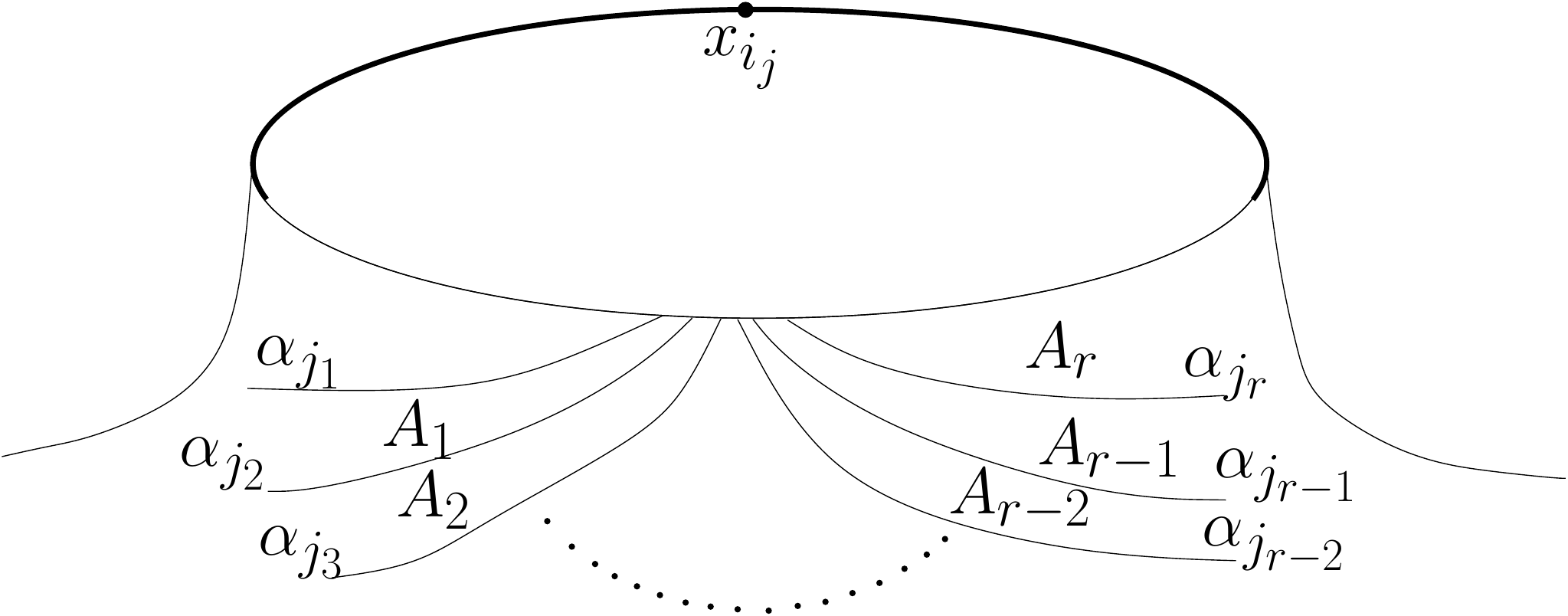}
  \caption{Local picture at the outgoing closed boundary $\partial_{i_j}S$ for an arc system $\alpha$ in which condition $ii$ of Definition \ref{arc_admissible} holds.}
  \label{caseii}
\end{figure}
Let $\Xi$ be the smallest subset of $\alpha$ such that: ${\alpha_{j_1}, \alpha_{j_2} \ldots \alpha_{j_r}}\subset\Xi\subset\alpha$ and $\Lambda:=\Xi\cap(S-\partial_{i_j}S)$ 
has a minimum number of connected components.  If condition $iii$ holds then $\Lambda$ is connected.  The areas $A_s$ and $A_l$ for $1\leq s\neq l \leq r$ belong to the same connected component in $\Sg-\alpha$, if and only if there is a path in $\Sg$ connecting them that does not intersect with $\alpha$, and this happens if and only if $\Lambda$ is not connected.  Therefore, if $iii$ holds each area $A_s$ contains a different vertex $v_s$ of $\Gamma_\alpha$.  
Let $e_s$ denote the edge in $\Gamma_\alpha$ that crosses $\alpha_{j_s}$, see Figure \ref{caseii_admissible}.  Then the $i_j$-th boundary corresponds to one side of the edges $e_s$ for $1\leq s\leq r$ i.e., $c_j$ is injective. 
\begin{figure}
  \centering
    \includegraphics[scale=0.3]{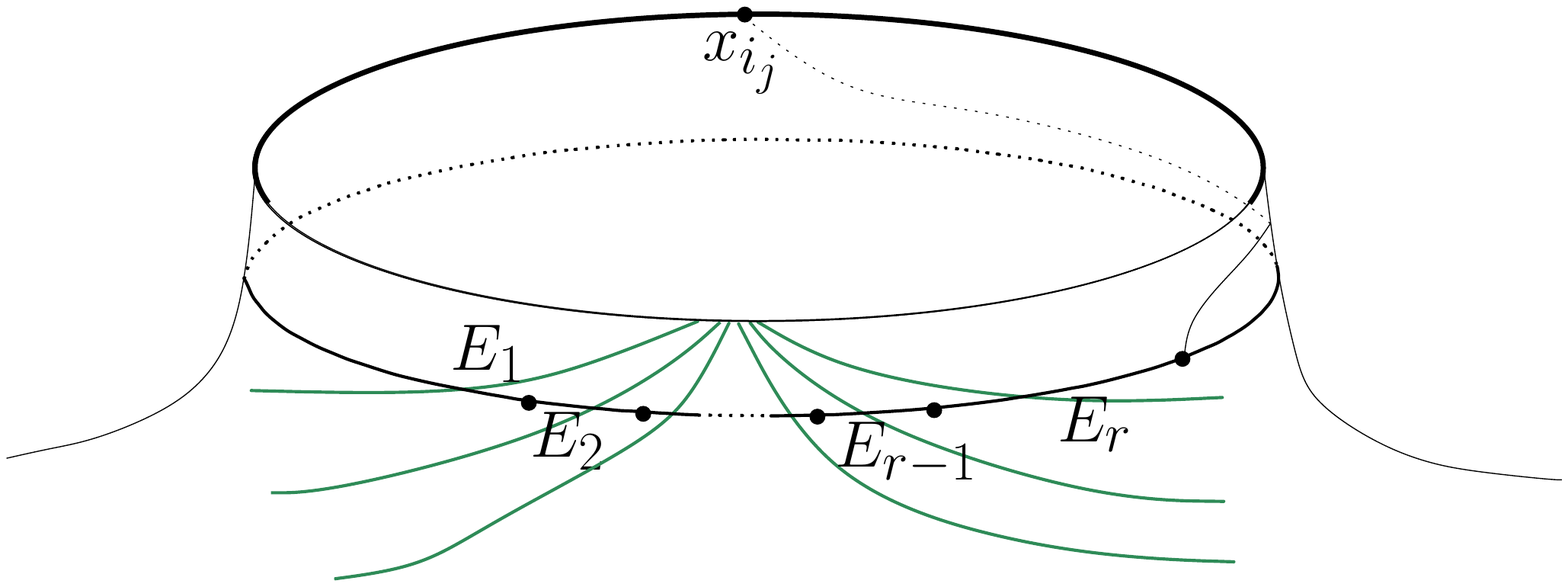}
  \caption{Local picture at the outgoing closed boundary $\partial_{i_j}S$ for an arc system $\alpha$ in which conditions $ii$ and $iii$ of Definition \ref{arc_admissible} hold.   The arcs are shown in green and their corresponding edges in black.}
  \label{caseii_admissible}
\end{figure}

If condition $iii$ does not hold then $\Lambda$ is not connected.  Assume for simplicity first that $\Lambda$ has two connected components. Then $\Xi$ must fall in one of the two following cases
\begin{itemize}
\item[$(a)$]  The two components of $\Lambda$ are next to each other in $\Sg$ i.e., there is an $t$ such that ${\alpha_{j_1}, \alpha_{j_2} \ldots \alpha_{j_t}}$ belong to one component and ${\alpha_{j_{t+1}}, \alpha_{j_{t+2}} \ldots \alpha_{j_r}}$ to the other (see Figure \ref{not_ad_next}.) Then by the argument above each $A_s$ for $1\leq s \leq t-1$ or $t+1\leq s\leq r-1$ contains a different vertex $v_s$ of $\Gamma_\alpha$.  However, $A_t$ and $A_r$ belong to the same connected component in $\Sg-\alpha$ so they both contain only one vertex, say $v$, of $\Gamma_\alpha$ which is connected to $\partial_{i_j}S$.  As before, the $i_j$-th boundary corresponds to one side of the edges $e_s$ for $1\leq s\leq r$ but these edges intersect at the point $v$.

\item[$(b)$] The two components of $\Lambda$ are nested in $\Sg$ i.e., there are $t < l$ such that 
$ \alpha_{j_1}, \alpha_{j_2} \ldots  $ 
$ \alpha_{j_t-1},\alpha_{j_{l+1}}, \alpha_{j_{l+2}} \ldots \alpha_{j_r} $ 
belong to one component and 
${\alpha_{j_{t}}, \alpha_{j_{t+1}} \ldots \alpha_{j_l}}$ 
to the other (see Figure \ref{not_ad_nested}).  Then similarly, each $A_s$ contains a different vertex $v_s$ of $\Gamma_\alpha$ except for $s=t-1$ and $s=l$, since $A_{t-1}$ and $A_l$ belong to the same connected component in $\Sg-\alpha$.  So they both contain only one vertex say $v$ of $\Gamma_\alpha$.  Then as before, the $i_j$-th boundary cycle intersects itself at $v$.
\end{itemize}
\begin{figure}
  \centering
    \includegraphics[scale=0.3]{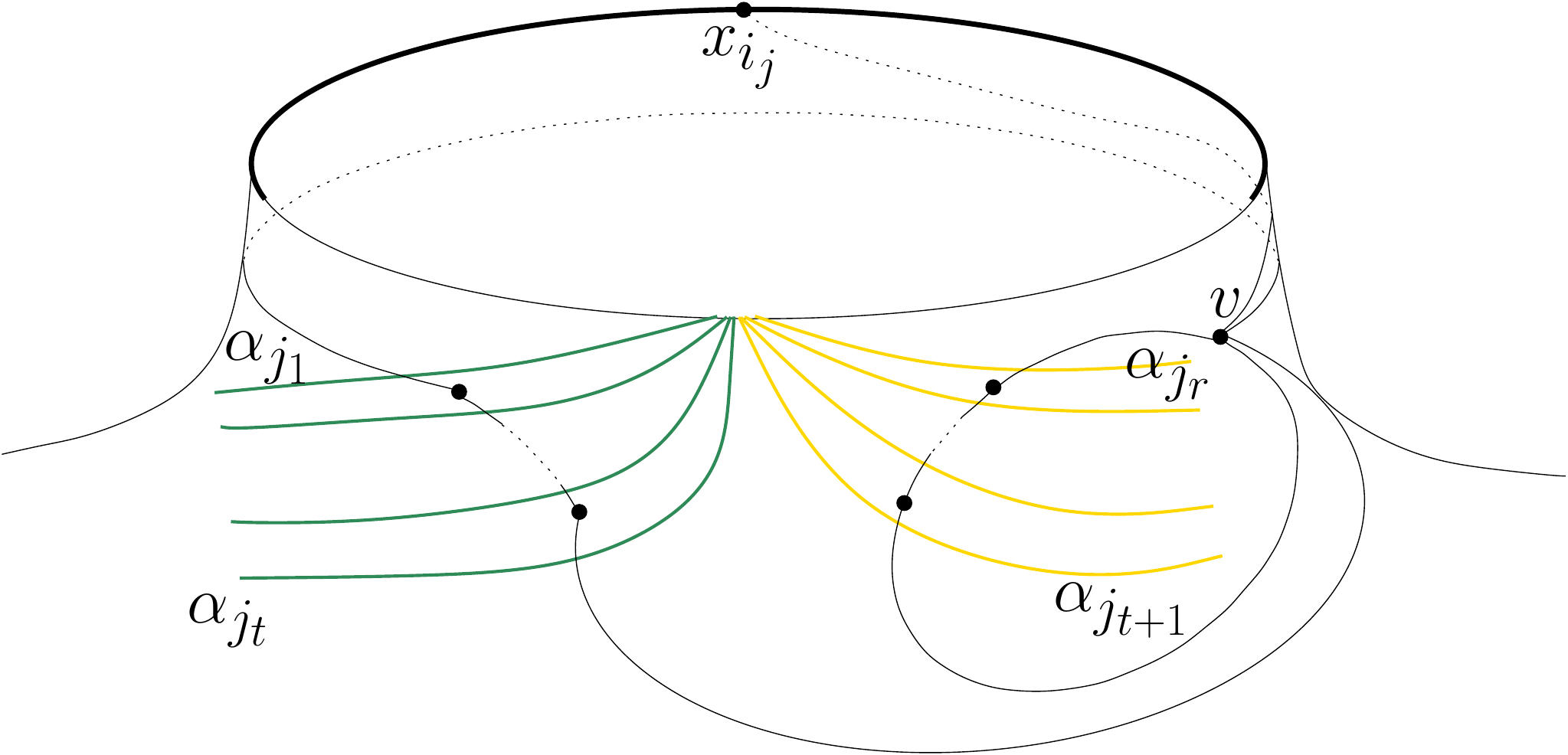}
  \caption{Local picture at the outgoing closed boundary $\partial_{i_j}S$ for an arc system $\alpha$ in which conditions $ii$ holds but condition $iii$ does not.  The arcs are shown in green and yellow to distinguish the connected components they belong to in $\Xi$.  The picture represents case $a$ in which the components are next to each other. The edges corresponding to the arcs are shown in black.}
  \label{not_ad_next}
\end{figure}
\begin{figure}
  \centering
    \includegraphics[scale=0.3]{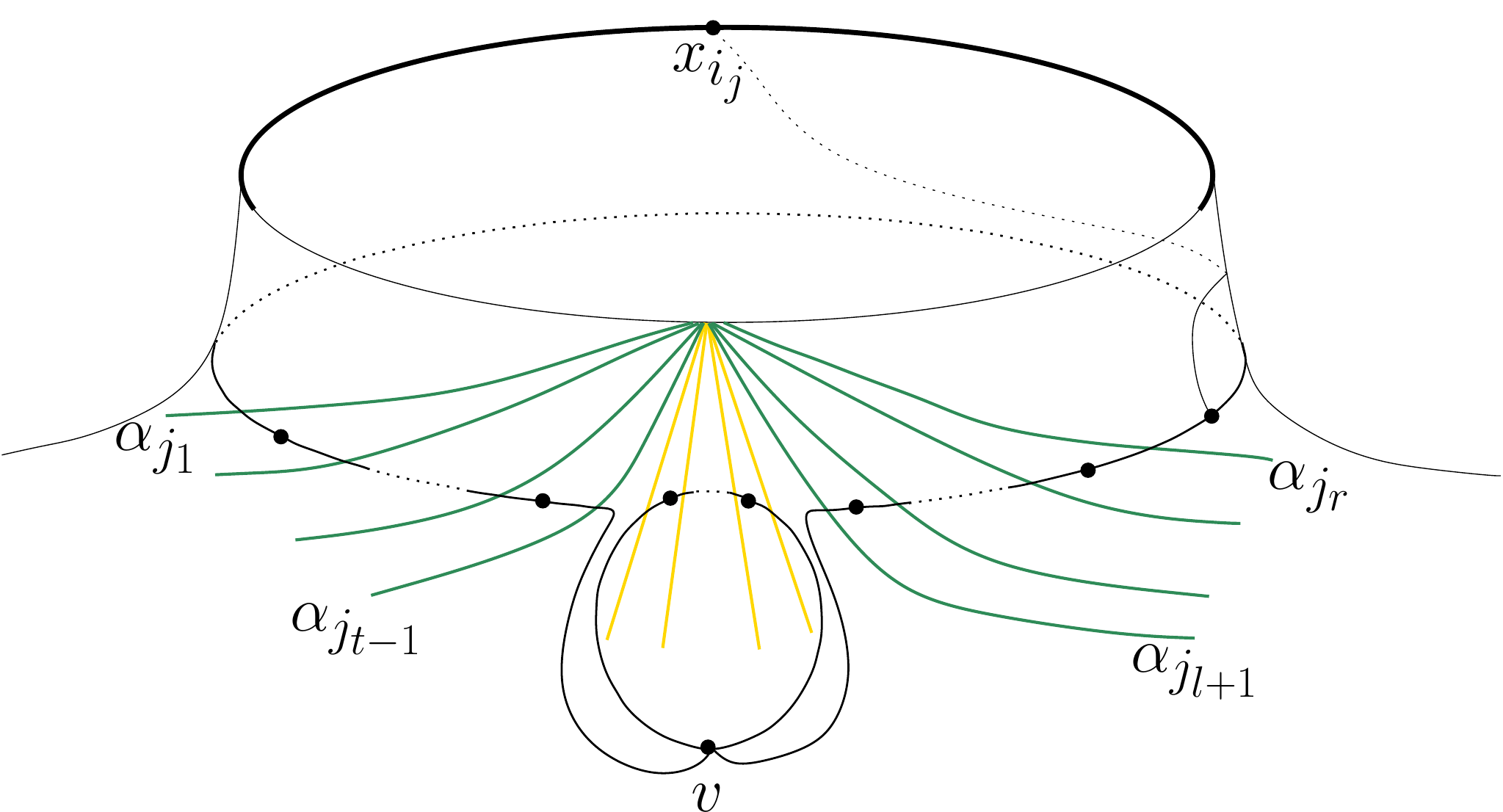}
  \caption{Local picture at the outgoing closed boundary $\partial_{i_j}S$ for an arc system $\alpha$ in which conditions $ii$ holds but condition $iii$ does not.  The arcs are shown in green and yellow to distinguish the connected components they belong to in $\Xi$.  The picture represents case $b$ in which the components are nested. The edges corresponding to the arcs are shown in black.}
  \label{not_ad_nested}
\end{figure}
The case for more connected components is a combination of these two cases giving that the $i_j$-th boundary cycle intersects itself in multiple points. Therefore conditions $ii$ and $iii$ together are equivalent to saying that the map $c_j$ is injective, which finishes the proof.
\end{proof}

\begin{prop}
\label{B0_con}
If $\Sg$ is an open-closed cobordism which is not a disk, and whose boundary is not completely outgoing closed, then the complex $\Bcat(\Sg,\Delta)$ is contractible.
\end{prop}
 
 Before proving the proposition we will state a corollary.
\begin{cor}
If $\Sg$ is an open-closed cobordism whose boundary is not completely outgoing closed, then poset category $\Bcat_0(\Sg,\Delta)$ is contractible.
\end{cor}
\begin{proof}
For the case of one boundary component, the contractibility of $\Bcat_0(\Sg,\Delta)$ follows immediately, since this just reduces to the case of $\Acat_0(\Sg,\Delta)$.  The general case, follows by induction just as in the proof of Proposition \ref{A0_con} from the contractibility of $\Bcat(\Sg,V)$.
\end{proof}

\begin{proof}[Proof of Proposition \ref{B0_con}]
Note first that by construction $(\Sg,\Delta)$ is never an annulus with one of its boundary components contained in $\Delta$.  Moreover, if $\Sg$ has no outgoing closed boundary components or it is an annulus with one outgoing boundary component, then $\Bcat(\Sg,\Delta)=\Acat(\Sg,\Delta)$ and the result follows by the contractibility of the arc complex.

In all other cases the result follows directly as a reduction of Hatcher's proof of the contractibility of the arc complex given in \cite{hatcher_arc}, so we just give a sketch of this proof.  We consider first the case where $\Sg$ has at most one marked point in each boundary component.  Hatcher writes a flow of the arc complex $\Acat(S,\Delta)$ onto the star of a vertex.  We will sketch the construction of this flow and see that it restricts to  $\Bcat(\Sg,\Delta)$ if one chooses the vertex correctly; which finishes the proof in this special case since the closure of the star of a vertex is contractible.  Given that not all the boundary is outgoing closed and $\Sg$ is not a disk or an annulus, we can find an essential arc $\beta$ that starts and ends at a boundary components which are not outgoing closed (possibly the same one).  Hatcher's construction gives a continuous flow 
\[\Bcat(\Sg,\Delta)\to\overline{Star([\beta])}\]
In order to construct this flow let $\tilde{\sigma}_1=[\alpha]$ be a $k-$simplex of $\Bcat(\Sg,\Delta)$ and choose a representatives $\lbrace\alpha_0\ldots\alpha_k\rbrace$ with minimal intersection with $\beta$.  Let $x_1,\ldots,x_l$ denote the intersection points of $\alpha$ and $\beta$ occurring in that order.  The first intersection point $x_1$ corresponds to an arc $\alpha_i$.  Let $\tilde{\alpha}_{i_1}$ and $\tilde{\alpha}_{i_2}$ be the arcs obtained by sliding $\alpha_i$ along $\beta$ all the way to the boundary of $\beta$, see figure \ref{arc_flow}.
\begin{figure}[h!]
  \centering
    \includegraphics[scale=0.4]{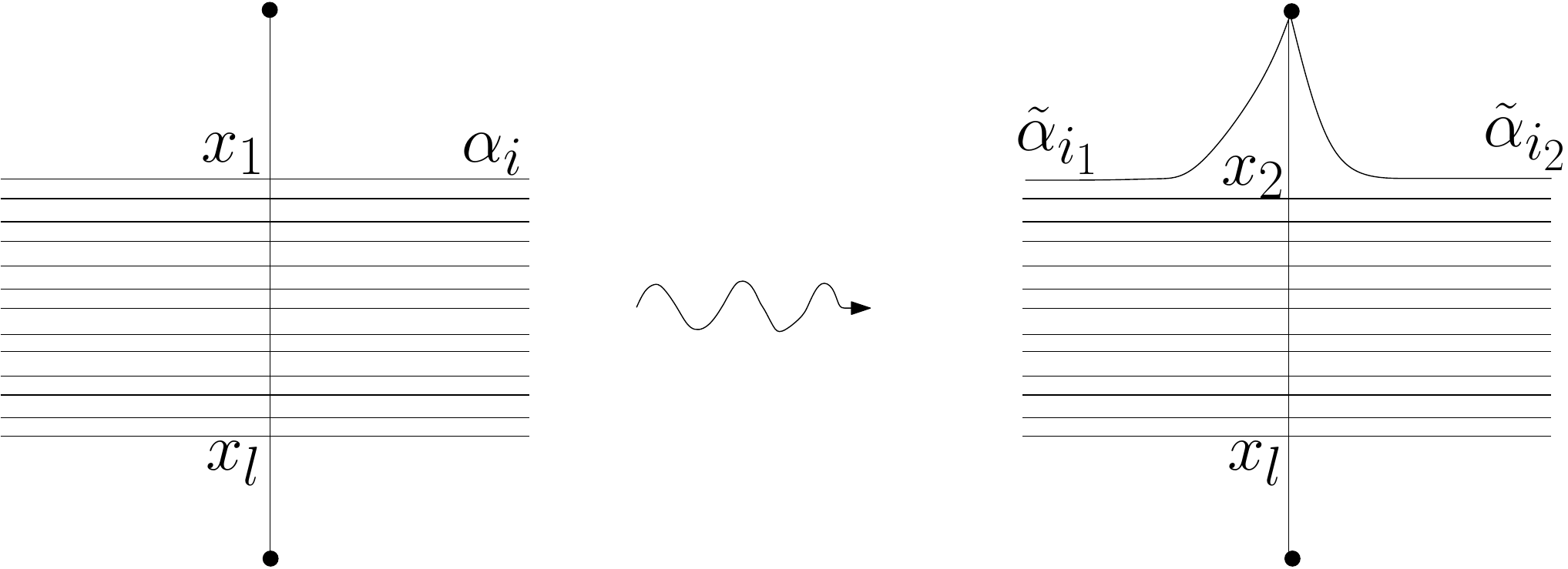}
  \caption{The arcs $\tilde{\alpha}_{i_1}$ and $\tilde{\alpha}_{i_2}$ obtained by sliding $\alpha_i$ along $\beta$}
  \label{arc_flow}
\end{figure}

Define $\sigma_2$ to be the simplex given by $\tilde{\sigma_1}\cup\tilde{\alpha}_{i_1}\cup\tilde{\alpha}_{i_2}$ and $\tilde{\sigma}_2$ to be the simplex given by replacing $\alpha_i$ with $\tilde{\alpha}_{i_1}\cup\tilde{\alpha}_{i_2}$ in $\tilde{\sigma}_1$.  If one of these new arcs is boundary parallel we just discard it, but notice that since there is at most one marked point in each boundary component then at least one of these two arcs is not boundary parallel.  Since $\beta$ doesn't intersect with any outgoing closed boundary component, we can see that this construction preserves conditions $i$-$iii$ of definition \ref{arc_admissible} i.e., $\sigma_2$ and $\tilde{\sigma}_2$ are admissible arc sets.  Furthermore, $\sigma_2$ contains $\tilde{\sigma}_1$ and $\tilde{\sigma}_2$ as faces and this last simplex intersects $\beta$ only at $x_2,\ldots,x_l$.  In this way we can define a sequence of simplices in $\Bcat(\Sg,\Delta)$
\begin{equation*}
\begin{tikzpicture}[scale=0.5]
\node (t1) at (0,0){$\tilde{\sigma_1}$};
\node (2) at (1,2) {$\sigma_2$};
\node (t2) at (2,0) {$\tilde{\sigma_2}$};
\node (3) at (3,2){$\sigma_3$};
\node (t3) at (4,0) {$\tilde{\sigma_3}$};
\node (tl) at (8,0) {$\tilde{\sigma_l}$};
\node (l) at (9,2) {$\sigma_l$};
\node (tl1) at (10,0) {$\tilde{\sigma}_{l+1}$};
\node (dot) at (6,1) {$\cdots\cdots$};

\path[auto,arrow,->]     (t1) edge [right hook->](2)
    		             (t2) edge [left hook->](2)		             
    		             (t2) edge [right hook->](3)
    		             (t3) edge [left hook->](3)
    		             (tl) edge [right hook->](l)
    		             (tl1) edge [left hook->](l);
\end{tikzpicture}
\end{equation*} 
where $\sigma_j$ contains $\tilde{\sigma}_{j-1}$ and $\tilde{\sigma}_j$ as faces and this last simplex intersects $\beta$ only at $x_j,\ldots,x_l$.  Finally, by construction $\tilde{\sigma}_{l+1}$ is in the closure of the star of $\beta$.  Thus, we can define a flow $\Bcat(\Sg,\Delta)\times I\to\Bcat(\Sg,\Delta)$ by use of barycentric coordinates which flows linearly along this finite sequence of simplices and when restricted to a face corresponds to the flow of the face.  Moreover, we can also show this flow is well defined on arc systems.  This finishes the proof in the special case.  

Now, to consider the case where there is a boundary component with more than one marked point. It is enough to consider what happens when we add a marked point $p\in\partial S-\Upsilon$.  Let $\Upsilon'=\Upsilon\cup\{p\}$ and $\Delta'=\Delta\cup I_p$ where $I_p$ is an interval around $p$ such that $I_p \cap \Delta =\emptyset$. This additional marked point $p$ can not be added to an outgoing closed boundary.  By using a similar argument as for the case with at most one marked point in the boundary component we can show that if $\Bcat(\Sg,\Delta)$ is $n$ connected then $\Bcat(\Sg,\Delta')$ is $n+1$ connected.  Wahl describes this argument in detail in \cite{wahlmcg} and we can see that her argument restricts to $\Bcat(\Sg,\Delta)$ in a similar way as for the special case.
\end{proof}

\subsection{Gluing admissible fat graphs}
\label{gluing_ad_section}
We want to think of open-closed cobordisms as morphisms between one dimensional manifolds, where composition is given by gluing along the boundary using the parametrizations. Furthermore, we want to model composition of cobordisms by a map described combinatorially in terms of fat graphs.  More precisely, consider $S_1$ and $S_2$ composable open-closed cobordisms i.e., $\partial_{out}S_1 \cong \partial_{in}S_2$ via an orientation reserving diffeomorphism.  Then we can glue $S_1$ to $S_2$ along $\partial_{out}S_1$ and $\partial_{in}S_2$ using the boundary parametrizations to obtain an oriented surface $S_2 \circ S_1$ together with a map 
\[S_1\sqcup S_2 \longrightarrow S_2 \circ S_1\]
which is injective everywhere except on $\partial_{out}S_1$ and $\partial_{in}S_2$. This induces a map 
\begin{equation}
\label{gluing_Mod}
\begin{array}{ccc}
\Mod(S_2)\times\Mod(S_1)          &\longrightarrow              &  \Mod(S_2\circ S_1)\\
(\varphi_2,\varphi_1)                 &\mapsto                         &  \varphi_2\circ\varphi_1  
\end{array}
\end{equation}
where $\varphi_2\circ\varphi_1$ is the diffeomorphism which restricts to $\varphi_i$ on the image of $S_i\cof S_2 \circ S_1$ for $i=1,2$.  The diffeomorphism  $\varphi_2\circ\varphi_1$ is well defined since $\varphi_1$ and $\varphi_2$ fix $\partial_{out}S_1$ and $\partial_{in}S_2$ and their collars point-wise.  Taking classifying spaces we get a continuous map 
\begin{equation}\
\label{gluing_BMod}
\text{B} \Mod(S_2) \times \text{B} \Mod(S_1) \longrightarrow \text{B} \Mod(S_2 \circ S_1)
\end{equation}
In this section we construct a map on admissible fat graphs which models this map up to homotopy.  This will be a topological map, constructed on the realization of the categories of admissible fat graphs.  First, we will introduce the notion of a metric fat graph to give a different interpretation of the elements of the realization of these categories. 

\begin{dfn}
A \emph{metric admissible fat graph} is a pair $(\Gamma,\lambda)$ where $\Gamma$ is an admissible fat graph and $\lambda$ is a \emph{length function}, i.e., a function $\lambda: E_{\Gamma}\to [0,1]$ where $E_{\Gamma}$ is the set of edges of $\Gamma$ and $\lambda$ satisfies:
\begin{enumerate}[(i)]
\item $\lambda(e)=1$ if $e$ is a leaf,
\item $\lambda^{-1}(0)$ is a forest in $\Gamma$ and $\Gamma/\lambda^{-1}(0)$ is admissible
\item for any admissible cycle $C$ in $\Gamma$ we have $\sum_{e\in C}\lambda(e) = 1$.
\end{enumerate}
We will call the value of $\lambda$ on $e$ the \emph{length} of the edge $e$ in $\Gamma$.
\end{dfn}

\begin{dfn}
Two metric admissible fat graphs $(\Gamma,\lambda)$ and $(\tilde{\Gamma}, \tilde{\lambda})$ are called \emph{isomorphic} if there is an isomorphism of admissible fat graphs $\varphi:\Gamma\to\tilde{\Gamma}$ such that $\lambda=\tilde{\lambda}\circ\varphi_*$, where $\varphi_*$ is the map induced by $\varphi$ on $E_{\Gamma}$.  We denote by $[\Gamma, \lambda]$ an isomorphism class of metric admissible fat graphs.
\end{dfn}

In other words, (i) we identify isomorphic admissible fat graphs with the same metric and (ii) we identify a metric admissible fat graph with some edges of length $0$ with the metric fat graph in which these edges are collapsed and all other edge lengths remain unchanged.

\begin{rmk}
\label{metric_nerve}
The elements of $\vert\Fatad\vert$ can be interpreted as (isomorphism classes of) metric admissible fat graph as follows.  Each point in the realization is given by $x=([\Gamma_0]\to[\Gamma_1]\to\ldots\to[\Gamma_k],s_0,s_1,\ldots s_k)\in N_k\Fatad\times\Delta^k$, where $N_k$ denotes the set of $k$-simplices of the nerve.  Choose representatives $\Gamma_i$ for $0\leq i\leq k$ and for each $i$, let $C^j_i$ denote the $j$th admissible cycle of $\Gamma_i$ and $n^j_i$ denote the number of edges in $C^j_i$. Each graph $\Gamma_i$ naturally defines a metric admissible fat graph $(\Gamma_0,\lambda_i)$ where $\lambda_i$ is given as follows:
\[
\begin{array}{rcl}
\lambda_i:E_{\Gamma_0} & \longrightarrow & \hspace{7pt} [0,1]\\
\hspace{7pt}e&\longmapsto & 
\left\lbrace\begin{array}{cl}
0 & \text{if } e \text{ is collapsed in }\Gamma_i\\
{1}/{n^i_j} & \text{if } e\in C^i_j \\
{1}& \text{otherwise}
\end{array}\right.
\end{array}
\]
Then $x$ determines an isomorphism class of metric fat graphs $[\Gamma_0,\sum_{i=0}^k s_i\lambda_i]$.  It is easy to see that this assignment respects the simplicial identities and is injective.  Thus we can represent elements of $|\Fatad|$ uniquely by their corresponding metric fat graphs.
\end{rmk}

\begin{cons}
\label{gluing_ad}
Given $S_1$ and $S_2$ composable cobordisms, we construct a map 
\begin{equation}
\label{gluing_ad_map}
\begin{array}{ccc}
|\Fatad_{S_2}| \times |\Fatad_{S_1}| &\longrightarrow &|\Fatad_{S_2 \circ S_1}|\\
([\Gamma_2,\lambda_2],[\Gamma_1,\lambda_1]) &\mapsto &[\Gamma_2 \circ \Gamma_1, \lambda_2 \circ \lambda_1]
\end{array}
\end{equation}
Choose representatives $(\Gamma_1,\lambda_1)$ and $(\Gamma_2,\lambda_2)$ such that there are no edges of length zero. We first fix some notation.
Let 
$l_1^1,l_1^2,\ldots ,l_1^{q_1}$
be the outgoing closed leaves of $\Gamma_1$ and 
$l_1^{q_1+1},l_1^{q_1+2},\ldots$,
$l_1^{q_1+q_2}$
be the outgoing open leaves of $\Gamma_1$.
Similarly, let 
$l_2^1,l_2^2,\ldots ,l_2^{q_1}$
be the incoming closed leaves of $\Gamma_2$ and 
$l_2^{q_1+1},l_2^{q_1+2},\ldots ,l_2^{q_1+q_2}$
be the incoming open leaves of $\Gamma_2$.
Moreover, let
$\Gamma_1^i\subset \Gamma_1$
be the sub-graph corresponding to $l_1^i$, the $i$th outgoing closed leave of $\Gamma_1$, 
and similarly let
$\Gamma_2^i\subset \Gamma_2$
be the sub-graph corresponding to $l_2^i$, the $i$th incoming closed leave of $\Gamma_2$.
Finally, define $B_i$ to be the total length of the boundary cycle corresponding to $l_2^i$ i.e., $B_i=\sum_{e\in \Gamma_2^i} m_e\cdot\lambda_2(e)$, where $m_e$ is the number of times $e$ appears in the boundary cycle of  $l_2^i$.  Note that $m_e\in \{1,2\}$.

Since $\Gamma_1$ is an admissible fat graph, all the sub-graphs $\Gamma_1^i$'s are disjoint and thus we can re-scale $(\Gamma_1,\lambda_1)$ to $(\Gamma_1,\tilde{\lambda}_1)$ where: 
\[
 \tilde{\lambda}_1(e) =
  \begin{cases}
   B_i \cdot \lambda_1(e) & \text{if } e\in \Gamma_1^i \\
   \lambda_1(e)          & \text{else} 
  \end{cases}
\]
In other words, we independently re-scale the sub-graphs corresponding to the outgoing closed leaves such that the total lengths of each incoming boundary cycle of $\Gamma_2$ equals the total lengths of its corresponding outgoing cycle in $\Gamma_1$.  Now we can define a metric fat graph $(\Gamma_2 \underset{c}{\circ} \Gamma_1, \lambda_2 \underset{c}{\circ} \lambda_1)$ obtained by gluing $\Gamma_1$ and $\Gamma_2$ along their closed boundary components.  More precisely, this is the metric fat graph obtained by:
\begin{itemize}
\item[$(i)$] Collapsing the leaves $l_1^i$ and $l_2^i$ to vertices $v_1^i$ and $v_2^i$.
\item[$(ii)$] gluing each $\Gamma_2^i \subset \Gamma_2$ to $\Gamma_1$ along the boundary cycle corresponding to $l_1^i$  such that $v_1^i$ and $v_2^i$ coincide in a single vertex denoted $v_i$. If $v_i$ is a bivalent vertex we delete it.
\item[$(iii)$] The metric $\lambda_2 \underset{c}{\circ} \lambda_1$ is the one induced by $\tilde{\lambda}_1$ and $\lambda_2$. 
\end{itemize} 
See Figure \ref{admissible_gluing_pic} for an example.
Finally, we define $\Gamma_2 \circ \Gamma_1$ to be the admissible fat graph obtained by gluing each open leave $l_1^{q_1+j}$ to the open leave $l_2^{q_1+j}$ for $1\leq j \leq q_2$ to obtain an edge $e_j$ and we endow this graph with the following metric
\[
 \lambda_2 \circ \lambda_1(e) =
  \begin{cases}
   1                                                    & \text{if } e=e_j \text{ for some } 1\leq j \leq q_2 \\
   \lambda_2 \underset{c}{\circ} \lambda_1(e)  & \text{else} 
  \end{cases}
\]
See Figure \ref{admissible_gluing_pic} an example.
\end{cons}

\begin{figure}[h]
  \centering
  \includegraphics[scale=0.7]{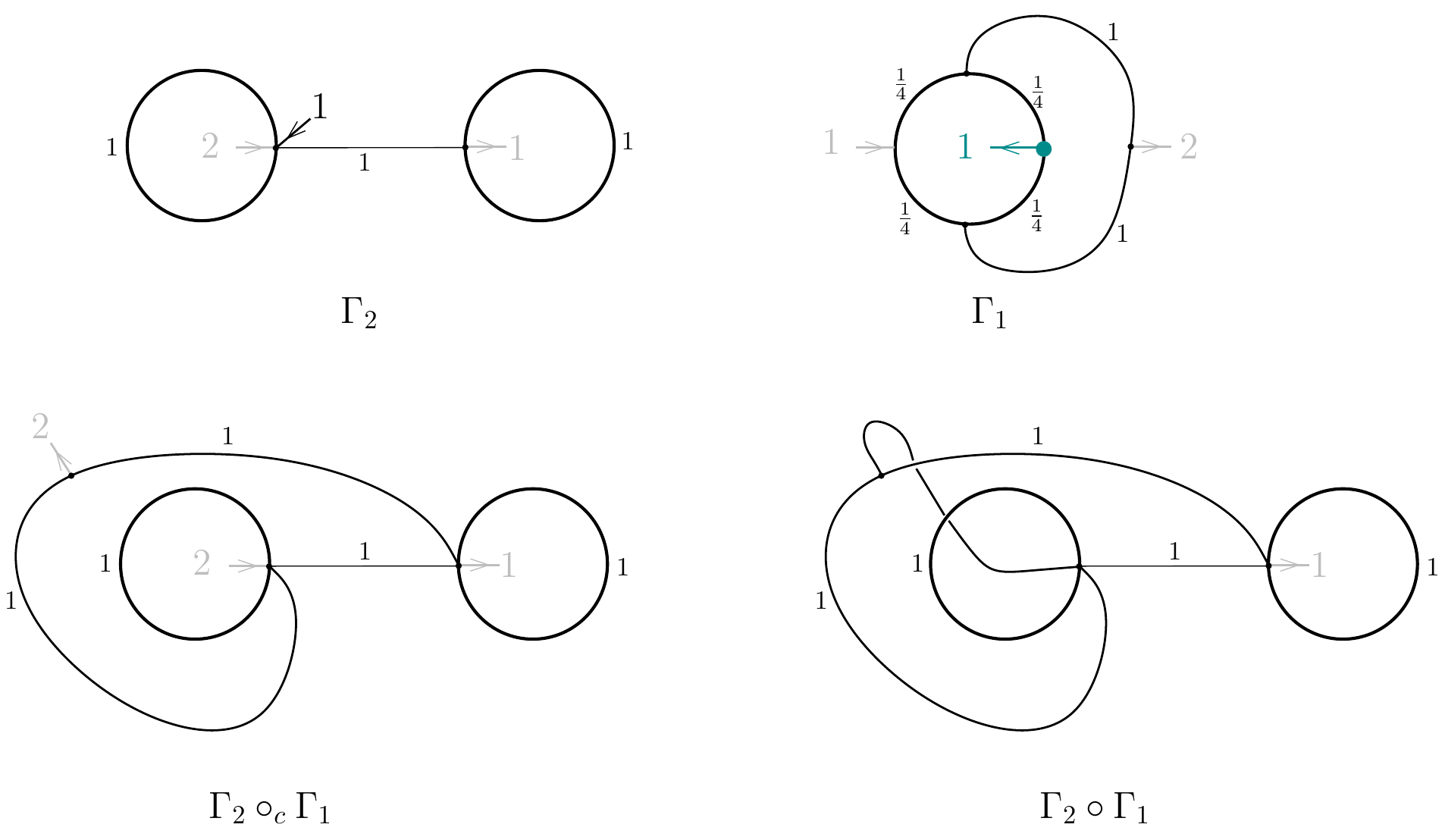}
  \caption{An example of gluing metric admissible fat graphs.  Incoming and outgoing leaves are marked with arrows.  Closed leaves are in black and open leaves in grey.  The lengths of the edges are written next to each edge.}
  \label{admissible_gluing_pic}
\end{figure}

\begin{thm}
\label{admissible_composition}
If $S_2\circ S_1$ is an oriented cobordism in which each connected component has a boundary component which is neither free nor outgoing closed, then Construction \ref{gluing_ad} models the map on classifying spaces of mapping class groups, map (\ref{gluing_BMod}), under the equivalence of Theorem \ref{ad_oc}.  
\end{thm}

The proof of this theorem will follow several steps.  The main idea is as follows.  Recall that the quotient map $\Ead_S \fib \Fatad_S$ is a model for the universal $\Mod(S)$-bundle.  Moreover, we have a bundle isomorphism 
\begin{equation*}
\begin{tikzpicture}[scale=1]
\node (a) at (0,0){$|\Fatad_S|$};
\node (b) at (3,0) {$|\underline{\Bcat}_0(S)|$};
\node (c) at (0,2) {$|\Ead_S|$};
\node (d) at (3,2){$|\Bcat_0(S)|$};
\path[auto,arrow,->]     (a) edge node[swap]{$\cong$}(b)
    		             (c) edge node {$\cong$}(d);
\path[auto,arrow,->>]    (c) edge (a)
    		             (d) edge (b);
\end{tikzpicture}
\end{equation*} 
where $\Bcat_0(S)$ is the poset category of filling admissible arc systems, $\underline{\Bcat_0}(S)$ is the quotient category under the action of the mapping class group and the isomorphisms are given by taking duals.  To proof the theorem, we will expand the construction in \cite{kaufmann_livernet_penner} to give a map 
\[|\Bcat_0(S_2)|\times |\Bcat_0(S_1)|\to|\Bcat_0(S_2\circ S_1)|\]
which descends to a well defined map 
\[|\underline{\Bcat_0}(S_2)|\times |\underline{\Bcat_0}(S_1)|\to 
|\underline{\Bcat_0}(S_2\circ S_1)|\]
which is dual to the one defined in Construction \ref{gluing_ad}.  Finally, we study what this map does on each fiber of the universal bundle to show that it models the map on classifying spaces given in (\ref{gluing_BMod}). 

\begin{rmk}
 It is important to remark that the map defined on metric graphs (\ref{gluing_ad_map}) is heavily dependent on the metric and it is thus only a topological map i.e., it is defined on the realization and can not be defined on the level of categories.  Moreover, this map is not strictly associative but only homotopy associative.
\end{rmk} 

As for the case of fat graphs, we will introduce the notion of weighted arc systems to give a specific interpretation of the elements of the realization of these categories of arcs. 

\begin{dfn}
A \emph{weighted arc system in $S$} is a pair $([\alpha],\omega)$ where $[\alpha]$ is a simplex of the arc complex $\Acat(S)$ and $\omega$ is a \emph{weight function}, i.e., a function $\omega: [\alpha]\to (0,1]$, $\alpha_i \mapsto \omega_i$.
Similarly, a \emph{weighted admissible arc system in $S$} is a weighted arc system $([\alpha],\omega)$ where $[\alpha]\in\Bcat_0(S)$ and $\omega$ is a weight function, such that for any outgoing closed boundary $\partial_jS$
\[\sum \omega(\alpha_i)=1\]
where the sum is taken over all $\alpha_i$ with an endpoint in $\partial_j S$.
\end{dfn}

\begin{rmk}
Just as in Remark \ref{metric_nerve} the elements of the realization of the poset category $\vert\Pcat(\Acat(S))\vert$ can be uniquely determined by weighted arc systems.  And each point in the the quotient under the action of the mapping class group, $\vert\underline{\Pcat(\Acat}(S))\vert$, can be uniquely determined as an equivalence class of weighted arc systems, where $\Mod(S)$ acts trivially on the weights.  
Similarly, each point of $\vert\Bcat_0(S)\vert$ can be uniquely determined by a weighted admissible arc system.  And each point in $\vert\underline{\Bcat}_0(S)\vert$ can be uniquely determined as an equivalence class of such under the action of the mapping class group. 
\end{rmk}

\begin{dfn}
An arc system $[\alpha]\in\Acat(S)$ is said to be \emph{exhaustive} if for each closed boundary component of $S$, say $\partial_jS$, there is at least one arc $\alpha_i\in[\alpha]$ with one of its endpoints in $\partial_jS$.
Let $\Arc(S)$ be the subspace of $\vert\Pcat(\Acat(S))\vert$ of weighted exhaustive arc systems and $\underline{\Arc}(S)$ be its quotient under the action of the mapping class group.
\end{dfn}

Let $\partial_i S_1$ denote the $i$-th outgoing closed boundary component of $S_1$ and $\partial_i S_2$ denote the $i$-th incoming closed boundary component of $S_2$.  Let $S_2\underset{i}{\circ}S_1$ denote the surface obtained by gluing $\partial_i S_1$ to $\partial_i S_2$ using the parametrizations.  In \cite{kaufmann_livernet_penner}, Kaufmann, Livernet and Penner construct a map
\begin{equation}
\label{KLP}
\circ_i:\Arc(S_2) \times \Arc(S_1)\to \Arc(S_2\underset{i}{\circ} S_1)
\end{equation} 
which descends to a well defined map 
\begin{equation}
\label{KLP_quot}
\circ_i:\underline{\Arc}(S_2) \times \underline{\Arc}(S_1)\to \underline{\Arc}(S_2\underset{i}{\circ} S_1)
\end{equation}
The proof that these maps are well defined and continuous is quite involved.  However, the main idea of the construction is simple and beautiful.  Here we informally describe this construction.  Consider $([\alpha],\omega)\in \Arc(S_1)$, we can interpret the weight $\omega_j$ of $\alpha_j$ as the height of a rectangular band $R^1_j$ whose core is identified with $\alpha_j$.  By perturbing the endpoints of the arcs $\alpha_j$ we can represent $([\alpha],\omega)$ by a family of bands in $S$ which intersect only at their boundaries and which form a solid band on a collar of the boundaries, see Figure \ref{arc_window}. 
The arcs incident on each outgoing closed boundary can be totally ordered using the orientation of the surface together with the marked point.  Thus the bands forming the solid band at the boundary appear in a specific order. Similarly we can geometrically interpret  $([\beta],\sigma)\in \Arc(S_2)$ as a family of rectangular bands $R^2_j$ with core $\beta_j$.
\begin{figure}[h]
  \centering
  \includegraphics[scale=0.4]{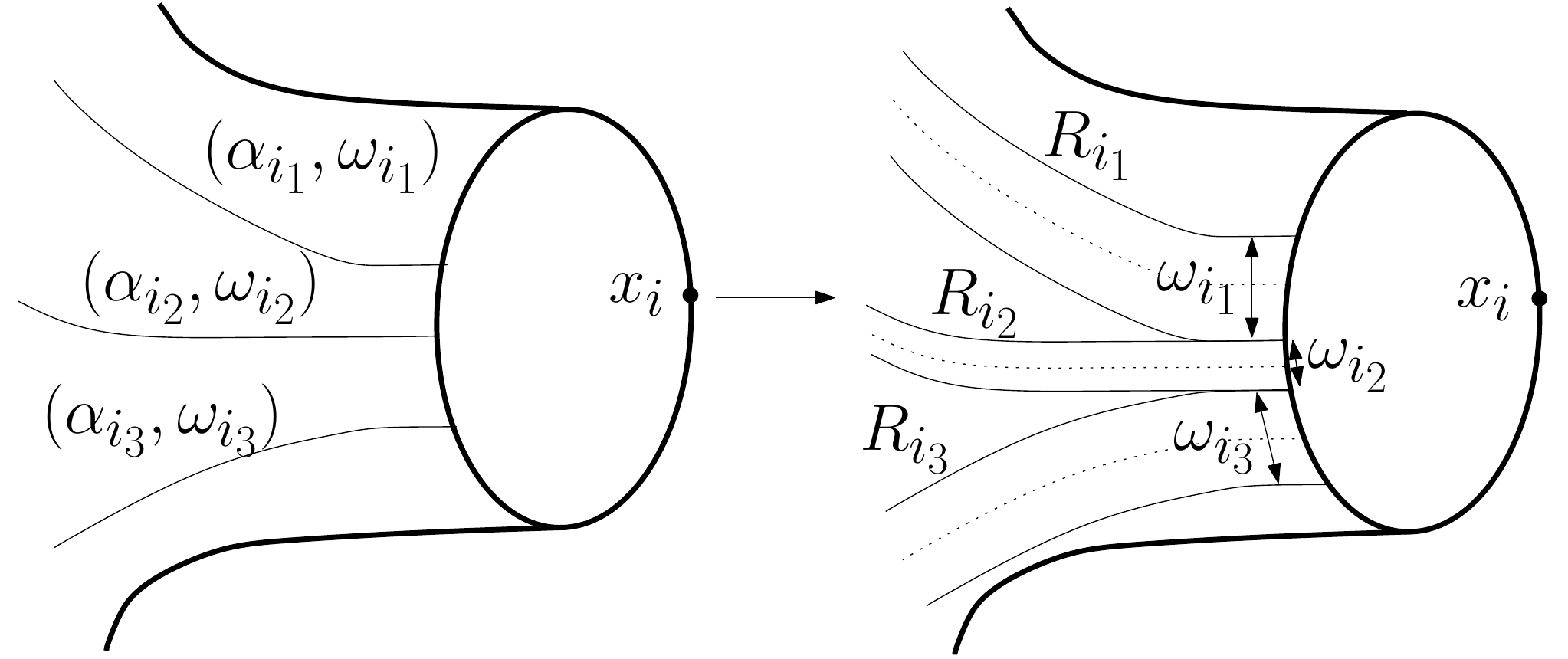}
  \caption{On the left a weighted arc system at a boundary $\partial_i S$ on the right its interpretation in terms of solid bands on the surface.}
  \label{arc_window}
\end{figure}

Let $A_i$ be the total height of the bands attached at $\partial_i S_1$ and $B_i$ be the total height of the bands attached at $\partial_i S_2$.  More precisely:
\[A_i:=\sum_{
\begin{array}{c}
j\\
\alpha_j\cap\partial_iS_1\neq\emptyset
\end{array}} m_j\cdot\omega_j
\]
\[B_i:=\sum_{
\begin{array}{c}
j\\
\beta_j\cap\partial_iS_2\neq\emptyset
\end{array}} n_j\cdot\sigma_j
\]
where $m_j$ is the number of end points the arc $\alpha_j$ has on $\partial_i S_1$ and $n_j$ is the number of end points the arc $\beta_j$ has on $\partial_i S_2$.  In other words, $A_i$ and $B_i$ are the sum of the weights of the arcs which intersect $\partial_i S_1$ and $\partial_i S_2$ respectively counted with multiplicities.

Assume first that 
\begin{equation}
\label{equal_heights}
A_i = B_i
\end{equation}
i.e., the solid bands at $\partial_i S_1$ and $\partial_i S_2$ have the same total height.  Thus, the bands $\{R^1_j\}$ at $\partial_i S_1$ can be attached to the bands $\{R^2_j\}$ at $\partial_i S_2$  to obtain a system of bands in $S_2\underset{i}{\circ}S_1$, which in turn determines a weighted arc system in $S_2\underset{i}{\circ}S_1$, see Figure \ref{arc_gluing}.  Note that the horizontal edges of  $\{R^1_j\}$ decompose the bands $\{R^2_j\}$ into sub-rectangles and vice-versa.  This decomposition depends on the heights of the bands.  Therefore, this construction depends on the weights of the arc systems, although in a very explicit manner.  This is why we only have a topological map. Finally, this construction could create simple closed curves or boundary parallel arcs.  If it does, we just discard them.

If the total height of the bands at the gluing boundaries do not match i.e.,
Equation (\ref{equal_heights}) does not hold, we first re-scale $([\alpha],\omega)$ to $([\alpha],\omega')$ such that the total heights agree.  More precisely, we set:
\[
\begin{array}{rcl}
\omega'_j & = &  
\left\lbrace
	\begin{array}{cl}
	 \frac{B_i}{A_i}\cdot\omega_j & \text{if } \alpha_j\cap\partial_i S_1\neq\emptyset\\
	\omega_j & \text{otherwise }
\end{array}\right.
\end{array}
\] 
This is possible since $[\alpha]$ and $[\beta]$ are is exhaustive and thus $A_i\neq 0 \neq B_i$. Then, the new heights agree and we combine these weighted arc systems as before.  This defines the maps (\ref{KLP}) and (\ref{KLP_quot}) constructed in \cite{kaufmann_livernet_penner}.
\begin{figure}[h]
  \centering
  \includegraphics[scale=0.6]{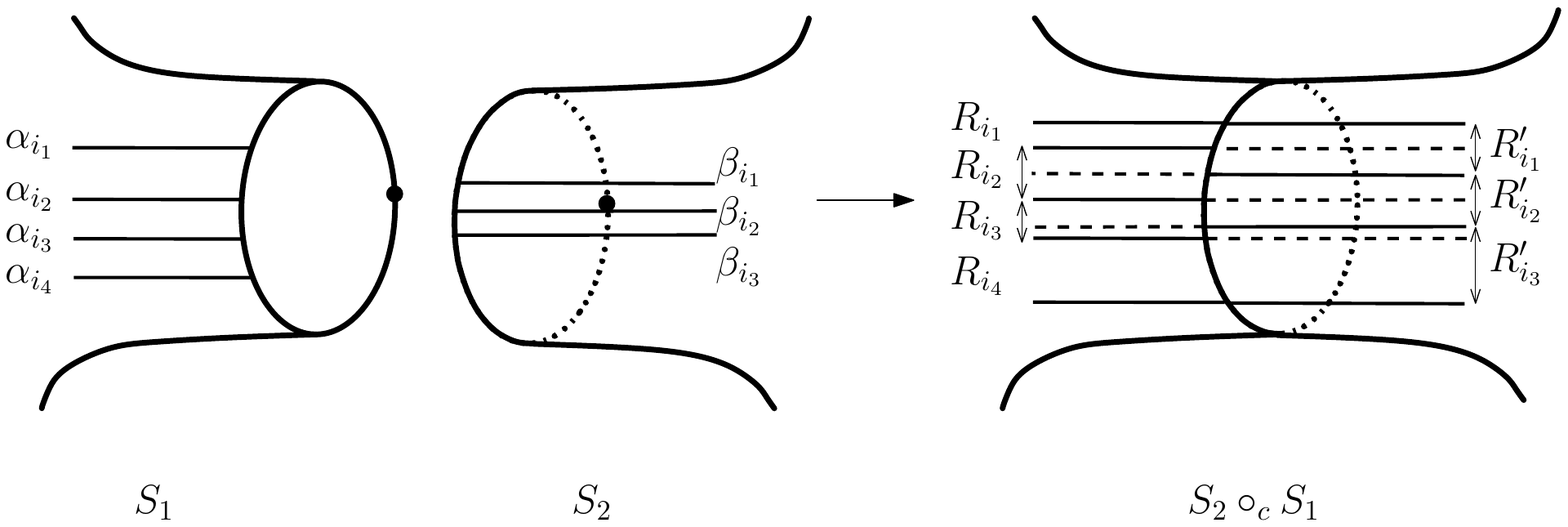}
  \caption{Local picture of gluing weighted arc systems along a closed boundary.  The dotted lines represent how the bands in $S_1$ subdivide the bands in $S_2$ and vice-versa.}
  \label{arc_gluing}
\end{figure}

We now show that if we restrict to the admissible case we can do this construction along all closed boundary components at once.
\begin{cons}
\label{admissible_arcs}
Note first that if $S$ is not a disk, every filling arc system is exhaustive and thus for any cobordism $S$ which is not a disk it holds that
\[|\Bcat_0(S)|\subset \Arc(S)\]
Let $S_2\underset{c}{\circ} S_1$ be the surface obtained by gluing the closed outgoing boundary components of $S_1$ to the closed incoming boundary components of $S_2$.  The construction of Kaufmann, Livernet and Penner extends to a map
\begin{equation*}
\begin{array}{ccl}
\underset{c}{\circ}:|\Bcat_0(S_2)| \times |\Bcat_0(S_1)| & \to  &|\Arc(S_2 \underset{c}{\circ} S_1)|\\
(([\beta],\sigma),([\alpha],\omega)) & \mapsto & ([\gamma],\theta)
:=([\beta \underset{c}{\circ}\alpha],\sigma \underset{c}{\circ} \omega)
\end{array}
\end{equation*}
where $([\gamma],\theta)$ is the weighted arc system obtained by gluing all the outgoing closed boundaries of $S_1$ to the incoming closed boundary components of $S_2$ simultaneously as in \cite{kaufmann_livernet_penner}. More precisely, for $i=1, \ldots,q_1$, let $A_i$  be the total heights of the bands at $\partial_iS_1$ and $B_i$ be the total heights of the bands at $\partial_iS_2$.  If $A_i=B_i$ for all $i=1,\ldots ,q_1$, then we can just combine the weighted arc systems as described above.  In case this does not hold, then again we can first do a re-scaling procedure on $([\alpha],\omega)$ and then glue.  To see that this is well defined, notice that since $([\alpha],\omega)$ is admissible, then the following conditions hold:
\begin{itemize}
\item[$(i)$] $A_i=1$ for all $i=1,\ldots,q_1$.
\item[$(ii)$] If $\partial_iS_1\cap\alpha_j\neq\emptyset$, then for all $\partial_k S_1$ outgoing closed such that $k\neq i$ it holds that $\partial_kS_1\cap\alpha_j=\emptyset$. 
\end{itemize}
Thus, we can re-scale $([\alpha],\omega)$ to $([\alpha],\omega')$ where $\omega'$ is defined as follows:
\[
\begin{array}{rcl}
\omega'_j & = &  
\left\lbrace
	\begin{array}{cl}
	 B_i\cdot\omega_j & \text{if } \alpha_j\cap\partial_i S_1\neq\emptyset\\
	\omega_j & \text{otherwise }
\end{array}\right.
\end{array}
\]
Then $A_i'=B_i$ for all $i=1,\ldots,q_1$ and we can glue as before.

It only remains to say what happens in the degenerate case when $S_2$ is a disk with one incoming closed boundary component,
in which case the empty arc system is the only object of $\Bcat_0(S_2)$.  In this case, we just forget the arcs incident at the outgoing closed boundary of $S_1$.  If $S_2$ has a disjoint union of disks with one incoming closed boundary, then similarly we forget the arcs incident at their corresponding outgoing closed boundaries in $S_1$.

\end{cons}

\begin{lem}
\label{gluing_closed}
The map given in Construction \ref{admissible_arcs} restricts to a map
\begin{equation*}
\underset{c}{\circ}:|\Bcat_0(S_2)| \times |\Bcat_0(S_1)| \to  |\Bcat_0(S_2 \underset{c}{\circ} S_1)|
\end{equation*}
which induces a well defined map on the quotients
\[\underset{c}{\circ}:|\underline{\Bcat}_0(S_2)| \times |\underline{\Bcat}_0(S_1)| \to  |\underline{\Bcat}_0(S_2 \underset{c}{\circ} S_1)|\]
\end{lem}
\begin{proof}
That this map is continuous and induces a well defined map on the quotients follows from the results in \cite{kaufmann_livernet_penner}.  Thus it is enough to show that
\[([\gamma],\theta):=([\beta \underset{c}{\circ}\alpha],\sigma \underset{c}{\circ} \omega)
\]
is a weighted admissible arc system. Recall that a filling arc system is admissible if it has properties $i$, $ii$ and $iii$ given in Definition \ref{arc_admissible}.

To see $([\gamma],\theta)$ has these properties, consider first the case where $S_1$ has only one outgoing closed boundary component and $S_2$ has only one incoming closed boundary component.  
Since $[\alpha]$ and $[\beta]$ have property $ii$ i.e., arcs that start at an outgoing closed boundary end elsewhere, then it is clear that $[\gamma]$ has property $ii$ as well.  In particular, notice that the gluing procedure restricted to admissible arc systems does not create any simple closed curves. Now, since $[\beta]$ has property $i$, i.e., it has arcs that cut the surface into sub-surfaces which separate the outgoing closed boundaries of $S_2$, then the arcs with endpoints at the incoming closed boundary of $S_2$ have another endpoint in at most one of the outgoing closed boundaries of $S_2$.  Therefore, $[\gamma]$ also has property $i$.  Finally, since $[\beta]$ has property $iii$, then $[\gamma]$ has property $iii$.  Now, if $S_1$ has more than one outgoing closed boundary component, since $[\alpha]$ has property $i$, then one can repeat the argument above independently for each outgoing close boundary component of $S_1$ to see that $[\gamma]$ has properties $i$, $ii$ and $iii$.

To see that $[\gamma]$ is filling, we study first the simplest case.  Assume that $S_1$ has only one outgoing boundary component.  Assume furthermore, that the weighted arc systems $([\alpha],\omega)$ and $([\beta],\sigma)$ have the same number of arcs attached at $\partial_1 S_1$ and $\partial_1 S_2$ (counted with multiplicities) and that their corresponding weights match. 
More precisely, let $\{\alpha_{1_1}, \alpha_{2_1}, \ldots \alpha_{k_1}\}$ be the arcs with endpoints in $\partial_1 S_1$ listed with multiplicities and ordered by the orientation of the surface and the marked point in $\partial_1 S_1$.  Since $[\alpha]$ is admissible, then $\alpha_{i_1}\neq \alpha_{j_1}$ if $i_1\neq j_1$.  Similarly, let $\{\beta_{1_1}, \beta_{2_1},\ldots \beta_{k_1}\}$ be the arcs with endpoints in $\partial_1 S_2$ listed with multiplicities and ordered by the orientation of the surface and the marked point in $\partial_1 S_2$.  Note that is possible that $\beta_{i_1}=\beta_{j_1}$ for some pairs $({i_1},{j_1})$.  Let $\omega_{i_1}$ be the weight of $\alpha_{i_1}$ and $\sigma_{i_1}$ be the weight of $\beta_{i_1}$.  We are assuming that $\omega_{i_1} = \sigma_{i_1}$ for all $i_1$.  We can choose representatives $\alpha$ and $\beta$ such that the end points of $\alpha_{i_1}$ and $\beta_{i_1}$ match under the parametrizations of $\partial_1 S_1$ and $\partial_1 S_2$.  Then $[\gamma]$ is the arc system in $S_2 \underset{c}{\circ} S_1$ obtained by gluing $\alpha_{i_1}$ with $\beta_{i_1}$ for all $i_1$.  See Figure \ref{easy_gluing}.  Therefore, 
\[S_2 \underset{c}{\circ} S_1-\gamma=(S_2-\beta)\bigcup_{\partial_1 S_1 \sim \partial_1 S2} (S_1-\alpha).\] 
Now, let 
\[S_1-\alpha=\sqcup T^1_i\bigsqcup \sqcup P^1_i\]
\[S_2-\beta=\sqcup T^2_i\bigsqcup \sqcup P^2_i\]
where the $T^j_i$'s and $P^j_i$'s are polygons and the $T^j_i$'s are the polygons which intersect non-trivially with $\partial_1 S_j$.  Since $\alpha$ is admissible, it has property $iii$, which implies that
\[\bigsqcup T^1_i=T^1_{1_1} \sqcup T^1_{1_1} \ldots \sqcup T^1_{k_1}\]
where $T^1_{i_1}$ is a disk with $\alpha_{i_1}$ and $\alpha_{(i+1)_1}$ on its boundary.  Finally, $S_2 \underset{c}{\circ} S_1-\gamma$ is obtained by gluing each $T^1_{i_1}$ to a $T^2_j$ along an interval in the boundary.  And these intervals are all distinct and intersect only at their endpoints. Therefore, 
\[S_2 \underset{c}{\circ} S_1-\gamma = \bigsqcup T_i\]
where the $T_i$'s are polygons i.e., $\gamma$ is filling.  

On the other hand, if the $([\alpha],\omega)$ and $([\beta],\sigma)$ do not have this structure, then the gluing procedure gives a refinement of the situation described above.  More precisely  
\[S_2 \underset{c}{\circ} S_1-\gamma=(S_2-\tilde{\beta})\bigcup_{\partial_1 S_1 \sim \partial_1 S2} (S_1-\tilde{\alpha}).\] 
where 
\[\tilde{\beta}=\beta\cup\{\tilde{\beta}_1,\ldots ,\tilde{\beta}_k\}\]
\[\tilde{\alpha}=\alpha\cup \{\tilde{\alpha}_1,\ldots, \tilde{\alpha}_l\}\]
and $\tilde{\beta}_i$ is an arc parallel to a $\beta_j\in\beta$ for some $j$
and similarly $\tilde{\alpha}_i$ is an arc parallel to a $\alpha_j\in\alpha$ for some $j$.  Then it is clear that $S_2 \underset{c}{\circ} S_1-\gamma$ is a refinement of the case above into sub-polygons and thus $\gamma$ is filling.  Furthermore, if $S_1$ has more than one outgoing closed boundary component, since $[\alpha]$ has property $i$, then one can repeat the argument above independently for each outgoing close boundary component of $S_1$ to see that $[\gamma]$ is filling.

\begin{figure}[h]
  \centering
  \includegraphics[scale=0.6]{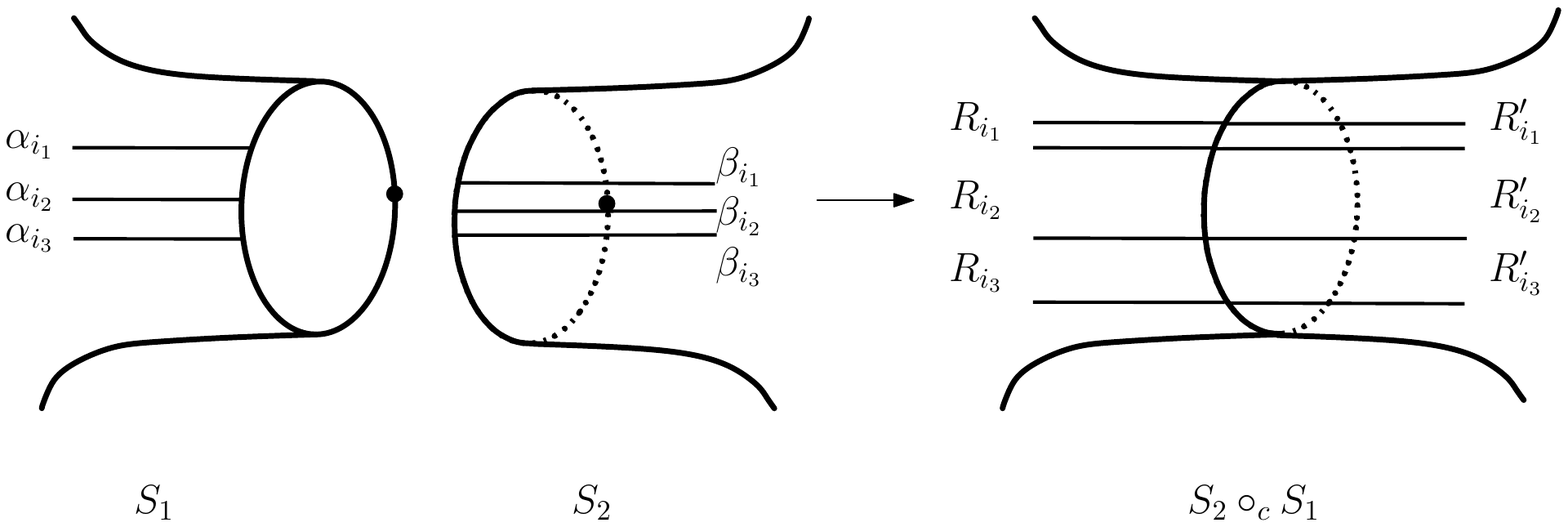}
  \caption{Gluing along closed boundary components in the case where the number of arcs and their weights match.}
  \label{easy_gluing}
\end{figure}

In the degenerate case, when $S_2$ is a disk with one incoming boundary component, let $\alpha_1, \ldots \alpha_k$ be the arcs with an endpoint at the outgoing closed boundary of $S_1$.  If we delete these arcs and then cut the surface $S_1$ along the remaining arcs, by the above argument we see that: 
\[S_1-\bigcup_{i> k} \alpha_i= A \bigsqcup \sqcup P^1_i\]
where $A$ is an annulus containing the outgoing closed boundary component of $S_1$ and $P^1_i$ are polygons.  Thus,
\[S_2\circ_{c} S_1-\bigcup_{i> k} \alpha_i= T \bigsqcup \sqcup P^1_i\]
where $T$ is the disk obtained by gluing $S_2$ to $A$.  The same argument holds if $S_2$ contains a disjoint union of disks each with one incoming closed boundary component.

Regarding the weights, since the gluing construction did not scale $([\beta],\omega)$, then it is clear that the total weight of $([\gamma],\theta)$ at each outgoing closed boundary is $1$.  Finally, the scaling of $([\alpha],\omega)$ could create arcs of weight greater than $1$.  However, since the gluing procedure subdivides the bands of $[\alpha]$ with the band of $[\beta]$ and vice-versa, there are no arcs in $([\gamma],\theta)$ with weight greater than $1$.  So $([\gamma],\theta)$ is a weighted admissible arc system.  
\end{proof}

\begin{cons}
\label{open_gluing}
We construct a map
\begin{equation*}
\underset{op}{\circ}:|\Bcat_0(S_2 \underset{c}{\circ} S_1)|\to |\Bcat_0(S_2\circ S_1)|
\end{equation*}
For $i=1,\ldots,q_2$, let $\iota^i_1:I\to S_1$ be the parametrization of the $i$th outgoing open boundary component of $S_1$ and $\iota^i_2:I\to S_2$ be the parametrization of the $i$th incoming open boundary component of $S_2$.  Let $\partial^i S_1 = \text{Im}(\iota^i_1)$ and $\partial^i S_2 = \text{Im}(\iota^i_2)$.  We can think of the $\partial^i S_j$'s as sub-spaces of $S_2 \underset{c}{\circ} S_1$.  Then the cobordism $S_2\circ S_1$ is obtained from $S_2 \underset{c}{\circ} S_1$ by gluing $\partial^i S_1$ to $\partial^i S_2$ using $\iota^i_1$ and $\iota^i_2$.  Thus, we can also think of the $\partial^i S_j$'s as arcs on $S_2\circ S_1$ and the gluing procedure gives that 
\[\partial^i S_1=\partial^i S_2\subset S_2\circ S_1\]

Now, let $([\gamma],\theta)\in |\Bcat_0(S_2 \underset{c}{\circ} S_1)|$, we define
\[\underset{op}{\circ}([\gamma],\theta):=([\gamma\cup\{\delta^1,\delta^2,\ldots,\delta^{q_2}\}],\tilde{\theta})\]
where $\delta^i=\partial^i S_1=\partial^i S_2$ and $\tilde{\theta}$ is given as follows:
\[
\tilde{\theta}:  
\left\lbrace
	\begin{array}{ccc}
	 \gamma_i & \mapsto & \theta(\gamma_i)\\
	 \delta^i & \mapsto & 1
\end{array}\right.
\]
See Figure \ref{open_gluing_example}.  This construction could create boundary parallel arcs.  If so, we just discard them.

\begin{figure}[h]
  \centering
  \includegraphics[scale=0.6]{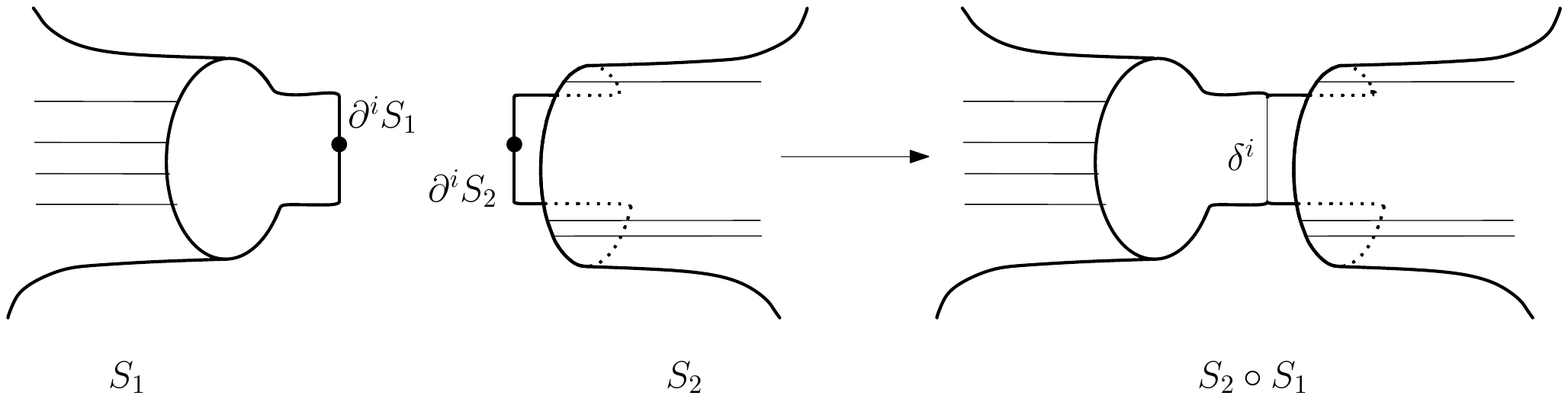}
  \caption{Local picture of gluing weighted arc systems along an open boundary.}
  \label{open_gluing_example}
\end{figure}
\end{cons}

\begin{rmk}
If $S_1$ has no outgoing closed boundary components, then 
\[S_2\underset{c}{\circ}S_1=S_2 \sqcup S_1\]
and the construction above still makes sense. On the other hand, if $S_1$ has no outgoing open boundary components, then 
\[S_2\underset{c}{\circ}S_1=S_2 \circ S_1\]
and the map $\underset{op}{\circ}$ is the identity.
\end{rmk}

\begin{lem}
The map 
\begin{equation*}
\underset{op}{\circ}:|\Bcat_0(S_2 \underset{c}{\circ} S_1)|\to |\Bcat_0(S_2\circ S_1)|\end{equation*}
described in Construction \ref{open_gluing} is well defined, continuous and induces a well defined map on quotients
\begin{equation}
\underset{op}{\circ}|\underline{\Bcat}_0(S_2 \underset{c}{\circ} S_1)|\to |\underline{\Bcat}_0(S_2\circ S_1)|.
\end{equation}
Moreover, the composition 
\begin{equation}
\label{gluing_B0quot}
\circ:|\underline{\Bcat}_0(S_2)| \times |\underline{\Bcat}_0(S_1)|
\stackrel{\underset{c}{\circ}}{\longrightarrow}
|\underline{\Bcat}_0(S_2 \underset{c}{\circ} S_1)|
\stackrel{\underset{op}{\circ}}{\longrightarrow}
|\underline{\Bcat}_0(S_2 \circ S_1)|
\end{equation}
is dual to the map (\ref{gluing_ad_map}) described in Construction \ref{gluing_ad}.
\end{lem}
\begin{proof}
If $([\gamma],\theta)\in |\Bcat_0(S_2 \underset{c}{\circ} S_1)|$, then $[\gamma]$ is admissible and in particular filling. It is clear that $\underset{op}{\circ}([\gamma],\theta)$ is filling, since we are adding arcs along all the gluing intervals. Note that the construction only adds boundary parallel arcs when gluing a disk. The admissibility properties hold trivially.  Moreover, since the mapping class group fixes the incoming and outgoing boundaries point-wise in $S_2\underset{c}{\circ}S_1$, this map induces a well defined map on the quotients.  It is easy to check that $\underset{op}{\circ}$ is continuous.

This construction is dual to the one described in Construction \ref{gluing_ad}.  To see this, consider 
$([\underline{\alpha_1}],\omega_1) \in |\underline{\Bcat}_0(S_1)|$ and $([\underline{\alpha_2}],\omega_2) \in |\underline{\Bcat}_0(S_2)|$ and let $[\Gamma_1,\lambda_1]$ respectively $[\Gamma_2,\lambda_2]$ be their dual metric admissible fat graphs.  Then $[\Gamma_2 \underset{c}{\circ} \Gamma_1]$ is obtained by scaling the closed outgoing boundary components of $\Gamma_1$ and then gluing them to the closed incoming boundary components on $\Gamma_2$.  The dual of this is scaling the arcs of $\underline{\alpha_1}$ with endpoints in the outgoing closed boundary components of $S_1$ and gluing them to the arcs with endpoints in the incoming closed boundaries of $S_2$.  Therefore 
$([\underline{\alpha_2} \underset{c}{\circ} \underline{\alpha_1} ],
\omega_2 \underset{c}{\circ} \omega_1)
\text{ is dual to }
[\Gamma_2 \underset{c}{\circ} \Gamma_1, \lambda_2 \underset{c}{\circ} \lambda_1].$
Finally, $\underset{op}{\circ}$ glues surfaces along intervals and adds an arc on each glued interval.  This is exactly dual to connecting the leaves corresponding to the open boundary components to create a new edge that crosses the glued interval.
\end{proof}

With all this at hand we can prove the main theorem.
\begin{proof}[Proof of Theorem \ref{admissible_composition}]
The map on metric fat graphs (\ref{gluing_ad_map}) is dual to the map on weighted arc systems (\ref{gluing_B0quot}), therefore it is well defined and continuous.  

To show that this construction models the map on classifying spaces of mapping class groups (\ref{gluing_BMod}) it 
is enough to study what it does on the universal bundles and show that on the fibers, it is given by the homomorphism on mapping class groups (\ref{gluing_Mod}).
Let 
\[p_i:\Bcat_0(S_i)\fib\underline{\Bcat}_0(S_i)\] 
be the $\Mod(S_i)$-universal bundles for $i=1,2$ and let 
\[p:\Bcat_0(S_2\circ S_1)\fib\underline{\Bcat}_0(S_2\circ S_1)\] 
be the $\Mod(S_2\circ S_1)$-universal bundle.  The action of the mapping class group on $\Bcat_0(S)$ does not change the weights.  Therefore, for simplicity we omit the weights from the notation.  For $i=1,2$ choose 
$\underline{\alpha_i}\in \underline{\Bcat}_0(S_i)$
and isomorphisms
\[f_i:\Mod(S_i)
\stackrel{\cong}{\longrightarrow} 
p^{-1}_i(\underline{\alpha_i})
\]
Note that for $\varphi_i\in \Mod(S_i)$ we have that $f_i(\varphi_i)=\varphi_i(f_i(1))$ i.e., it is the arc system obtained by acting with $\varphi_i$ on $f_i(1)$.  Let
\[g:\Mod(S_2 \circ S_1)
\stackrel{\cong}{\longrightarrow} 
p^{-1}(\underline{\alpha_2}\circ \underline{\alpha_1})\]
be the isomorphism given by $g(1)=f_2(1)\circ f_1(1)$.  Then it is clear that the composite
\[\Mod(S_2)\times \Mod(S_1)
\stackrel{(f_2,f_1)}{\longrightarrow}
p^{-1}_2(\underline{\alpha_2}) \times p^{-1}_1(\underline{\alpha_1})
\stackrel{\circ}{\longrightarrow}
p^{-1}(\underline{\alpha_2}\circ \underline{\alpha_1})
\stackrel{g^{-1}}{\longrightarrow} 
\Mod(S_2 \circ S_1)
\]
is the homomorphism (\ref{gluing_Mod}) since 
\[g^{-1}(\circ((f_2,f_1)(\varphi_2,\varphi_1)))=
g^{-1}(\varphi_2(f_2(1))\circ\varphi_1(f_1(1)))=
\varphi_2\circ\varphi_1\]
which finishes the proof.
\end{proof}

\section{The chain complex of black ad white graphs}
\subsection{The definition}
In \cite{costellotcft}, Costello gives a complex which models the mapping class group of open-closed cobordisms.  In \cite{wahlwesterland}, Wahl and Westerland rewrite this complex in terms of fat graphs.  In this section we describe this complex as it is defined in \cite{wahlwesterland}.

\begin{dfn}
\label{bw_generalized_def}
A \emph{generalized black and white graph} $G$ is a tuple 
\[G=(\Gamma, V_b, V_w, In, Out, Open, Closed)\] 
where $\Gamma$ is a fat graph, and $V_b$, $V_w$ are subsets of $V$ the set of vertices of $\Gamma$.  We call $V_b$ the set of \emph{black vertices} and $V_w$ the set of \emph{white vertices}.  The sets, $In$, $Out$, $Open$ and $Closed$ are subsets of $L_{\Gamma}$, the set of leafs of $\Gamma$.  In the tuple $G$ the following must hold
\begin{itemize}
\item[-] $V=V_b\sqcup V_w$ i.e., all vertices are either black or white.
\item[-] All black inner vertices are at least trivalent, white vertices are allowed to have valence $1$ or $2$.
\item[-] The subsets $Open$ and $Closed$ are disjoint.
\item[-] The subsets $In$ and $Out$ are disjoint.
\item[-] $In\sqcup Out = Open\sqcup Closed$, we say the leaves are labeled as in/out and open/closed.
\item[-] $Out\subset Open$, all the outgoing leaves are open and $In\subset Open\sqcup Closed$ all incoming leaves are either open or closed.
\item[-]$V_w \cap (Open\sqcup Closed) = \emptyset$.
\item[-] A closed leaf is the only labeled leaf on its boundary cycle.
\end{itemize} 
Additionally, we have as part of the data:  
\begin{itemize}
\item[-] An ordering of the white vertices or equivalently a labeling of the white vertices by $1,2, \ldots \vert V_w\vert$.
\item[-] A choice of a start half edge on each white vertex  i.e., the half edges incident at a white vertex are totally ordered not only cyclically ordered.
\item[-] An ordering of the incoming leaves
\item[-] An ordering of the outgoing leaves.
\end{itemize}
We allow degenerate graphs which are either the empty graph, or a corolla with 1 or 2 leaves on a black vertex. 
\end{dfn}
\begin{dfn}
\label{bw_def}
A \emph{black and white graph} is a generalized black and white graph in which all the leaves are labeled, except possibly the leaves which are connected to the start of a white vertex, which are allowed to be unlabeled.
Figure \ref{BW_example} shows an example of a black and white graph.
\end{dfn}

\begin{rmk}
Note that a black and white fat graph with no white vertices is just an open-closed fat graph with no outgoing closed leaves.
\end{rmk}

\begin{figure}[h!]
  \centering
    \includegraphics[scale=0.7]{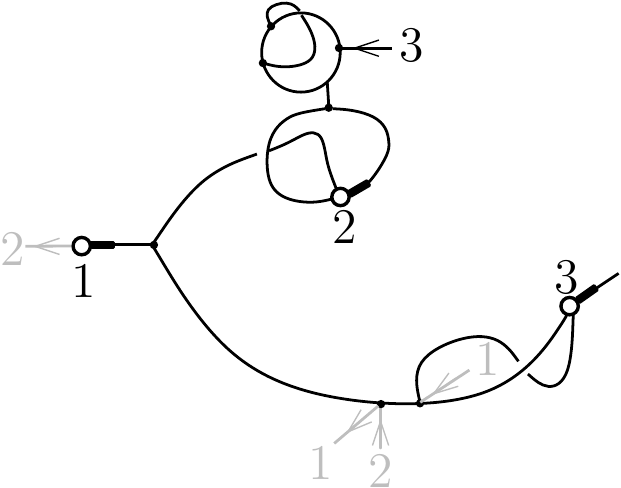}
  \caption{An example of a black and white fat graph.  The incoming and outgoing leaves are marked with arrows. Leaves in black are closed and leaves in grey are open.  The start half edges of the white vertices are thickened.}
  \label{BW_example}
\end{figure}

As for open-closed fat graphs, from a black and white fat graph $G$ we can construct an open-closed cobordism $\Sg$.  First construct a bordered oriented surface $\Sigma_{G}$.  To do this, we thicken the edges of $G$ to strips, the black vertices to disks and glue them together according to the cyclic ordering.  Then, we thicken each white vertex to an annulus and glue to its outer boundary the strips corresponding to the edges attached to it according to the cyclic ordering.  We label the inner boundary of the annuli as outgoing closed and order these components by the ordering of the white vertices.  We label and order the rest of the boundary of  $\Sigma_{G}$ in the same way as for open-closed fat graphs.  This construction gives and open-closed cobordism well defined up to topological type (see Figure \ref{BW_fattening}).

\begin{figure}[h!]
  \centering
    \includegraphics[scale=0.7]{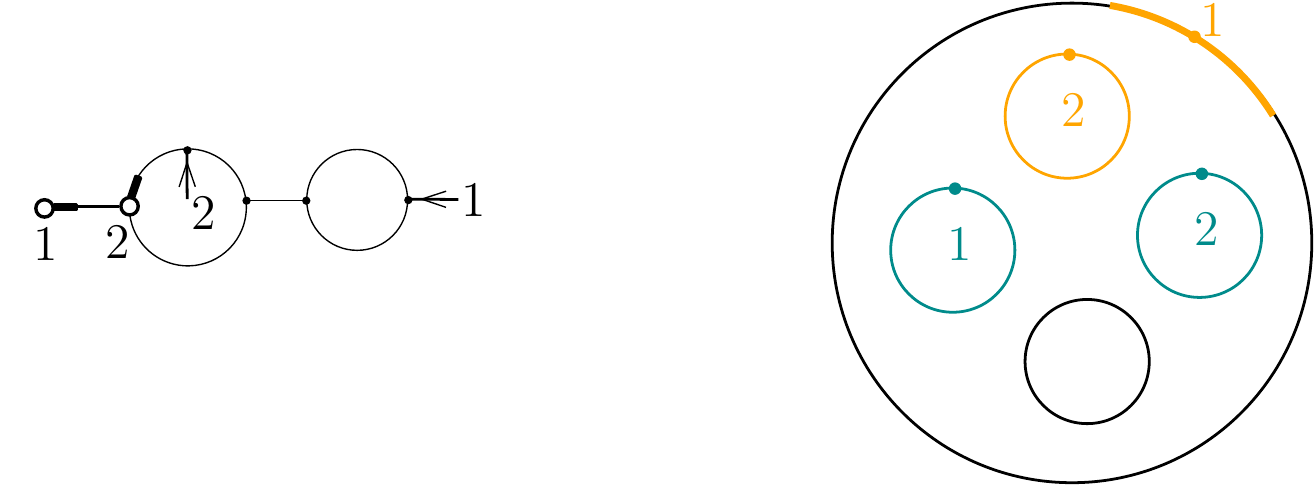}
  \caption{A black and white fat graph and its corresponding open-closed cobordism.  The incoming boundary is shown in yellow and the outgoing boundary in green.}
  \label{BW_fattening}
\end{figure}

\begin{dfn}
An \emph{orientation} of a fat graph $\Gamma$ is a unit vector in $\mathrm{det}(\R(V\sqcup H))$, this is equivalent to an ordering of the set of vertices and an orientation for each edge.  
\end{dfn}

\begin{dfn}
\label{underlying_BW_graph}
A generalized black and white graph $G$ has an \emph{underlying black and white graph} $\lfloor G\rfloor$ defined as $\lfloor G \rfloor = G$ if $G$ is already a black and white graph i.e., $G$ has no unlabeled leaves that are not connected to the start of a white vertex.  On the other hand, let $G$ be a graph with an unlabeled leaf $l$ which is not connected to the start of a white vertex, let $e_l$ denote the edge of $l$ and $v_l$ the other vertex to which $e_l$ is attached.  If $v_l$ is white or if $v_l$ is black and $|v_l|>3$, then $\lfloor G \rfloor$ is the empty graph, where $|v_l|$ denotes the valence of $v_l$. If $v_l$ is black and $|v_l|=3$, then $\lfloor G \rfloor$ is the graph obtained by forgetting $l$, $e_l$ and $v_l$.

If $G$ has an orientation, it induces an orientation on $\lfloor G \rfloor$ which we only need to describe in the case where $\lfloor G \rfloor$ is not $G$ or the empty graph.  In this case, let $l$, $e_l$ and $v_l$ be given as above and let $e_l=\lbrace h_l, \tilde{h}_l\rbrace$ where $s(h_l)=v_l$ and $s(\tilde{h}_l)=l$.  Let $s^{-1}(v_l)=(h_1,h_2,h_l)$ occurring in that cyclic ordering.  Rewrite the orientation of $G$ as $v_l\wedge h_1\wedge h_2\wedge h_l\wedge \tilde{h}_{l}\wedge l\wedge x_1\wedge \ldots \wedge x_k$.  The induced orientation in $\lfloor G \rfloor$ is $x_1\wedge \ldots \wedge x_k$. 
\end{dfn}

\begin{dfn} [Edge Collapse]
Let $G$ be a (generalized) black and white graph, and $e$ be an edge of $G$ which is neither a loop nor does it connect two white vertices.  The \emph{set of edge collapses} $G/e$ is the collection of (generalized) black and white graphs obtained by collapsing $e$ in $G$ and identifying its two end vertices.  If both vertices are black we declare the new vertex to be black. If one of the vertices is white, we declare the new vertex to be white with the same label as the white vertex of $e$.
\begin{itemize}
\item[-] \emph{Fat structure} The collapse of $e$ induces a well defined cyclic structure of the half edges incident at the new vertex.
\item[-] \emph{Start half edge} If $e$ does not contain the start half edge of a white vertex, then there is a unique black and white fat graph obtained by collapsing $e$.  If $e$ contains the start half edge of a white vertex, there is a finite collection of black and white graphs obtained by collapsing $e$.  Each graph in this collection corresponds to a choice of placement of the start half edge among the half edges incident at the collapsed black vertex.  See Figure \ref{edge_collapse} an example of this collection.
\item[-] \emph{Orientation}  An orientation of $G$ induces an orientation of the elements of $G/e$ as follows.  Let $e:=\lbrace h_1,h_2\rbrace$,  $s(h_1)=v_1$, and $s(h_2)=v_2$.  Write the orientation of $G$ as $v_1\wedge v_2\wedge h_1\wedge h_2\wedge x_1\wedge \ldots \wedge x_k$. Then the induced orientation of an element of $G/e$ is given by $v\wedge x_1\wedge \ldots \wedge x_k$.
\end{itemize}
\end{dfn}

\begin{figure}[h!]
  \centering
    \includegraphics[scale=0.9]{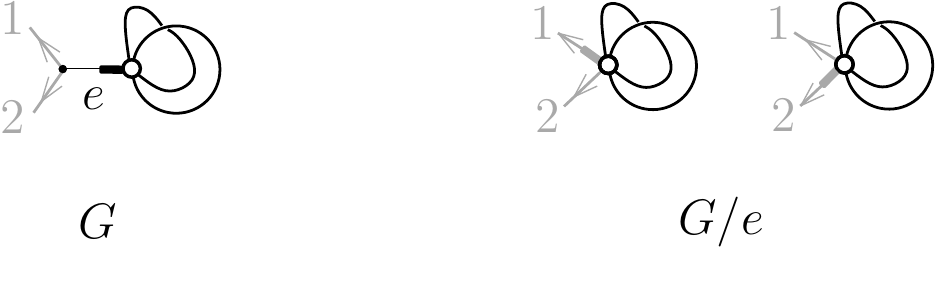}
  \caption{On the left a black and white fat graph $G$ and to the right its edge collapse set $G/e$. }
  \label{edge_collapse}
\end{figure}

\begin{dfn}
Let $G$ and $\tilde{G}$ be generalized black and white graphs.  We say $\tilde{G}$ is a \emph{blow-up} of $G$ if there is an edge $e$ of $\tilde{G}$ such that $G\in\tilde{G}/e$.
\end{dfn}
\begin{rmk}
Note that the blow-up of a black and white graph is not necessarily a black and white graph again, since the blow might contain unlabeled leaves which are not the start of a white vertex. See Figure \ref{blow-up} for an example.
\end{rmk}

\begin{figure}[h!]
  \centering
    \includegraphics[scale=0.9]{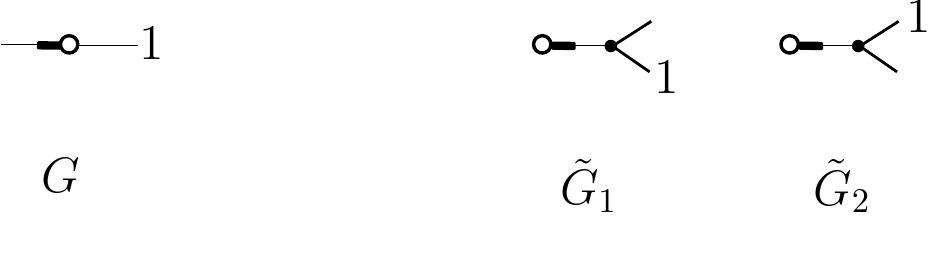}
  \caption{An example of a black and white fat graph $G$ with blow-ups $\tilde{G}_1$ and $\tilde{G}_2$ that have unlabeled leaves that are not connected to the start edge of a white vertex.}
  \label{blow-up}
\end{figure}

\begin{dfn} [The chain complex of Black and White Graphs]
\label{bw_cpx_def}
The chain complex of black and white fat graphs $\bwgraphs$ is the complex generated as a $\Z$ module by isomorphism classes of oriented black and white graphs modulo the relation where $-1$ acts by reversing the orientation.  The degree of a black and white graph $G$ is 
\[\mathrm{deg}(G):=\sum_{v\in V_b}(|v|-3)+\sum_{v\in V_w}(|v|-1)\]
where $|v|$ is the valence of the vertex $v$.  The differential of a black and white graph $G$ is
\[d(G):=\sum_{
\begin{array}{c}
\scriptstyle{(\tilde{G},e)}\\
\scriptstyle{G\in\tilde{G}/e}
\end{array}}
\lfloor\tilde{G}\rfloor\]
where the sum runs over all isomorphism classes of generalized black and white graphs which are blow-ups of $G$.  Figure \ref{differential} gives some examples of the differential.
\end{dfn}

\begin{rmk}
In \cite{wahlwesterland}, it is shown that $d$ is indeed a differential.  Note that, since the number of white vertices, and the number of boundary cycles remain constant under blow-ups and edge collapses, the chain complex $\bwgraphs$ splits into chain complexes each of which is finite and corresponds to a topological type of an open-closed cobordism. 
\end{rmk}

\begin{figure}[h!]
  \centering
    \includegraphics[scale=0.7]{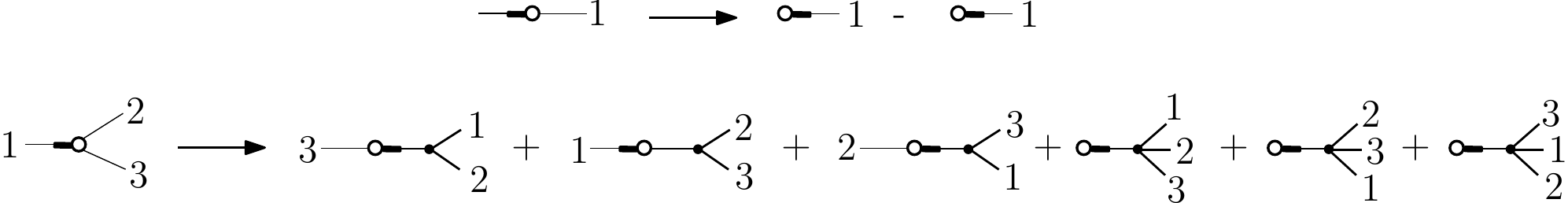}
  \caption{Differential for two black and white fat graphs.  All the labeled leaves are incoming open.}
  \label{differential}
\end{figure}

\subsection{Black and white graphs as models for the mapping class group}
Using a partial compactification of the moduli space of open-closed cobordisms, Costello proves that the chain complex $\bwgraphs$ is a model for the mapping class groups of open-closed cobordisms (cf.\cite{costellorg, costellotcft}). We give a new proof this result by showing that $\bwgraphs$ is a chain complex of $\vert\Fatad\vert$.  In \cite{Godinunstable}, Godin gives a CW structure on $\vert\Fatoc\vert$ which restricts to $\vert\Fatad\vert$ in which each $p$-cell is given by a fat graph $[\Gamma]$ of degree $p$ where 
\[\mathrm{deg}([\Gamma]):=\sum_v(\vert v \vert -3)\]
and the sum ranges over all inner vertices of $[\Gamma]$ and $\vert v \vert$ denotes the valence of $v$.  From this structure, she constructs a chain complex which is the complex generated as a $\Z$ module by isomorphism classes of oriented fat graphs modulo the relation where $-1$ acts by reversing the orientation.  The differential of a fat graph $[\Gamma]$ is
\[d([\Gamma]):=\sum_{
\begin{array}{c}
\scriptstyle{([\tilde{\Gamma}],e)}\\
\scriptstyle{[\Gamma]=[\tilde{\Gamma}/e]}
\end{array}}
[\tilde{\Gamma}]\]

While working with Sullivan diagrams, a quotient of $\bwgraphs$, Wahl and Westerland give a natural association that constructs a black and white graph from an admissible fat graph by collapsing the admissible boundary to a white vertex and using the leaf marking the admissible boundaries to mark the start half edge \cite{wahlwesterland}.  This construction is only well defined in special kind of admissible fat graphs.

\begin{dfn}
Let $\Gamma$ be an admissible fat graph.  We say $\Gamma$ is \emph{essentially trivalent at the boundary}, if every vertex on the admissible cycles of $\Gamma$ is trivalent or it has valence $4$ and is attached to the leaf marking the admissible cycle.
\end{dfn}
\begin{rmk}
\label{iso_bw_Fat3}
There is a bijection between the set of isomorphism classes of black and white graphs and the set of isomorphism classes of admissible fat graphs which are essentially trivalent at the boundary.  To see this, let $G$ be a black and white graph. Construct an admissible fat graph $\Gamma_G$ by expanding each white vertex to an admissible cycle.  The start half edge of the white vertex gives the position of the leaf marking its corresponding admissible cycle.  That is, if the start half edge is an unlabeled leaf, then the leaf of the corresponding admissible cycle in $\Gamma_G$ is a attached to a trivalent vertex.  Otherwise, the leaf corresponding to the admissible cycle is attached to the same vertex to which the start half edge is attached to.  Label all the admissible leaves using the labeling of the white vertices in $G$.  The fat graph $\Gamma_G$ is by construction an admissible fat graph which is essentially trivalent at the boundary.  Figure \ref{BW_Fat_cons} shows an example of this construction.   In the other direction, given an admissible fat graph $\Gamma$ which is essentially trivalent at the boundary, construct a black and white fat graph $G_\Gamma$ by collapsing the admissible boundaries to white vertices and placing the start half edge according to the position of the admissible leaves in $\Gamma$.  These constructions are clearly inverse to each other. 
\begin{figure}[h!]
  \centering
    \includegraphics[scale=0.7]{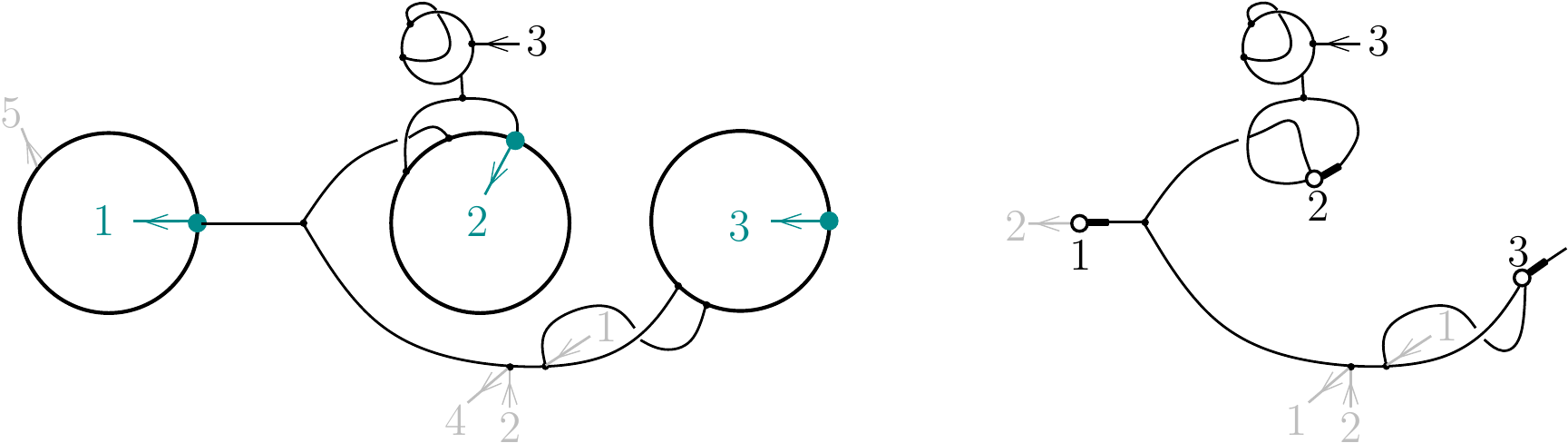}
  \caption{On the left an admissible fat graph that is essentially trivalent at the boundary and on the right its corresponding black and white graph}
  \label{BW_Fat_cons}
\end{figure}
\end{rmk}

However, this natural association does not give a chain map between $\bwgraphs$ and the chain complex constructed by Godin.  To realize this, note that by expanding white vertices to admissible cycles on a black and white graph, all black vertices remain unchanged i.e., a black vertex of degree $n$ is sent to a black vertex of degree $n$.  However, a white vertex of degree $n$ is sent to an admissible cycle with $n+1$ edges where the sum of the degrees of its vertices is at most $1$.  Instead of giving a chain map we will construction a filtration 
\[\Fatad\ldots \supset \Fat^{n+1} \supset \Fat^{n} \supset \Fat^{n-1} \ldots \Fat^{1}\supset\Fat^{0}\]
that gives a cell-like structure on $\Fatad$ where the quasi-cells are indexed by black and white graphs i.e., $\vert \Fat^{n}\vert / \vert \Fat^{n-1}\vert\cong\vee S^n$ where the wedge sum is indexed by isomorphism classes of black and white graphs of degree $n$.  

\subsubsection{The Filtration}

In order to give such a filtration we use a mixed degree on $\Fatad$ which is given by the valence of the vertices and the number edges on the admissible cycles.

\begin{dfn}
Let $\Gamma$ be an admissible fat graph with $k$ admissible cycles.  Let $E_a$ denote the set of edges on the admissible cycles, $V_{b}$ the set of vertices that do not belong to the admissible cycles, $V_a$ the set of vertices on the admissible cycles which are not attached to an admissible leaf, and $V_{a,*}$ be the set of vertices on the admissible cycles which are attached to an admissible leaf.  The \emph{mixed degree of $\Gamma$} is 
\[\md(\Gamma):=\vert E_a\vert -k + \sum_{v\in V_a\cup V_{b}}(\vert v\vert -3)+ \sum_{v\in V_{a,*}}(\max\lbrace0,\vert v\vert-4\rbrace)\]
Figure \ref{mixed_degree} shows some examples of admissible fat graphs of mixed degree two.
\end{dfn}
Notice that the mixed degree is well defined for isomorphism classes of admissible fat graphs.  We will use this degree to describe a filtration of $\Fatad$

\begin{figure}
  \centering
    \includegraphics[scale=0.7]{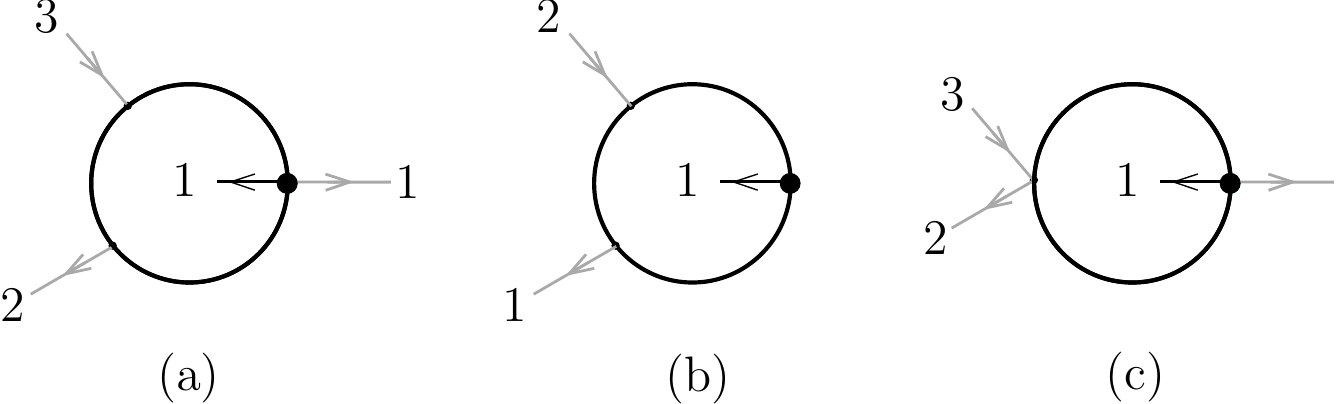}
  \caption{Three different admissible fat graphs all of mixed degree two.  In particular (a) is $l_3$ and (b) is $\tilde{l}_3$.}
  \label{mixed_degree}
\end{figure}

\begin{dfn}
$\Fat^n$ is the full subcategory of $\Fatad$ on objects isomorphism classes of admissible fat graphs $[\Gamma]$ s.t. $\md([\Gamma])\leq n$.
\end{dfn}

\subsubsection{The Quasi-cells}  We now describe the quasi-cell corresponding to a black and white graph $G$.
\begin{dfn}
An admissible fat graph $\tilde{\Gamma}$ is a \emph{blow-up} of an admissible fat graph $\Gamma$ if there is an edge $e$ of $\tilde{\Gamma}$ such that $\Gamma=\tilde{\Gamma}/e$.  Furthermore, $\tilde{\Gamma}$ is a \emph{blow-up away from the admissible boundary} if $e$ does not belong to an admissible cycle in $\tilde{\Gamma}$. If $e$ contains a vertex on an admissible cycle but does not belong to one we say $\tilde{\Gamma}$ is obtained from $\Gamma$ by \emph{pushing away from the admissible cycles}. Finally, $\tilde{\Gamma}$ is a \emph{blow-up at the admissible boundary} if $e$ belongs to an admissible cycle in $\tilde{\Gamma}$.  
\end{dfn}

\begin{dfn}
A white vertex on a black and white graph is called \emph{generic} if all its leaves are labeled and \emph{suspended} otherwise.  Similarly, an admissible cycle $C$ in a graph which is essentially trivalent at the boundary is called \emph{generic} if the vertex connected to the admissible leaf has valence at least $4$ and \emph{suspended} otherwise.
\end{dfn}

\begin{dfn} We define the following full subcategories of $\Fatad$
\begin{itemize}
\item[-] For $n\geq 3$, $\Tcat_n$ is the full subcategory of $\Fatad$ on objects trees with $n$ leaves $\lbrace1, 2, \ldots n\rbrace$ occurring in that cyclic order.

\item[-] Let $l_n$ be the admissible fat graph of mixed degree $n-1$ with one admissible boundary cycle which consists of $n$ edges, together with $n$ leaves labeled $\lbrace 1,2 \ldots n\rbrace$ attached to it in that cyclic order, such that leaf $1$ is attached to the vertex connected to the admissible leaf, see Figure \ref{mixed_degree} (a).  $\Lcat_n$ is the full subcategory of $\Fatad$ on objects $l_n$ and all admissible fat graphs $[\Gamma]$ obtained from $l_n$ by collapsing edges at the admissible cycles and blow-ups away from the admissible cycles.  See Figure \ref{L3} for an example.

\item[-] Let $\tilde{l}_n$ be the admissible fat graph of mixed degree $n-1$ with one admissible boundary cycle which consists of $n$ edges and $n-1$ leaves labeled $\lbrace 1,2 \ldots n-1\rbrace$ attached to it in that cyclic ordering such that there is no leaf attached to the vertex connected to the admissible leaf, see Figure \ref{mixed_degree} (b).  $\tLcat_n$ is the full subcategory of $\Fatad$ on objects $\tilde{l}_n$ and all admissible fat graphs $[\Gamma]$ obtained from $\tilde{l}_n$ by collapsing edges at the admissible cycles and blow-ups away from the admissible cycles.  See Figure \ref{tL3} for an example.
\end{itemize}
\end{dfn}

\begin{figure}
  \centering
    \includegraphics[scale=0.7]{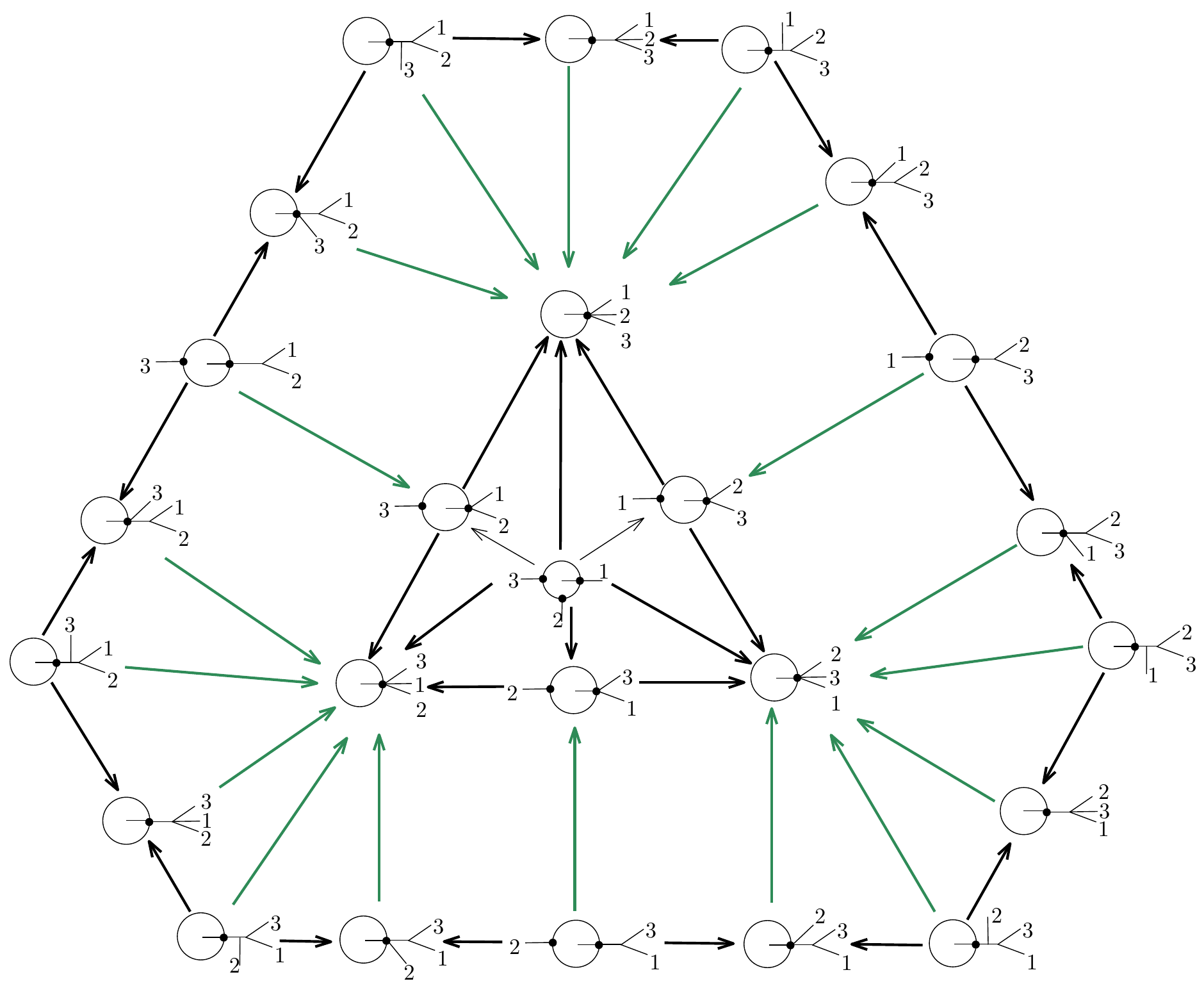}
  \caption{The category $\Lcat_3$.  The arrows in green indicate the deformation retraction onto the core $\Cat_3$}
  \label{L3}
\end{figure}

\begin{figure}
  \centering
    \includegraphics[scale=0.7]{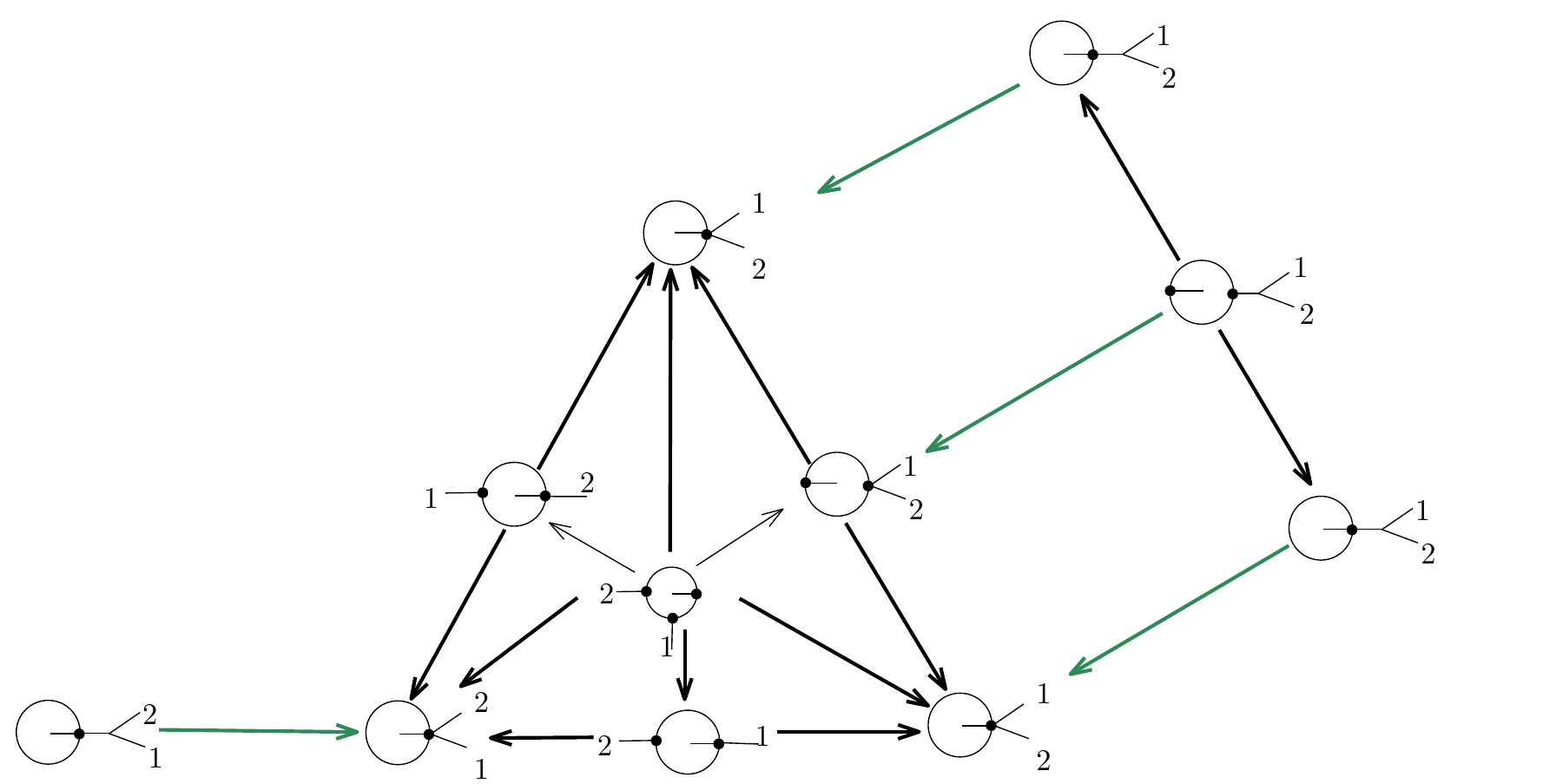}
  \caption{The category $\tLcat_3$.  The arrows in green indicate the deformation retraction onto the core $\tCat_3$}
  \label{tL3}
\end{figure}

\begin{dfn}
Let $G$ be a black and white graph, $V_b$ be the set of its black vertices, $V_g$ be the set of generic white vertices and $V_{s}$ be the set of suspended white vertices.  The \emph{quasi-cell} of $G$ is the category
\[\Ecat_G\cong
\prod_{v\in V_{b}}\Tcat_{\vert v\vert}\times
\prod_{v\in V_g}\Lcat_{\vert v\vert}\times
\prod_{v\in V_{s}}\tLcat_{\vert v\vert}
\]
\end{dfn}

\begin{rmk}
\label{Godin_asso}
In \cite{Godinunstable}, Godin shows that $\vert\Tcat_n\vert$ is homeomorphic to a disk $D^{n-3}$.  In fact by choosing a root of the trees in $\Tcat_n$ we can show that $\vert\Tcat_n\vert$ is a Stasheff polyhedron or associahedron whose vertices are given by  different ways in which we can bracket a product of $n-1$ variables.
\end{rmk}

\begin{dfn}
The \emph{core of $\Lcat_n$} which we denote $\Cat_n$ is the full subcategory of $\Lcat_n$ on objects obtained from $l_n$ by edge collapses.  Similarly the \emph{core of $\tLcat_n$} which we denote $\tCat_n$ is the full subcategory of $\tLcat_n$ on objects obtained from $\tilde{l}_n$ by edge collapses.
\end{dfn}
\begin{rmk}
Note that $\Cat_n$ is the full subcategory of $\Lcat_n$ on objects admissible fat graphs of mixed degree $n$.
\end{rmk}

\begin{dfn}
\hspace{20 mm}
\begin{itemize}
\item[-] The \emph{boundary of $\Cat_n$ (respectively $\tCat_n$)}, which we denote $\partial\Cat_n$ (respectively $\partial\tCat_n$), is the full subcategory of $\Cat_n$ (resp. $\tCat_n$) on objects different from $l_n$ (respectively $\tilde{l}_n$).
\item[-] The \emph{interior of the realization of $\Cat_n$}, is the subspace 
\[int(\vert \Cat_n\vert )= \vert \Cat_n\vert-\vert \partial\Cat_n\vert\]
similarly
\[int(\vert \tCat_n\vert )= \vert \tCat_n\vert-\vert \partial\tCat_n\vert\]
\end{itemize}
\end{dfn}

\begin{lem}
\label{core_simplex}
The nerve $N\Cat_n$ is isomorphic to the barycentric subdivision of $\Delta[n-1]$, thus $\vert \Cat_n\vert$ is homeomorphic to $\Delta^{n-1}$.  The interior of the realization of $\tCat_n$ is homeomorphic to the interior of $\Delta^{n-1}$ i.e., $int(\vert \tCat_n\vert)\cong int(\Delta^{n-1})$.
\end{lem}
\begin{proof}
Note first that the fat structure together with the admissible leaf induce an ordering of the vertices on the admissible cycles of $l_n$ and $\tilde{l}_n$, where the first vertex is the vertex connected to the admissible leaf. In the case of $\Cat_n$, for $0\leq i\leq n-1$, let $e_i$ denote the edge connecting the vertices $(i-1)$ and $i$.  Let $[n]$ denote the set $\lbrace 0, 1, \dots  n\rbrace$.  It is enough to show that $\Cat_n$ is isomorphic to the poset category $\Pcat([n-1])$.  Note that the fat structure and the labeling of the leaves gives that for any object $[\Gamma]$ in $\Cat_n$ there is a unique morphism $l_n\to [\Gamma]$.  Therefore $\Cat_n$ is isomorphic to the under-category $l_n/\Lcat_n$. An object $\beta:l_n\to [\Gamma]$ in $l_n/\Lcat_n$, is uniquely determined by a set of edges on the admissible cycle $\zeta_{\beta}:=\lbrace  e_{\beta_1}, e_{\beta_2},\ldots e_{\beta_r}\rbrace$ whose union is not the entire boundary cycle.  For the object given by the identity, the set  $\zeta_{id}$ is the empty set.  We define a functor $\Phi:l_n/\Lcat_n\to\Pcat([n-1])$ on objects by $\Phi(\beta):=\lbrace 0, 1,\ldots ,n-1\rbrace-\lbrace\beta_1,\beta_2,\ldots \beta_r\rbrace$ this induces a natural map on morphisms and it is easy to see that it is an isomorphism.
 
In the case of $\tCat_n$, for $0\leq i\leq n-1$ let $e_i$ denote the edge connecting the vertices $i$ and $(i+1)$.  Then the argument above shows that $\tilde{l}_n/\tCat_n$ is isomorphic to $\Pcat[n-1]$.  However, the forgetful functor $F:\tilde{l}_n/\tCat_n\to \tCat_n$ is injective on morphisms but not on objects.  To see this, let $\zeta_{\beta_1}:=\lbrace e_1,e_2,\ldots ,e_{n-1}\rbrace$ and let $\zeta_{\beta_2}:=\lbrace e_0,e_1,\ldots ,e_{n-2}\rbrace$, then $F(\beta_1)=F(\beta_2)$.  Therefore the realization of $\tCat_n$ is not homeomorphic to the simplex. However, the geometric realization of $F$ induces a map $\vert F \vert: \vert \tilde{l}_n/\tCat_n \vert\cong\vert \Pcat[n-1]\vert=\Delta^{n-1} \to \vert\tCat_n\vert$ which is injective on the interior of the simplex.
\end{proof}

\begin{dfn}
Let $\Gamma$ be an admissible fat graph, $V_a$ be the set of vertices on the admissible cycles which are not attached to an admissible leaf, and $V_{a,*}$ be the set of vertices on the admissible cycles which are attached to an admissible leaf.  Let $\xi_\Gamma$ be the set 
\[\xi_\Gamma:=\lbrace v\in V_a\vert \vert v\vert > 3\rbrace\cup\lbrace v\in V_{a,*}\vert \vert v\vert > 4\rbrace\]  
We can construct from $\Gamma$, an admissible fat graph which essentially trivalent at the boundary, which we denote $\hat{\Gamma}$, by pushing out all the vertices of $\xi_\Gamma$ i.e., by blow-ups away from the admissible boundary given by a single edge on each of the vertices of $\xi_\Gamma$.  We call this procedure \emph{making the graph $\Gamma$ essentially trivalent}.
\end{dfn}

\begin{dfn}
The \emph{black and white degree of an admissible fat graph $[\Gamma]$} is 
\[\bwd([\Gamma]):=\mathrm{deg}(G_{\hat{\Gamma}})\]
where $[\hat{\Gamma}]$ is the graph obtained by making $[\Gamma]$ essentially trivalent, $G_{\hat{\Gamma}}$ is the black and white graph corresponding to $\hat{\Gamma}$ under the isomorphism given in \ref{iso_bw_Fat3} by collapsing admissible boundaries to white vertices, and $\mathrm{deg}$ is the degree of black and white graphs.
\end{dfn}

We define a few special subcategories which are the building blocks of a quasi-cell.

\begin{dfn}
\label{boundaries} 
\hspace{20 mm}
\begin{itemize}
\item[-] The \emph{boundary of $\Tcat_n$, $\Lcat_n$ and $\tLcat_n$} which we denote  $\partial\Tcat_n$, $\partial\Lcat_n$ and $\partial\tLcat_n$, are the full subcategories of respectively $\Tcat_n$, $\Lcat_n$ and $\tLcat_n$ on objects of mixed degree $k<n$.
\item[-] The \emph{thick boundary of $\Tcat_n$, $\Lcat_n$ and $\tLcat_n$} which we denote  $\eth\Tcat_n$, $\eth\Lcat_n$ and $\eth\tLcat_n$, are the full subcategories of respectively $\Tcat_n$, $\Lcat_n$ and $\tLcat_n$ on objects of black and white degree $k<n$.
\end{itemize}
\end{dfn}

\begin{rmk}
Note that $\eth\Tcat_n = \partial \Tcat_n$.  Moreover, note that $\vert \bL\vert$ intersects $\vert \Lcat_n \vert$  exactly at the boundary of the core $\vert \partial  \Cat_n\vert$, and similarly, $\vert \btL\vert$ intersects $\vert \tLcat_n \vert$  exactly at the boundary of the core $\vert \partial \tCat_n\vert$.
\end{rmk}

We now construct functors $P:\Lcat_n \to \Cat_n$ and $\widetilde{P}:\tLcat_n \to \tCat_n$.  For an object $[\Gamma]$ of $\Lcat_n$, let $F_\Gamma$ denote the sub-forest of all edges that are not on the admissible cycles and are not connected to a leaf.  We define the functor $P$ on objects by $[\Gamma]\mapsto [\Gamma/F_\Gamma]$.  This induces a natural map on morphisms.  To see this, let $\psi_F: [\Gamma]\to [\Gamma/F]$ be a morphism in $\Lcat_n$ and note that $[\Gamma/(F\cup F_\Gamma)]=[(\Gamma/F)/(F_{\Gamma/F})]$.  We define $\widetilde{P}$ similarly, see Figures \ref{L3} and \ref{tL3}.

\begin{lem}
\label{thick_boundary}
The functors $P$ and $\widetilde{P}$ induce maps $\vert P \vert:(\vert \Lcat_n \vert,\vert \partial\Lcat_n\vert) \to (\vert \Cat_n \vert,\vert \partial\Cat_n\vert)$  and $\vert \widetilde{P} \vert:(\vert \tLcat_n \vert,\vert \partial\tLcat_n\vert) \to (\vert \tCat_n \vert,\vert \partial\tCat_n\vert)$ which is are homotopy equivalences of pairs.
\end{lem}
\begin{proof}
In this proof we always use isomorphism classes of graphs, but we exclude the brackets from the notation, to avoid clutter. Note that the objects of $\Cat_n$ have no edges which are not on the admissible cycles or connected to a leaf, thus $P$ is the identity on objects of the core.  Therefore, $P$ restricts to a functor $p:=P\vert: \partial\Lcat_n \to \partial\Cat_n$. We show first that $\vert P \vert$ is a homotopy equivalence. Let $\iota$ denote the inclusion functor $\iota:\Cat_n\cof \Lcat_n$.  It is clear that $P\circ \iota = id_{\Cat_n}$.  On the other hand, we have a natural transformation $\eta:id_{\Lcat_n}\Longrightarrow \iota\circ P$ given by $\eta_{\Gamma}:\Gamma\to \Gamma/F_\Gamma$.  So $\vert P\vert$ is a homotopy equivalence.  Note that, $\eta_{\Gamma}=id_{\Gamma}$ for $\Gamma\in\Cat_n$. Therefore, $\vert \eta \vert$ is a strong deformation retraction of $\Lcat_n$ onto its core. This argument depends only on the fact that there is a unique morphism $\Gamma \to P(\Gamma)$.  We will this idea several times in what comes next.

The functor $p$, pushes $\partial\Lcat_n$ onto $\partial\Cat_n$.  We define a notion of depth, and show that $p$ is the composition of $n-1$ functors which sequentially push in the graphs according to their depth and that each functor induce a homotopy equivalence on realizations.  Let $\Gamma$ be an object of $\Lcat_n$.  The \emph{depth of $\Gamma$} is
\[\dep(\Gamma):=\vert E_a\vert\]
where $E_a$ is the set of edges on the admissible cycle.  Recall that $\vert \Cat_n\vert$ is the barycentric subdivision of $\Delta[n-1]$, and thus we can interpret an object $\Gamma$ in $\Cat_n$ as representing a face of $\Delta[n-1]$ of a certain dimension.  We call this the \emph{dimension of $\Gamma$} and denote it $\dim(\Gamma)$.  For $1\leq i\leq n$ we define a category $X_i$ to be the full subcategory of $\eth\Lcat_n$ on objects:
\begin{itemize}
\item[-] $\Gamma\in\partial\Lcat_n$ such that $\dep(\Gamma)\geq i$
\item[-] $\Gamma\in\partial\Cat_n$ such that $\Gamma$ represents a face of  $\vert\Cat_n \vert$ of dimension $\leq n-2$ 
\end{itemize}

Note that for $\Gamma\in\partial\Lcat_n$, it holds that $1\leq \dep(\Gamma)\leq n-1$. Therefore, $X_1=\partial\Lcat_n$ and $X_n=\partial\Cat_n$.  

For $1\leq i\leq n-1$ we define a functors $\psi_i:X_i\to X_{i+1}$ on objects by:

\[\psi_i(\Gamma):=\left\lbrace 
\begin{array}{cl}
p(\Gamma) & \Gamma\in\partial\Lcat_n, \text{ }\dep(\Gamma)=i, \\
\Gamma  & \text{else}
\end{array}
\right. 
\] 
with the natural map induced on morphisms.

Thus, we have a sequence of functors 
\[\partial\Lcat_n= X_1 \stackrel{\psi_1}{\longrightarrow} X_2\stackrel{\psi_2}{\longrightarrow}\ldots X{n-1}=\stackrel{\psi_{n-1}}{\longrightarrow} X_n=\partial\Cat_n\]
and it clearly holds that $p:=\psi_{n-1}\circ\ldots\psi_2\circ\psi_1$.

The over category $\psi_i/\Gamma$ has objects $(\tilde{\Gamma},\alpha)$ where $\tilde{\Gamma}\in X_i$ and $\alpha: p(\tilde{\Gamma})\to\Gamma$ is a morphism in  $X_{i+1}$. Morphisms from $(\tilde{\Gamma_1},\alpha_1)$ to $(\tilde{\Gamma_2},\alpha_2)$ in $\psi_i/\Gamma$ are given by morphisms $\beta$ in $X_i$ such that the bottom triangle in diagram \ref{left_fiber} commutes.

\begin{equation}
\begin{tikzpicture}[scale=0.5]
\node (a) at (0,0){$p(\tilde{\Gamma}_1)$};
\node (b) at (5,0) {$p(\tilde{\Gamma}_2)$};
\node (c) at (0,3){$\tilde{\Gamma}_1$};
\node (d) at (5,3){$\tilde{\Gamma}_2$};
\node (e) at (2.5,-3){$\Gamma$};

\path[auto,arrow,->] (c) edge node{$\beta$} (d)
					 (c) edge node{} (a)
					 (d) edge node{} (b)
					 (a) edge node{$p(\beta)$} (b)
					 (a) edge node [swap] {$\alpha_1$} (e)
					 (b) edge node{$\alpha_2$} (e);	
\end{tikzpicture}
\label{left_fiber}
\end{equation}

We separate $\psi_i/\Gamma$ into three different cases
\begin{description}
\item[If $\Gamma\in\partial\Lcat_n$]  Morphisms of fat graphs are given by collapsing edges.  Thus, if $\Gamma\in\partial\Lcat_n$, all the graphs and arrows in diagram \ref{left_fiber} are be objects and morphisms in $X_i$.  Therefore $\psi_i/\Gamma=X_i/\Gamma$ which is a contractible category.

\item[If $\Gamma\in\partial\Cat_n$, $\dim(\Gamma)\leq i-2$] For $j=1,2$, the graph $\tilde{\Gamma}_j$ is a blow-up away from the admissible boundary of the graph $p(\tilde{\Gamma}_j)$.  Moreover, the condition on the dimension of $\Gamma$ implies that $\Gamma\in X_i$.  Thus, the existence of morphisms $\alpha_j$ implies that there are morphisms $\tilde{\alpha}_j:\tilde{\Gamma}_j\to\Gamma$ in $X_i$.  Then the category $\psi_i/\Gamma$ is contractible, since the object $(\Gamma,id_\Gamma)$ is terminal.

\item[If $\Gamma\in\partial\Cat_n$, $\dim(\Gamma)= i-1$] In this case, $\Gamma$ is not an object in $X_i$.  However, by the case above, we can see that the objects of $\psi/\Gamma$ are all blow-ups of $\Gamma$ together with a map to $\Gamma$ in $X_{i+1}$.  These collapse maps onto $\Gamma$ are unique.  Therefore, $\psi/\Gamma$ is the full subcategory of $X_i$ on objects that are are blow-ups of $\Gamma$.  Let $\Dcat_1$ denote the full subcategory of $X_i$ on objects that are obtained from $\Gamma$ by blow-ups away from the admissible cycle.  Similarly, let $\Dcat_2$ denote the full subcategory of $X_i$ on objects that are obtained from $\hat{\Gamma}$ by blow-ups away from the admissible cycle, where $\hat{\Gamma}$ is the graph obtained by making $\Gamma$ essentially trivalent at the boundary.  Then, we have inclusions of categories 

\[\Dcat_2\cof \Dcat_1\cof \psi_i/\Gamma\]

Let $\tilde{\Gamma}$ be an object in $\psi_i/\Gamma$  which is not an object in $\Dcat_1$.  There is a unique morphism $\gamma_{\tilde{\Gamma}}$ in $X_{i+1}$ of the form $\gamma_{\tilde{\Gamma}}:p(\tilde{\Gamma})\to \Gamma$ and this morphism is given by collapsing edges on the admissible cycle.  Note that $\tilde{\Gamma}$ and $p(\tilde{\Gamma})$ have the same structure on the admissible cycle, in particular they have the same number of edges on the admissible cycle. Thus, the map $\gamma_{\tilde{\Gamma}}$ lifts to a unique map $\gamma_{\tilde{\Gamma} *}:\tilde{\Gamma}\to\Gamma'$ where $\Gamma'$ is an object in $\Dcat_1$. More precisely, the morphism $\gamma_{\tilde{\Gamma} *}$ is given by collapsing the same edges on the admissible cycles of $\tilde{\Gamma}$ that  $\gamma_{\tilde{\Gamma}}$ collapses on the admissible cycles of $p(\tilde{\Gamma})$.  This defines a functor $G_1:\psi_i/\Gamma\to \Dcat_1$ that is the identity on objects of $\Dcat_1$ and on all other objects it is given by 
$\tilde{\Gamma}\mapsto \gamma_{\tilde{\Gamma} *}(\tilde{\Gamma})$.  Note that since $\gamma_{\tilde{\Gamma} *}$ is uniquely defined, the same argument used to show that $P$ induces a homotopy equivalence shows that $G_1$ induces a homotopy equivalence on realizations.  
Similarly, define a functor $G_2:\Dcat_1\to \Dcat_2$ that it is given on objects by $\Gamma\mapsto\widehat{\Gamma'}$ where $\widehat{\Gamma'}$ is the graph obtained from $\Gamma'$ by making it essentially trivalent at the boundary.  Note that $G_2$ is the identity on objects of $\Dcat_2$ and that there is a unique morphism $\widehat{\Gamma'}\to \Gamma'$.  Thus the same argument shows that $G_2$ induces a homotopy equivalence on realizations.

Finally, we show that $\Dcat_2$, the subcategory on objects that are obtained from $\hat{\Gamma}$ by blow-ups away from the admissible cycle, has a contractible realization.  Let $v_1, v_2 \ldots v_r$ denote the vertices on the admissible cycle of $\Gamma$ and let $k_1, k_2, \ldots k_r$ denote the number of leaves that are attached at each vertex.  Consider the functor
\[\Phi:\Dcat_2\longrightarrow \prod_{j=1}^r \Tcat_{k_j+1}\]
that it is given on objects by $\tilde{\Gamma}\to(T_1,T_2, \ldots, T_r)$, where $T_j$ is the tree attached to the vertex $v_j$ of $\tilde{\Gamma}$ and the map on morphisms is defined in a natural way.  It is easy to see that $\Phi$ induces an isomorphisms of categories.  The inverse functor is give by reattaching the trees at the vertices of the admissible cycle.  Then by Remark \ref{Godin_asso}, $\Dcat_2$ is a contractible category and thus so is $\psi_i/\Gamma$
\end{description}

Then by Quillen's Theorem A, each $\psi_i$ induce a homotopy equivalence and therefore so does $p$.  The proof for $\widetilde{P}$ follows exactly the same way.

\end{proof}

We define subcategories and sub-spaces of the quasi-cell of a black and white graph $G$.

\begin{dfn}
The \emph{core of the quasi-cell of $G$} is 
\[\overline{\ie}_G:=\prod_{v\in V_{b}} (\Tcat_{\vert v\vert})\times
\prod_{v\in V_g} ( \Cat_{\vert v\vert})\times
\prod_{v\in V_{s}} ( \tCat_{\vert v\vert} )\]
The \emph{boundary of the core of the quasi-cell of $G$} is 
\[\partial\overline{\ie}_G:=\prod_{v\in V_{b}} (\partial\Tcat_{\vert v\vert})\times
\prod_{v\in V_g} ( \partial\Cat_{\vert v\vert})\times
\prod_{v\in V_{s}} ( \partial\tCat_{\vert v\vert} )\]
The \emph{boundary of the quasi-cell of $G$} is 
\[\partial\Ecat_G\cong
\prod_{v\in V_{b}}\partial\Tcat_{\vert v\vert}\times
\prod_{v\in V_g}\partial\Lcat_{\vert v\vert}\times
\prod_{v\in V_{s}}\partial\tLcat_{\vert v\vert}
\]
The \emph{thick boundary of the quasi-cell of $G$} is 
\[\eth\partial\Ecat_G\cong
\prod_{v\in V_{b}}\eth\Tcat_{\vert v\vert}\times
\prod_{v\in V_g}\eth\Lcat_{\vert v\vert}\times
\prod_{v\in V_{s}}\eth\tLcat_{\vert v\vert}
\]
The \emph{open quasi-cell of $G$} is 
\[\ie_G:=\prod_{v\in V_{b}} int(\vert\Tcat_{\vert v\vert}\vert)\times
\prod_{v\in V_g} int(\vert \Cat_{\vert v\vert}\vert)\times
\prod_{v\in V_{s}} int(\vert \tCat_{\vert v\vert} \vert)\]
\end{dfn}

\begin{cor}
\label{hom_eq_pairs}
There is a functor $P_G:\Ecat_G\to \overline{\ie}_G$ that after realization, induces a homotopy equivalence of pairs 

\[\vert P_G\vert :(\vert \Ecat_G\vert, \vert \partial\Ecat\vert)\to 
(\vert \overline{\ie}_G\vert, \partial \vert \overline{\ie}_G\vert)\]
\end{cor}
\begin{proof}
This follows immediately from Lemma \ref{thick_boundary}. The functor $P_G$ is obtained by using $P$ and $\widetilde{P}$ on the building blocks of $\Ecat_G$.
\end{proof}

\begin{rmk}
\label{boundary_ln}
Let $\Gamma_{G}$ denote the the fat graph corresponding to a black and white graph $G$. Consider $l_n$ as a black and white graph. For any $G$ in the differential of $l_n$, the graph $\Gamma_{G}$, is obtained from $l_n$ by collapsing $m$ consecutive edges in the admissible cycle for $1\leq m\leq n-1$ and then making the graph essentially trivalent.  
Similarly, consider $\tilde{l}_n$ as a black and white graph. For any $G$ in the differential of $\tilde{l}_n$, $\Gamma_{G}$ is obtained from $\tilde{l}_n$ by collapsing $m$ consecutive edges in the admissible cycle that do not contain the admissible leaf for $1\leq m\leq n-2$ and then making the graph essentially trivalent or by collapsing an edge that contains the admissible leaf. 
\end{rmk}

\begin{rmk}
We have shown that $\vert\Lcat_n\vert$ is an $n-1$ disk whose boundary is a sphere which is given by quasi-cells corresponding to the black and white graphs $G$ in the differential of $l_n$. In an analogous way than for the category $\Tcat_n$, we can interpret the graphs in the differential of $l_n$ as meaningful bracketings on $n$ variables arranged in a circle using one parenthesis.  Thus $\Lcat_n$ is a realization of the cyclohedron.  This realization is  close but not equivalent to the one given in \cite{kaufmann_schwell}.  In fact it seems to be a thickened version of that realization.
\end{rmk}

\subsubsection{The Cell-like structure on Admissible Fat Graphs}
\label{cell_like_section}
We now use the quasi-cells described in the previous subsection to give a cell like structure on $\Fatad$.

\begin{dfn}
\label{cell_functor}
Let $G$ be a black and white graph of degree $n$. We will define a functor 
\[\varphi_g:\Ecat_G\to \Fat^n\]
Let $H$ denote the set of half edges of $G$ and $V_b$ the set of black vertices.  Choose an ordering of $V_b$, and for each $v\in V_b$ choose a start half edge.   Then we can describe $H$ as $H:=\amalg_{1\leq i\leq \vert V_b\vert} H_{i}$, where $H_{i}$ is the subset of half edges attached at the $i$-th vertex.  Note that the cyclic ordering and the start half edges give a total ordering of the sets $H_{i}$.  Let $v_{l_1}, v_{l_2}, \ldots v_{l_s}$ denote the generic white vertices of $G$ ordered by their labeling and $v_{j_1}, v_{j_2}, \ldots v_{j_t}$ denote the suspended white vertices of $G$ ordered by their labeling. Cut in half all the edges of $G$ and complete each half edge $h\in H_{i}$ to a leaf labeled by the label of $h$ in the total ordering of $H_{i}$.  This gives a disjoint union of corollas on black and white vertices and $m$ chords, where the chords correspond to the leaves of $G$.  Expand the white vertices to admissible cycles. This gives a tuple of graphs 
\[\alpha_G:=  ( T_{G_1},T_{G_2},\ldots T_{G_{\vert V_b \vert}}, \Gamma_{G_{l_1}}, \Gamma_{G_{l_2}},\ldots \Gamma_{G_{l_s}}, \Gamma_{G_{j_1}}, \Gamma_{G_{j_2}},\ldots \Gamma_{G_{j_t}})\]
where $T_{G_i}$ is the corolla corresponding to $i$-th black vertex, $\Gamma_{G_{l_i}}$ is $l_{\vert v_{l_i}\vert}$,  and $\Gamma_{G_{j_i}}$ is $\tilde{l}_{\vert v_{j_i}\vert}$.  Note that $\alpha_G$ is an object of $\Ecat_G$.  Let $(i,j)$ denote the $j$-th leave of the $i$-th graph of $\alpha_G$ and let $\lbrace (i_1,j_1), (i_2,j_2),\ldots (i_m,j_m) \rbrace$ be the leaves of $\alpha_G$ that correspond to leaves in $G$.  This procedure gives an involution 
\[\iota: \bigcup_{i,j} (i,j) - \bigcup_{l=1}^m(i_l,j_l)\to \bigcup_{i,j} (i,j) - \bigcup_{l=1}^m(i_l,j_l)\]
given by the involution in $H$ which attaches its half edges and a bijection 
\[g:\lbrace1,2,\ldots m\rbrace\to \bigcup_{l=1}^m(i_l,j_l)\]
given by the labeling of the leaves of G. Let 
\[\alpha:=  ( T_{1},T_{2},\ldots T_{{\vert V_b \vert}}, \Gamma_{{l_1}}, \Gamma_{{l_2}},\ldots \Gamma_{{l_s}}, \Gamma_{{j_1}}, \Gamma_{{j_2}},\ldots \Gamma_{{j_t}})\]
be an object in $\Ecat_G$.  Then we define $\varphi_G(\alpha)$ to be the graph obtained from $\alpha$ by gluing together the leaves of $\alpha$ according to $\iota$ and then forgetting the attaching vertex so that the graph obtained has inner vertices of valence at least $3$, and then label the remaining leaves of $\alpha$ according to $g$.  Notice that $\varphi_G(\alpha)$ has mixed degree at most $n$ and that $\varphi_G(\alpha_G)$ is the admissible fat graph obtained from $G$ by expanding its white vertices as shown in \ref{iso_bw_Fat3}.  The functor is naturally defined on morphisms since morphisms in $\Ecat_G$ and $\Fat^n$ are given by collapses of inner forests that do not contain any leaves.
\end{dfn}

\begin{lem}
\label{cover}
Let $\Fatadg$  and $\Fatng$ denote the full subcategories of $\Fatad$ and $\Fat^n$ on fat graphs of topological type $\Sg$.  Since the category $\Fatadg$ is finite, then there is an $N$ such that $\FatNg=\Fatadg$.  If $n\leq N$, $\Fatng$ is covered by quasi-cells of dimension $n$ i.e.,  $\bigcup_G\vert \mathrm{Im}(\varphi_G)\vert=\vert \Fat^n\vert$ where the union runs over all isomorphism classes of black and white graphs of degree $n$ and of topological type $\Sg$.
\end{lem}
\begin{proof}

In this proof, for a fat graph $[\Gamma]$, let $G_\Gamma$ denote its corresponding black and white graph as given in Remark \ref{iso_bw_Fat3}. Let $[\Gamma]$ be an object in $\Fatng$, we will show there is a $G$ of degree $n$ such that $[\Gamma]\in Im(\varphi_G)$.  If $[\Gamma]$ is an admissible fat graph of mixed degree $n$ which is essentially trivalent at the boundary, then $[\Gamma]\in Im(\varphi_{G_\Gamma})$. If $[\Gamma]$ is an admissible fat graph of mixed degree $k<n$ which is essentially trivalent at the boundary.  Then, since $n\leq N$, by collapsing edges that do not belong to the admissible cycles and blow ups at the admissible cycles of $[\Gamma]$, we can obtain an graph $[\tilde{\Gamma}]$ which is essentially trivalent at the boundary and of degree $n$. Thus, $[\Gamma]\in Im(\varphi_{G_{\tilde{\Gamma}}})$ and furthermore,  $Im(\varphi_{G_\Gamma})\subset Im(\varphi_{G_{\tilde{\Gamma}}})$.  Finally, assume $[\Gamma]$ is not essentially trivalent at the boundary. Note that collapsing an edge on a generic admissible boundary does not change the mixed degree of the graph.  Similarly, collapsing an edge on a suspended admissible boundary that does not contain the admissible leave does not change the mixed degree of the graph.  Equivalently, blow-ups at an admissible boundary that do not separate the admissible leave do not change the mixed degree of the graph.  Therefore, we can blow up $[\Gamma]$ at the admissible boundary to an admissible fat graph $[\tilde{\Gamma}]$ of degree at most $n$ which is essentially trivalent at the boundary.  Then $[\Gamma]\in Im(\varphi_{G_{\tilde{\Gamma}}})$ and we are done on objects by the argument above.

Now we show that given a morphism $\psi_e:[\Gamma]\to [\Gamma/e]$ in $\Fat^n$, then $\psi_e\in Im(\varphi_G)$  for some black an white graph $G$ of degree $n$.  If $e$ doe not belong to an admissible cycle, then $\md([\Gamma])< \md([\Gamma/e])$.  Then by the procedure described above, we can construct a graph $[\widetilde{\Gamma/e}]$ such that $\psi_e$ is a morphism in the image of $\Ecat_{G_{\widetilde{\Gamma/e}}}$.  Similarly, if $e$ is an edge on an admissible cycle then $\md([\Gamma])\geq \md([\Gamma/e])$ and thus there is a graph $[\tilde{\Gamma}]$ such that $\psi_e$ is a morphism in the image of $\Ecat_{G_{\tilde{\Gamma}}}$.  Similarly, for a general $k$-simplex $\xi:=[\Gamma_0]\to [\Gamma_1]\ldots\to [\Gamma_k]$, we choose a vertex of $\xi$ say $[\Gamma_l]$ such that it has maximum degree in $\xi$, this is not a unique choice. Then by the procedure described above, we can construct a graph $[\widetilde{\Gamma_l}]$ such that $\xi$ is contained in the image of $\Ecat_{G_{\widetilde{\Gamma_l}}}$.  
\end{proof}

\begin{rmk}
\label{cell_boundary}
Let $G$ be a black and white graph of degree $n$ and let $\Gamma_G$ be its corresponding admissible fat graph.  By remark \ref{boundary_ln}, for any $\tilde{G}$ in the differential of $G$ its corresponding admissible fat graph $\Gamma_{\tilde{G}}$ is obtained from $\Gamma_G$ by one of the following procedures:
\begin{itemize}
\item[-] A blow-up at a vertex that does not belong to an admissible cycle
\item[-] Collapsing consecutive edges on an admissible cycle that do not contain a trivalent vertex connected to the admissible leaf, and then making the graph essentially trivalent.
\item[-] Collapsing an edge on an admissible cycle that contains a trivalent vertex connected to the admissible leaf.
\end{itemize} 
Note then that each $\Gamma_{\tilde{G}}$ is an admissible fat graph of mixed degree $n-1$ which is essentially trivalent at the boundary which is obtained from $\Gamma$ by collapses at the admissible cycles and expansions away from the admissible cycles.  Notice moreover, that any graph $\Gamma'$ of mixed degree $k<n$ that is obtained from $\Gamma$ by collapses at the admissible cycles and expansions away from the admissible cycles can be obtained in this way from some $\Gamma_{\tilde{G}}$.  Therefore, the argument of the proof of the lemma above gives 
\[\vert \varphi_G(\partial\Ecat_G) \vert =
\bigcup
\vert \varphi_{\tilde{G}}(\Ecat_{\tilde{G}})
\vert
\]
where the union is taken over all $\tilde{G}$ in the differential of $G$.
\end{rmk}

We know that $\vert \Fat^n\vert$ is covered by quasi-cells of dimension $n$, now we want to show that they sit together nicely inside this space.  Recall that $\ie_G$ is the interior of the core of the quasi-cell $\Ecat_G$.

\begin{lem}
\label{disjoint_interior}
Let $G$ and $G'$ be different isomorphism classes of black and white graphs of degree $n$.  Then the following hold
\begin{itemize}
\item[-] The restriction $\varphi_G\vert_{\ie_G}: \ie_G\to \vert Fat^n\vert$ is injective
\item[-] The image of $\ie_G$ is disjoint from the image of $\ie_{G'}$ i.e.,  $\mathrm{Im}(\varphi_G\vert_{\ie_G})\cap\mathrm{Im}(\varphi_{G'}\vert_{\ie_{G'}})=\emptyset$
\end{itemize}
\end{lem}
\begin{proof}
Note that the functor $\varphi_G:\Ecat_G\to \Fat^n$ is not necessarily injective on objects. Let $[\Gamma]$ be an object in $\Fat^n$ of mixed degree $n$ which is essentially trivalent at the boundary.  By the bijection of \ref{iso_bw_Fat3}, there is a unique black an white graph $G_\Gamma$ corresponding to $[\Gamma]$, and thus $[\Gamma]$ lies only on the image of $\Ecat_{G_\Gamma}$.  Moreover, there is a unique object of $\Ecat_{G_\Gamma}$ in the preimage of $[\Gamma]$, namely $\alpha_{G_\Gamma}$, where $\alpha_{G_\Gamma}$ is given by cutting edges of $G_\Gamma$ as given in definition \ref{cell_functor}.  Consider the map induced by $\varphi$ on the $k$-nerve of the core i.e., the map 
$N_k\varphi: N_k\overline{\ie}_G\longrightarrow N_k\Fat^n$ which sends
$\zeta:=(\alpha_0\to \ldots \to \alpha_k)\mapsto \xi:=([\Gamma_0]\to\ldots \to [\Gamma_k])$. If the simplex $\xi$ intersects the image of $\ie_G$, then there is an $l\leq k$ such that $[\Gamma_l]$ is essentially trivalent at the boundary and $\md([\Gamma_k])=n$.  This implies that $\alpha_l$ is in the interior of the core, and since the interior of the core is a disk, there is a unique simplex defined by $(\alpha_l\to \alpha_{l+1}\to\ldots \to \alpha_k)$ which maps to the simplex $([\Gamma_l]\to [\Gamma_{l+1}]\to\ldots \to [\Gamma_k])$. Moreover, the image of the simplex defined by $\alpha_0\to \alpha_1 \ldots \to \alpha_{l-1}$ does not intersect the image of $\ie_G$.  Therefore the map $\varphi_G\vert_{\ie_G}$ is injective. The image of $\ie_G$ is disjoint from the image of $\ie_{G'}$ for any $G'$ different than $G$ by the same argument.
\end{proof}
\begin{rmk}
\label{nice_intersection}
The functor $\varphi_G:\Ecat_G\to \Fat^n$  is not necessarily injective on objects.  If $\varphi_G$ is not injective on objects of mixed degree $k\leq n-1$, then $\vert \varphi_G(\Ecat_G)\vert$ is not injective at the boundary of the quasi-cell.  On the other hand if $\vert \mathrm{Im}(\varphi_G) \vert$ is not injective on the interior, then it must be so already at the boundary of the core i.e., there must be $\alpha_1$, $\alpha_2$ in $\Ecat_G$ such that $\varphi_G(\alpha_1)=\varphi_G(\alpha_2)$ and $\md(\varphi_G(\alpha_1))=n$. If this happens, then $\alpha_1$ and $\alpha_2$ are in a way symmetric, in the sense that they only differ from each other on the numbering of their leaves, since the same graph is obtained from both configurations by attaching their leaves through the functor $\varphi$.  Therefore, for each morphism in the thick boundary $\psi_{i_1}:\alpha_{i_1}\to\alpha_1$ in $\eth\Ecat_G$ there is exactly one morphism $\psi_{i_2}:\alpha_{i_2}\to\alpha_2$ in $\eth\Ecat_G$ such that $\varphi_G(\psi_{i_1})=\varphi_G(\psi_{i_2})$.  That is if $\vert \mathrm{Im}(\varphi_G) \vert$ is not injective on the interior, then it self intersects at vertices of the boundary of the core and simplicially on all simplices on the thick boundary containing such vertices.  The same argument show that if $\vert \mathrm{Im}(\varphi_G) \vert$  and $\vert \mathrm{Im}(\varphi_G') \vert$ intersect on their interior, then they intersect at vertices of the boundary of their cores and simplicially on all simplices on the thick boundary containing such vertices. 
\end{rmk}

The following theorem is originally proved (rationally) by Costello in \cite{costellorg,costellotcft} by very different methods.
\begin{thm}
The chain complex of black and white graphs is a model for the classifying spaces of mapping class groups of open-closed cobordisms.  More specifically there is an isomorphism
\[\mathrm{H}_*(\bwgraphs)\cong\mathrm{H}_*\left( \coprod_{\Sg}\mathrm{B}\Modgpq\right) \]
where the disjoint union runs over all topological types of open-closed cobordisms in which each connected component has at least one boundary component which is neither free nor outgoing closed. 
\end{thm}
\begin{proof}
It is enough to show that $\bwgraphs$ is a chain complex of $\vert \Fatad\vert$ since by \ref{ad_oc}, $\Fatad$ is a model for the classifying space of the mapping class group.  

We define a chain complex $\Cquasi_*$  using the filtration on $\Fatad$ given by the mixed degree of the graphs i.e., we define $\Cquasi_n:=H_n(\vert \Fat^n\vert, \vert \Fat^{n-1}\vert)$.   Since the quasi-cells of dimension $n$ cover $\Fat^n$ and their boundaries cover  $\Fat^{n-1}$ we have that 
\[H_*(\vert \Fat^n\vert,\vert \Fat^{n-1}\vert) =
H_*\left( \bigcup_G \vert \varphi_G(\Ecat_G)\vert,\bigcup_G \vert \varphi_G(\partial\Ecat_G)\vert\right)\]
Using Corollary \ref{hom_eq_pairs} we get a functor $\Pi_n:\amalg_G \Ecat_G\to \amalg_G\overline{\ie}_G$ that induces a homotopy equivalence of pairs

\[\vert \Pi_n\vert :(\amalg_G\vert \Ecat_G\vert, \amalg_G\vert \partial\Ecat\vert)\longrightarrow 
(\amalg_G\vert \overline{\ie}_G\vert, \amalg_G  \vert \partial\overline{\ie}_G\vert)\]

Recall that $\Pi_n$ is the identity on objects of the core.  Then, since the images of the quasi-cells intersect nicely on the thick boundary as mentioned in Remark \ref{nice_intersection}, the map $\vert \Pi_n\vert$ descends to a map 
\[\vert \pi_n\vert :(\bigcup_G\vert \varphi(\Ecat_G)\vert, \bigcup_G\vert \varphi(\partial\Ecat\vert))\longrightarrow
(\bigcup_G\vert \varphi(\overline{\ie}_G)\vert, \bigcup_G  \vert \varphi(\partial\overline{\ie}_G)\vert)\]
which is a homotopy equivalence of pairs. Since these are a CW pairs we have that

\[\tilde{H}_*(\bigcup_G\vert \varphi(\overline{\ie}_G)\vert, \bigcup_G  \vert \varphi(\partial\overline{\ie}_G)\vert) \cong
\tilde{H}_*\left( \frac{\bigcup_G\vert \varphi(\overline{\ie}_G)\vert}{
\bigcup_G  \vert \varphi(\partial\overline{\ie}_G)\vert}\right) 
\]

Recall that the interior of the associahedron and the cores are disks as stated in \ref{Godin_asso} and \ref{core_simplex}.  Therefore, the interior of the core of a quasi-cell $\ie_G$ is a open disk of dimension $n$ where $n$ is the degree of $G$ as a black and white graph.  Moreover, the image of the interiors of the cores of the quasi-cells are non-intersecting in $\Fat^n$ as described in Lemma \ref{disjoint_interior}.  Therefore,

\[\tilde{H}_*(\vert \Fat^n\vert,\vert \Fat^{n-1}\vert)\cong
\tilde{H}_*\left( \frac{\bigcup_G\vert \varphi(\overline{\ie}_G)\vert}{
\bigcup_G  \vert \varphi(\partial\overline{\ie}_G)\vert}\right) \cong
\tilde{H}_* (\vee_G S^n)\]

Thus, $\Cquasi_n$ is the free group generated by black and white graphs of degree $n$. The differential $\dquasi_n: H_n(\vert \Fat^n\vert, \vert \Fat^{n-1}\vert)\to H_{n-1}(\vert \Fat^{n-1}\vert, \vert \Fat^{n-2}\vert)$, is given by the connecting homomorphism of the long exact sequence of the triple $(\vert \Fat^n\vert, \vert \Fat^{n-1}\vert, \vert \Fat^{n-2}\vert)$.  We can show, see for example \cite{Godinunstable}, that a choice of orientation of a black and white graph corresponds to a compatible choice of orientations of the simplices that correspond to its quasi-cell.  Thus the differential takes a generator given by an $n$ dimensional quasi-cell, to its boundary in $\Fat^{n-1}$ and by \ref{cell_boundary} the boundary of a quasi-cell is given by the union of the quasi-cells corresponding to the differential of $G$.  So the chain complex $\Cquasi_*$ is the chain complex of black and white graphs $\bwgraphs$.  

On the other hand, the same argument that shows that cellular homology is isomorphic to singular homology, gives that $H_n(\Cquasi_*)\cong H_n(\vert \Fatad \vert)$ (cf. \cite[4.13]{spectral_guide}).  We give a brief sketch of this argument.  Consider the spectral sequence arising from the filtration of $\Fatad$.  The first page is given by $E^1_{p,q}=H_{p+q}(\vert \Fat^p\vert, \vert \Fat^{p-1}\vert)$.  Since the quotients in the filtration are wedges of spheres we have that
\[H_{p+q}(\vert \Fat^p\vert,\vert \Fat^{p-1}\vert)=\left\lbrace 
\begin{array}{lr}
\Cquasi_p & q=0\\
0 & q\neq 0
\end{array} \right. 
\]
Moreover the $d^1$ differential is given by the $\dquasi$ and thus by definition 
\[E^2_{p,q}=\left\lbrace 
\begin{array}{lr}
H_p(\Cquasi_*) & q=0\\
0 & q\neq 0
\end{array} \right. 
\]
Since all the terms of $E^2$ are concentrated on the row $q=0$ all higher differentials are trivial and $E^2_{p,q}=E^\infty_{p,q}$.  Finally, for this spectral sequence $E^\infty_{p,q}\cong H_p(\vert \Fatad\vert)$.  The easiest way to show that is by considering the argument in each connected component where $\Fatadg$ is a finite complex and thus the filtration is finite.
\end{proof}

\subsection{Gluing black and white Graphs}
The chain complex of black and white graphs was originally built from degenerate surfaces, and thus the gluing of surfaces along closed boundary components is not natural in this context.  The previous section gives a new point of view of black and white graphs, relating them directly to admissible fat graphs.  Furthermore, in Section \ref{gluing_ad_section} we give a topological map that models the gluing of surfaces in terms on admissible fat graphs.  We use both of these results to show that $\mathscr{BW}$-graphs is a model for the (positive-boundary) cobordism category.

The following chain map is defined in \cite{wahlwesterland}.
\begin{dfn}
Let $\mathscr{BW}_S$ denote the chain complex of black and white graphs of topological type $S$.  Let $S_1$ and $S_2$ be composable cobordisms.  We define a chain map
\[\circ_{BW}: \mathscr{BW}_{S_2} \otimes \mathscr{BW}_{S_1}\longrightarrow \mathscr{BW}_{S_2 \circ S_1}\]
Let $G_2\otimes G_1\in \mathscr{BW}_{S_2} \otimes \mathscr{BW}_{S_1}$ then 
\[G_2\circ_{BW} G_1:=\circ_{BW}(G_2\otimes G_1):=\sum \lfloor G\rfloor\]
where $\lfloor G\rfloor$ is the underlying black and white graphs of $G$ as defined in Definition \ref{underlying_BW_graph} and the sum runs over all graphs $G$ that can be obtained by the following procedure:
\begin{description}
\item[Closed gluing] \hspace*{2mm}
\begin{itemize}
\item Removing the white vertices of $G_1$ say $v_1,\ldots,v_{q_1}$.
\item Identifying the edge containing the start half edge of $v_i$ to the edge connected to $x_i$ the $i$-th closed leave of $G_2$. 
\item attaching the remaining edges incident at $v_i$ to the boundary cycle of $x_i$ respecting the cyclic ordering
\end{itemize}
\item[Open gluing] Attaching the $i$-th outgoing open leaf of $G_1$ to the $i$-th incoming open leave of $G_2$ to form a new edge for all $i=1,\ldots,q_2$.
\end{description}
\end{dfn}

\begin{rmk}
The open gluing was previously defined in \cite{costellotcft} where it was also shown that this gluing corresponds to gluing of surfaces along open boundary components in moduli space.  
\end{rmk}

\begin{thm}
Let $S_1$ and $S_2$ be composable cobordisms.  If the composite $S_2\circ S_1$ is an oriented cobordism in which each connected component has a boundary component which is neither free nor outgoing closed, then the chain map $\circ_{BW}$ is a chain model of the topological map 
\[\mathrm{BMod}(S_2) \times \mathrm{BMod}(S_1) \longrightarrow \mathrm{BMod}(S_2 \circ S_1)\]
This composition is associative.  Therefore $\mathscr{BW}$-graphs is a model for the $2$d-cobordism category.
\end{thm}

\begin{proof}
In \cite{wahlwesterland} it is shown that this is a chain map and that when interpreted as composition of graphs it is associative.  On the other hand, in Construction \ref{gluing_ad} we define a map 
\begin{equation*}
\circ_{\Fat}:|\Fatad_{S_2}| \times |\Fatad_{S_1}| \longrightarrow |\Fatad_{S_2 \circ S_1}|
\end{equation*}
which models the map on $\text{BMod}(S)$.  In Section \ref{cell_like_section} we give a cell-like structure on $|\Fatad|$ where the quasi-cells are indexed by black and white graphs.  Therefore, an element
$G_2\otimes G_1\in \mathscr{BW}_{S_2} \otimes \mathscr{BW}_{S_1}$ represents a product of quasi cells 
\[\varphi_{G_2}(\Ecat_{G_2})\times\varphi_{G_1}(\Ecat_{G_1})\subset |\Fatad_{S_2}| \times |\Fatad_{S_1}|.\] 
Consider the restriction of $\circ_{\Fat}$ to 
\[\circ_{G_2,G_1}:
\varphi_{G_2}(\Ecat_{G_2})\times\varphi_{G_1}(\Ecat_{G_1})
\longrightarrow 
|\Fatad_{S_2 \circ S_1}|\]
It is enough to show that 
\[\text{Im}(\circ_{G_2,G_1})=\bigcup_{\lfloor G\rfloor} \varphi_{\lfloor G\rfloor}(\Ecat_{\lfloor G\rfloor})
\]
where the union is taken over all ${\lfloor G\rfloor}$ such that $G_2\circ_{BW} G_1=\sum{\lfloor G\rfloor}$.  This holds by inspection. To see this recall that a quasi-cell in $\Fatad$ corresponding to a black and white graph $G$ is the subcategory on objects obtained from $\Gamma_G$ (the admissible graph corresponding to $G$ which is essentially trivalent at the boundary) by collapses on the admissible boundaries and expansions away from the admissible boundaries.  Thus in terms of metric graphs, the quasi-cell of $G$ is the subspace of graphs obtained by the above procedure, with all their possible metrics.  When we compose point-wise, each pair of metrics gives a priori different admissible fat graphs.  And the image is the union over all possible admissible fat graphs obtained in this way together with all their possible metrics.
\end{proof}

\bibliographystyle{amsalpha}
\bibliography{stringtopology}
\end{document}